\documentclass{amsart}
\usepackage{amssymb,amsfonts,amsmath,graphicx}
\usepackage[alphabetic,y2k,lite]{amsrefs}
\usepackage{fullpage, mystyle}
\usepackage{MnSymbol}
\usepackage{tikz}
\usetikzlibrary{decorations.markings}
\usetikzlibrary{shapes.geometric, arrows}
\usetikzlibrary{shapes.multipart} 
\usetikzlibrary{calc}
\usetikzlibrary{scopes}
\usetikzlibrary{math}
\usetikzlibrary{decorations.markings,decorations.pathreplacing}
\usetikzlibrary{fadings}
\usepackage{tikz-cd}

\usepackage{cancel}

\tikzset{
every picture/.style={line width=0.8pt, >=stealth,
                       baseline=-3pt,label distance=-3pt},
emptynode/.style={circle,minimum size=0pt, inner sep=0pt, outer
sep=0},
dotnode/.style={fill=black,circle,minimum size=2.5pt, inner sep=1pt, outer
sep=0},
small_dotnode/.style={fill=black,circle,minimum size=2pt, inner sep=0pt, outer
sep=0},
morphism/.style={fill=white,circle,draw,thin, inner sep=1pt, minimum size=15pt,
                 scale=0.8},
small_morphism/.style={fill=white,circle,draw,thin,inner sep=1pt,
                       minimum size=5pt, scale=0.8},
ellipse_morphism/.style args={#1}{fill=white,circle,draw,thin,inner sep=1pt,
                       minimum size=5pt, scale=0.8,
												ellipse, draw, rotate=#1},
coupon/.style={draw,thin, inner sep=1pt, minimum size=18pt,scale=0.8},
semi_morphism/.style args={#1,#2}{
                  fill=white,semicircle,draw,thin, inner sep=1pt, scale=0.8,
                  shape border rotate=#1,
                  label={#1-90:#2}},
regular/.style={densely dashed}, 
edge/.style={very thick, draw=green, text=black},
overline/.style={preaction={draw,line width=2mm,white,-}},
thin_overline/.style={preaction={draw,line width=#1 mm,white,-}},
thin_overline/.default=2,
thick_overline/.style={preaction={draw,line width=3mm,white,-}},
really_thick/.style={line width=3mm, gray},
boundary/.style={thick,  draw=blue, text=black},
ribbon/.style={line width=1.5mm, postaction={draw,line width=1mm,white}},
ribbon_u/.style args={#1,#2}{line width=#1mm, postaction={draw,line width=#2mm,white}},
cell/.style={fill=black!10},
subgraph/.style={fill=black!30},
midarrow/.style={postaction={decorate},
                 decoration={
                    markings,
                    mark=at position #1 with {\arrow{>}},
                 }},
midarrow/.default=0.5,
midarrow_rev/.style={postaction={decorate},
                 decoration={
                    markings,
                    mark=at position #1 with {\arrow{<}},
                 }},
midarrow_rev/.default=0.5,
block/.style={rectangle, rounded corners, text centered, draw=black, align=center}
}

\tikzstyle{block} = [rectangle, rounded corners, text centered, draw=black, align=center]

\tikzfading[name=fade inside, inner color=transparent!80, outer
color=transparent!10]

\newcommand{\YY}{{\mathbf{Y}}}
\newcommand{\ii}{{\mathbf{i}}}
\newcommand{\coYY}{\reflectbox{\rotatebox[origin=c]{180}{$\YY$}}}
\newcommand{\coii}{\reflectbox{\rotatebox[origin=c]{180}{$\ii$}}}
\newcommand{\YYbar}{\ov{\YY}}
\newcommand{\iibar}{\ov{\ii}}
\newcommand{\coYYbar}{\ov{\coYY}}
\newcommand{\coiibar}{\ov{\coii}}

\newcommand{\cD}{{\mathfrak{D}}}
\newcommand{\Dk}[1]{{\mathcal{D}^{#1}}}
\newcommand{\hDk}[1]{{{h{\mathcal{D}}}^{#1}}}
\newcommand{\delhDk}[1]{{{h{\del \mathcal{D}}}^{#1}}}

\newcommand{\wdtld}[1]{\widetilde{#1}}

\newcommand{\ZRT}{Z_\text{WRT}}    
\newcommand{\ZCY}{{Z_\text{CY}}}
\newcommand{\Skein}{{\text{Skein}}}
\newcommand{\VV}{\mathbf{V}}

\newcommand{\cH}{\mathcal{H}}
\newcommand{\emptyskein}[1]{{\emptyset_{#1}^{\text{sk}}}}

\DeclareMathOperator{\trace}{tr}

\begin{document}

\title{The Y-Product}
\author{Alice Kwon and Ying Hong Tham}
\maketitle

\begin{abstract}
We present a topological construction that
provides many examples of non-commutative Frobenius algebras
that generalizes the well-known pair-of-pants.
When applied to the solid torus,
in conjunction with Crane-Yetter theory,
we provide a topological proof of the Verlinde formula.
We also apply the construction to a solid handlebody
of higher genus,
leading to a generalization of the Verlinde formula
(not the higher genus Verlinde formula);
in particular, we define a generalized $S$-matrix.
Finally, we discuss the relation between our construction
and Yetter's construction of a handle as a Hopf algebra,
and give a generalization.
\end{abstract}

\section{Basic Constructions}
\label{s:basic-construction}

Convention:
Manifolds may be unorientable, in which case the
opposite orientation $\ov{M}$ will just mean $M$ itself.
We will work with smooth manifolds possibly with boundary
and corners;
the boundary will implicitly come with a collar neighborhood structure.
\\

Let $\YY$ be the cone over the discrete space of three points,
or equivalently,
three closed intervals with one endpoint from each identified:
\begin{equation}
\YY =
\begin{tikzpicture}
\node[emptynode] (o) at (0,0) {};
\node[emptynode] (a) at (0,-0.5) {};
\node[emptynode] (b) at (0.43,0.25) {};
\node[emptynode] (c) at (-0.43,0.25) {};
\draw (a) -- (o) -- (b);
\draw (o) -- (c);
\end{tikzpicture}
= [0,1] \sqcup [0,1] \sqcup [0,1] / 0 \sim 0 \sim 0
\end{equation}

\begin{definition}
\label{d:Yprod}
Let $M$ be a manifold with boundary $\del M = N$.
Let $Q$ be the double of $M$, i.e. $Q = M \cup_N \ov{M}$.

The \emph{Y-product on $Q$},
denoted $\YY_Q$,
is the manifold constructed as follows.
Thicken $\YY$ into a surface with corners,
denoted $\wdtld{\YY}$, as follows:
\begin{equation}
\label{e:Y-thick}
\wdtld{\YY} :=
\begin{tikzpicture}
\node[emptynode] (o) at (0,0) {};
\node[emptynode] (a) at (0,-0.5) {}; 
\node[emptynode] (b) at (0.43,0.25) {}; 
\node[emptynode] (c) at (-0.43,0.25) {}; 
\draw[gray] (a) -- (o) -- (b); 
\draw[gray] (o) -- (c);
\node[emptynode] (a1) at (-0.05,-0.5) {}; 
\node[emptynode] (a2) at (0.05,-0.5) {};
\node[emptynode] (b1) at (0.4,0.3) {}; 
\node[emptynode] (b2) at (0.46,0.2) {};
\node[emptynode] (c1) at (-0.46,0.2) {}; 
\node[emptynode] (c2) at (-0.4,0.3) {};
\draw (a2) to[out=90,in=-150] (b2) -- (b1);
\draw (b1) to[out=-150,in=-30] (c2) -- (c1);
\draw (c1) to[out=-30,in=90] (a1) -- (a2);
\node at (-0.3,-0.2) {\footnotesize $I_1$};
\node at (0.35,-0.2) {\footnotesize $I_2$};
\node at (0, 0.35) {\footnotesize $I_3$};
\end{tikzpicture}
\;\;
;
\;\;
\wdtld{\mathbf{i}} :=
\begin{tikzpicture}
\draw[gray] (0,0) -- (0,-0.5);
\draw (-0.1,-0.5) -- (0.1,-0.5);
\draw (-0.1,-0.5) -- (-0.1,0)
	.. controls +(90:0.13cm) and +(90:0.13cm) ..
	(0.1,0) -- (0.1,-0.5);
\end{tikzpicture}
\end{equation}
Then glue three copies of $M \times [-1,1]$
to $N \times \wdtld{\YY}$
by identifying $N \times [-1,1] \subset M \times [-1,1]$
with $N \times I_1, N \times I_2, N \times I_3$,
respectively.
The resulting manifold with boundary (no corners)
is the Y-product on $Q$.

It is clear that the six copies of $M$
at the ends of the three copies of $M \times [-1,1]$
glue up to form three disjoint copies of $Q$.
(See \figref{f:Yprod-decomp} in
\secref{s:half-handle-decomp-Yprod}.)
\hfill $\triangle$
\end{definition}

Clearly, the usual pair of pants is obtained from $M = [0,1]$,
$N = \del M$, $Q = S^1$.

This is the basic construction of the Y-product;
in our main application,
we will consider variants where $N$ is not the entire boundary
of $M$, so that $Q$ is a manifold with boundary.
One may also consider embedded manifolds, say a knot $K$
in $S^3$, and take $Q = K \# \ov{K}$.

An intuitive description of $\YY_Q$ is to hold
$Q$ above still water, so that it and its reflection
is $Q \sqcup Q$.
Slowly lower $Q$ into the water,
until the surface of the water cuts $Q$ at exactly $N$,
so that we see just one $Q$.
This intuitive picture is the key idea behind the
half-handle decomposition of $\YY_Q$
as described in \secref{s:half-handle-decomp-Yprod}
- each time the water passes a critical point,
a corresponding handle is added to build $\YY_Q$.

A Y-product is naturally a cobordism
$\YY_Q : Q \sqcup Q \to Q$.
By taking the dual cobordism, we get a cobordism
$\coYY_Q : Q \to Q \sqcup Q$,
which we call the \emph{Y-coproduct}.

They naturally come with (co)units:
simply take $M \times [-1,1]$
where we take
$M \times \{-1\} \cup N \times [-1,1] \cup M \times \{1\}$
as the incoming or outgoing boundary.
Alternatively, in fitting with the Y-product construction,
we attach $M \times [-1,1]$ to $N \times \wdtld{\mathbf{i}}$
($\wdtld{\ii}$ is defined in \eqnref{e:Y-thick})
along $N \times [-1,1]$.
It is easy to see that attaching $\wdtld{\ii}$
to one of the upper arms of $\wdtld{\YY}$
results in the identity cobordism.
(For more details,
see discussions in
\secref{s:half-handle-decomp-Yprod},
before and in the examples.)
We denote the unit and counit by
$\ii_Q : \emptyset \to Q$ and $\coii_Q: Q \to \emptyset$,
respectively.

More generally, given a fat graph $\Gamma$
(a graph with a thickening to a surface like $\wdtld{\YY}$),
we can associate a construction that attaches
$M \times I$'s to $N \times \Gamma$.
Moreover, under certain modifications,
e.g. resolving a 4-valent vertex into two 3-valent vertices,
the resulting manifold will be unchanged;
for example, the usual diagram for depicting
the Frobenius algebra relations (shown below)
would automatically make $Z(Q)$ a Frobenius algebra,
for any appropriate TQFT $Z$
(here the graph thickenings are implicitly given by
blackboard framing):
\begin{equation}
\label{e:frobenius-topology}
\begin{tikzpicture}
\draw (0,0) -- (-0.3,-0.5);
\draw (0,0) -- (-0.3,0.5);
\draw (0,0) -- (0.3,-0.5);
\draw (0,0) -- (0.3,0.5);
\end{tikzpicture}
\;\; = \;\;
\begin{tikzpicture}
\draw (0,0.2) -- (-0.3,0.5);
\draw (0,0.2) -- (0.3,0.5);
\draw (0,0.2) -- (0,-0.2);
\draw (0,-0.2) -- (-0.3,-0.5);
\draw (0,-0.2) -- (0.3,-0.5);
\end{tikzpicture}
\;\; = \;\;
\begin{tikzpicture}
\draw (-0.3,0.5) -- (-0.3,-0.5);
\draw (0.3,0.5) -- (0.3,-0.5);
\draw (-0.3,-0.2) -- (0.3,0.2);
\end{tikzpicture}
\;\; = \;\;
\begin{tikzpicture}
\draw (-0.3,0.5) -- (-0.3,-0.5);
\draw (0.3,0.5) -- (0.3,-0.5);
\draw (-0.3,0.2) -- (0.3,-0.2);
\end{tikzpicture}
\end{equation}

See \xmpref{x:torus-frobenius-1}, \ref{x:torus-frobenius-2}.

\section{Handle Decompositions for Manifolds with Corners/
Relative Cobordisms}

In order to get a handle on Y-products
for manifolds $Q$ with boundary,
we need a theory of handle decompositions
on manifolds with corners,
or more precisely, relative cobordisms.



It seems that such theory of handle decompositions/
Morse theory is still not very well-known
(at least, they were not known to the authors
at the onset of this project).
Thus, while there are already several works
that lay out such a theory in full
(see e.g. \cite{borodzik2016morse},
\cite{laudenbach2011morse}),
we give a lightning tour through some of the theory,
presenting only constructions and propositions
(mostly without proof) that are relevant to our application.
The only things new are some terminology.

Let us briefly recall the theory of handle decompositions;
we will expand on this theory to include
handle decompositions for manifolds with corners/relative cobordisms.


Let $\Dk{k}$ be the standard $k$-dimensional disk.

\begin{definition}
\label{d:handle}
An \emph{abstract $n$-dimensional $k$-handle},
is $\cH_k := \Dk{k} \times \Dk{n-k}$,
with the following distinguished submanifolds:
\begin{itemize}
\item $\del \Dk{k} \times \Dk{n-k}$: the \emph{attaching region};
\item $\del \Dk{k} \times \{0\}$: the \emph{attaching sphere};
\item $\Dk{k} \times \{0\}$: the \emph{core};
\item $\Dk{k} \times \del \Dk{n-k}$: the \emph{belt region};
\item $\{0\} \times \del \Dk{n-k}$: the \emph{belt sphere};
\item $\{0\} \times \Dk{n-k}$: the \emph{co-core}.
\end{itemize}
\hfill $\triangle$
\end{definition}

\begin{definition}
\label{d:handle-attaching}
Let $M$ be an $n$-manifold with (possibly empty) boundary $\del M$.
Let $\vphi: \del \Dk{k} \times \Dk{n-k} \to \del M$
be an embedding of the attaching region of an abstract $k$-handle
into the boundary of $M$.
The manifold $M' = M \sqcup \Dk{k} \times \Dk{n-k} / \vphi$
is said to be
\emph{obtained from $M$ by attaching a $k$-handle}.
(We smooth corners at the boundary of the attaching region
in a canonical manner.)
We refer to the region corresponding to the abstract handle
as the ($k$-)handle.
\hfill $\triangle$
\end{definition}

\begin{definition}
An \emph{elementary cobordism of index $k$} is a cobordism
$M: \del_- M \to \del_+ M$
that is obtained from the identity cobordism
$\del_- M \times [0,1]$ by attaching a $k$-handle
to its outgoing boundary.
\hfill $\triangle$
\end{definition}

\begin{definition}
\label{d:handle-decomp}
A \emph{handle decomposition} of a cobordism
$M : \del_- M \to \del_+ M$
is a filtration of $M$
\[
\del_- M \times [0,1] =: M_0 \subset M_1 \subset \cdots
\subset M_{l-1} \subset M_l := M
\]
such that $M_{i+1}$ is obtained from $M_i$ by attaching a handle
to the outgoing boundary of $M_i$.
\hfill $\triangle$
\end{definition}

By smoothing corners at the boundary of attaching regions,
it is clear that a handle decomposition
allows us to write any cobordism as a
composition of elementary cobordisms.

\begin{definition}
\label{d:dual-decomp}
Given a handle decomposition of a cobordism
$M : \del_- M \to \del_+ M$,
the \emph{dual handle decomposition}
is the handle decomposition of
$M : \ov{\del_+ M} \to \ov{\del_- M}$
by simply treating each $k$-handle as a $(n-k)$-handle
(where $n$ is the dimension of $M$).
\hfill $\triangle$
\end{definition}

Thus, a handle decomposition on $M: \emptyset \to N$
naturally endows a handle decomposition on its double
$Q = M \cup_N \ov{M} : \emptyset \to N \to \emptyset$.

\begin{proposition}
\label{p:handle-decomp-exist}
Any cobordism admits a handle decomposition.
\end{proposition}
\begin{proof}
See \cite{milnor2016morse}.
\end{proof}

Now let us recall \emph{relative cobordisms},
which are cobordisms between manifolds with boundary
(see e.g. \cite{borodzik2016morse}).
Note that in some texts, relative cobordisms are only defined
between manifolds with boundary when the boundaries
are the same,
and the vertical wall (as defined below)
is always the identity cobordism.

\begin{definition}
\label{relative-cobordism}
A \emph{relative cobordism}
from a manifold with boundary $M_{in}$
to another manifold with boundary $M_{out}$,
denoted $(W, \del_v W):
(M_{in}, \del M_{in}) \to (M_{out}, \del M_{out})$
is a manifold $W$ with corners
$\del_-^2 W, \del_+^2 W$
which separate the boundary $\del W$
into three submanifolds $\del_- W, \del_+ W, \del_v W$,
and identifications
$M_{in} \simeq \overline{\del_- W},
M_{out} \simeq \overline{\del_+ W}$;
the three submanifolds satisfy:
\begin{itemize}
\item their union is $\del W$,
\item $\del_- W, \del_+ W$ are disjoint,
\item $\del_- W \cap \del_v W = \del_-^2 W$,
\item $\del_+ W \cap \del_v W = \del_+^2 W$.
\end{itemize}
We refer to $\del_- W, \del_+ W,$ and $\del_v W$
as the \emph{incoming boundary}, \emph{outgoing boundary},
and \emph{vertical wall} of $W$, respectively.

We call $\del_-^2 W$ and $\del_+^2 W$
the \emph{incoming corner} and \emph{outgoing corner}, respectively.
Note that $\del_v W$ is a cobordism from the former to the latter.

If $W$ is oriented, we make the following choices for orientations:
\begin{itemize}
\item $\del_+ W$ has the induced orientation,
\item $\del_- W, \del_v W$ have the opposite induced orientation,
\item $\del_+^2 W, \del_-^2 W$ have the induced orientation with respect to
	$\del_+ W, \del_- W$, respectively.
\end{itemize}

See \figref{f:half-handles}.
\hfill $\triangle$
\end{definition}

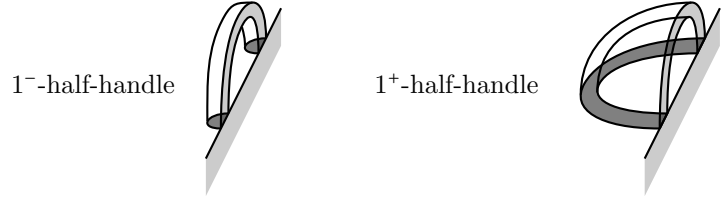
\begin{figure}[h] 
\centering
\begin{tikzpicture}
\node at (-1.5,1) {$1^-$-half-handle};
\draw[fill=gray]
	(0.2,0.4) .. controls +(180:0.3cm) and +(180:0.3cm) ..
	(0.3,0.6);
\draw[fill=gray]
	(0.8,1.6) .. controls +(180:0.3cm) and +(180:0.3cm) ..
	(0.7,1.4);
\draw[fill=black, opacity=0.2, draw=none]
	(0.2,0.4) .. controls +(90:0.5cm) and +(-130:0.4cm) ..
	(0.5,2) .. controls +(50:0.2cm) and +(90:0.5cm) ..
	(0.8,1.6) --
	(0.7,1.4) .. controls +(90:0.5cm) and +(60:0.2cm) ..
	(0.5,1.8) .. controls +(-120:0.3cm) and +(90:0.5cm) ..
	(0.3,0.6) -- (0.2,0.4);
\draw[fill=black, opacity=0.2, draw=none]
	(0,0) -- (1,2) -- (1,1.5) -- (0,-0.5) -- (0,0);
\draw (0,0) -- (1,2);
\draw (0.2,0.4) .. controls +(90:0.5cm) and +(-130:0.4cm) ..
	(0.5,2) .. controls +(50:0.2cm) and +(90:0.5cm) ..
	(0.8,1.6);
\draw (0.3,0.6) .. controls +(90:0.5cm) and +(-120:0.3cm) ..
	(0.5,1.8) .. controls +(60:0.2cm) and +(90:0.5cm) ..
	(0.7,1.4);
\draw (0.02,0.5) .. controls +(90:0.5cm) and +(-120:0.5cm) ..
	(0.25,1.9) .. controls +(60:0.2cm) and +(180:0.2cm) ..
	(0.62,2.075);
\draw (0.52,1.5) .. controls +(90:0.1cm) and +(-75:0.1cm) ..
	(0.48,1.75);
\end{tikzpicture}
\;\;\;\;\;\;\;\;\;
\begin{tikzpicture}
\node at (-2.5,1) {$1^+$-half-handle};
\draw[fill=gray]
	(0.2,0.4) .. controls +(180:0.8cm) and +(-120:0.4cm) ..
	(-0.8,1) .. controls +(60:0.4cm) and +(180:0.8cm) ..
	(0.8,1.6) --
	(0.7,1.4) .. controls +(180:0.6cm) and +(60:0.3cm) ..
	(-0.6,1) .. controls +(-120:0.3cm) and +(180:0.6cm) ..
	(0.3,0.6) -- (0.2,0.4);
\draw[fill=black, opacity=0.2, draw=none]
	(0.2,0.4) .. controls +(90:0.5cm) and +(-130:0.4cm) ..
	(0.5,2) .. controls +(50:0.2cm) and +(90:0.5cm) ..
	(0.8,1.6) --
	(0.7,1.4) .. controls +(90:0.5cm) and +(60:0.2cm) ..
	(0.5,1.8) .. controls +(-120:0.3cm) and +(90:0.5cm) ..
	(0.3,0.6) -- (0.2,0.4);
\draw[fill=black, opacity=0.2, draw=none]
	(0,0) -- (1,2) -- (1,1.5) -- (0,-0.5) -- (0,0);
\draw (0,0) -- (1,2);
\draw (0.2,0.4) .. controls +(90:0.5cm) and +(-130:0.4cm) ..
	(0.5,2) .. controls +(50:0.2cm) and +(90:0.5cm) ..
	(0.8,1.6);
\draw (0.3,0.6) .. controls +(90:0.5cm) and +(-120:0.3cm) ..
	(0.5,1.8) .. controls +(60:0.2cm) and +(90:0.5cm) ..
	(0.7,1.4);
\draw (-0.85,0.85) .. controls +(90:0.8cm) and +(180:0.7cm) ..
	(0.62,2.075);
\draw (-0.63,0.88) .. controls +(90:0.8cm) and +(180:0.5cm) ..
	(0.59,1.881);
\end{tikzpicture}
\caption{
Comparison of positive and negative type half-handles.
The light gray regions are part of the vertical boundary,
the dark gray regions are the attaching regions.
Note that the vertical boundaries are the same.
}
\label{f:half-handles}
\end{figure}

We may sometimes implicitly identify
$M_{in}$ with $\del_- W$ and $M_{out}$ with $\del_+ W$.

\begin{definition}
\label{d:half-handle}
Let $\hDk{k}$ denote the half-disk,
$\hDk{k} =
	\{\mathbf{x} = (x_1,\ldots,x_k)
	| \; |\mathbf{x}| \leq 1, x_1 \geq 0\}$.
Let $\delhDk{k}$ denote the hemisphere portion of the boundary,
$\delhDk{k} = 
	\{\mathbf{x} = (x_1,\ldots,x_k)
	| \; |\mathbf{x}| = 1, x_1 \geq 0\}$,
and let $\del_v \hDk{k}$ denote the flat portion of the boundary,
$\del_v \hDk{k} = 
	\{\mathbf{x} = (x_1,\ldots,x_k)
	| \; |\mathbf{x}| \leq 1, x_1 = 0\}$.

Let $0 \leq k \leq n-1$.
An \emph{abstract $n$-dimensional $k^-$-half-handle}
is $h\cH_k^- := \Dk{k} \times \hDk{n-k}$,
with the following distinguished submanifolds:
\begin{itemize}
\item $\del \Dk{k} \times \hDk{n-k}$:
	the \emph{attaching region};
\item $\del \Dk{k} \times \{0\}$:
	the \emph{attaching sphere};
\item $\Dk{k} \times \{0\}$: the \emph{core};
\item $\Dk{k} \times \delhDk{n-k}$:
	the \emph{belt region};
\item $\{0\} \times \delhDk{n-k}$:
	the \emph{belt disk};
\item $\{0\} \times \hDk{n-k}$: the \emph{co-core}.
\end{itemize}

An \emph{abstract $n$-dimensional $k^+$-half-handle}
is $h\cH_k^+ := \hDk{k+1} \times \Dk{n-k-1}$,
with the following distinguished submanifolds:
\begin{itemize}
\item $\delhDk{k+1} \times \Dk{n-k-1}$:
	the \emph{attaching region};
\item $\delhDk{k+1} \times \{0\}$:
	the \emph{attaching disk};
\item $\hDk{k+1} \times \{0\}$: the \emph{core};
\item $\hDk{k+1} \times \del \Dk{n-k-1}$:
	the \emph{belt region};
\item $\{0\} \times \del \Dk{n-k-1}$:
	the \emph{belt sphere};
\item $\{0\} \times \Dk{n-k-1}$: the \emph{co-core}.
\end{itemize}

We say that these are half-handles
of dimension $n$, index $k$, and type + and -, respectively.

The \emph{vertical wall}
of a $k^\pm$-half-handle
is the $(n-1)$-dimensional $k$-handle in its boundary given by:
\begin{itemize}
\item $\Dk{k} \times \del_v \hDk{n-k}$ for $k^-$ half-handles;
\item $\del_v \hDk{k+1} \times \hDk{n-k-1}$ for $k^+$ half-handles.
\end{itemize}
\hfill $\triangle$
\end{definition}

Essentially,
a $k^-$-half-handle is obtained from
a $k$-handle of the same dimension
by cutting the $\Dk{n-k}$ factor,
the ``thickening'' factor, in half;
similarly, a $k^+$-half-handle is obtained from
a $(k+1)$-handle of the same dimension
by cutting the $\Dk{k+1}$ factor,
the ``core'' factor, in half.

Positive/negative type half-handles are called
left/right half-handles, respectively,
in \cite{borodzik2016morse},
and are called
D-type/N-type in \cite{laudenbach2011morse}.

We also note that by swapping the (half-)disk factors,
and swapping the labels
``attaching'' $\leftrightarrow$ ``belt''
and ``core'' $\leftrightarrow$ ``co-core'',
a $k^\pm$-half-handle turns into an $(n-1-k)^\mp$-half-handle
(note the sign change);
in other words,
the dual of a $k^\pm$-half-handle,
i.e. when it is viewed ``upside-down'',
is a $(n-1-k)^\mp$-half-handle.

\begin{definition}
\label{d:half-handle-attaching}
Let $(W,\del_v W) : (\del_- W, \del_-^2 W)
\to (\del_+ W, \del_+^2 W)$ be a relative cobordism.
Let $\vphi$ be an embedding of the attaching region of
an abstract $k^\pm$-half-handle into $\del_+ W$
such that the restriction of $\vphi$ to the vertical wall
of the half-handle is an attaching map.
The manifold
$W' = W \sqcup (k^\pm\text{-half-handle}) / \vphi$
is said to be
\emph{obtained from $W$ by attaching a $k^\pm$-half-handle}.
We refer to the region corresponding to the abstract half-handle
as the ($k^\pm$-)half-handle.
\hfill $\triangle$
\end{definition}

Note that attaching a positive type half-handle
to a relative cobordism $W$
does not change the diffeomorphism type of $W$,
since the attaching region is a disk.
For negative type half-handles,
see \lemref{l:negative-half-handle-identity}.

\begin{definition}
\label{d:half-handle-decomp}
A \emph{half-handle decomposition} of a relative cobordism $W$
is a filtration of $W$
\[
\del_- W \times [0,1] =: W_0 \subset W_1 \subset \cdots
\subset W_{l-1} \subset W_l := W
\]
such that $W_{i+1}$ is obtained from $W_i$ by attaching a
handle or half-handle,
and the restriction of the filtration to the vertical wall $\del_v W$
is a handle decomposition
\footnote{Possibly with ``trivial steps'',
which occur when a handle (not a half-handle) is attached.}.
\hfill $\triangle$
\end{definition}

\begin{definition}
\label{d:dual-decomp-hf}
Given a handle decomposition of a relative cobordism
$(W,\del_v W) : (\del_- W, \del_-^2 W)
\to (\del_+ W, \del_+^2 W)$,
the \emph{dual handle decomposition}
is the handle decomposition of
$(W,\del_v W) : (\ov{\del_- W}, \ov{\del_-^2 W})
\to (\ov{\del_+ W}, \ov{\del_+^2 W})$,
by simply treating each $k$-handle as a $(n-k)$-handle,
and each $k^\pm$-half-handle as a $(n-1-k)^\mp$-half-handle
(where $n$ is the dimension of $M$).
\hfill $\triangle$
\end{definition}

\begin{proposition}
\label{p:half-handle-decomp-exist}
Any relative cobordism admits a half-handle decomposition.
\end{proposition}

\begin{proof}
The idea is to double the relative cobordism along the
vertical boundary,
find a $\ZZ/2$-invariant Morse function,
apply Morse theory to get a handle decomposition,
and observe that we get a half-handle decomposition
on the original relative cobordism.
See \cite{borodzik2016morse}*{Section 2}
for a complete proof.
\end{proof}

Next, we consider the dual decomposition:

\begin{proposition}
\label{p:dual-decomp-half}
Let $M$ be a relative cobordism,
and let $(M_i)_{i=0,\ldots,l}$ be a half-handle decomposition.
Let $M'$ be the relative cobordism
with the same underlying manifold $M$,
but incoming and outgoing boundaries/corners are swapped.

Then there exists a half-handle decomposition
$(M_i')_{i=0,\ldots,l}$ of $M'$
such that,
if the $i$-th (half-)handle attachment for $M$ is a
$k$-handle/$k^\pm$-half-handle,
then the $(l+1-i)$-th (half-)handle attachment for $M'$ is a 
$(n-k)$-handle/$(n-1-k)^\mp$-half-handle.
\end{proposition}

\begin{proof}
Follows from the doubling argument
and the argument for usual handle decompositions.
\end{proof}

Before we move on to applications,
we consider another variant of cobordisms,
namely \emph{cornered cobordisms},
(or \emph{cobordism with corners} in \cite{yetter1997portrait}).
These are also cobordisms between manifolds with boundary,
but they are closer to the other notion of
relative cobordism as mentioned above
(when the boundary is fixed).

Cornered cobordisms were used in \cite{tham-thesis}
to formulate the extended Crane-Yetter theory
(see \defref{d:zcy-elementary} below).
Note the use of a semicolon instead of a comma
in $(W;N)$ below.

\begin{definition}
A \emph{cornered cobordism} $(W;N): M_{in} \to M_{out}$,
is a manifold $W$ with corner $N$
which separates $\del W$ into two submanifolds
$M_-, M_+$ with boundary $N$,
together with identifications
$M_{in} \simeq \ov{M_-}$ and $M_{out} \simeq M_+$.

The composition $(W';N') \circ (W;N)$
of cornered cobordisms
$(W;N) : M_{in} \to M_{out},
(W';N') : M_{in}' = M_{out} \to M_{out}'$
consists of the manifold
$W' \cup_{\ov{M_-'} \simeq M_{out} \simeq M_+} W$,
along with the obvious identifications
for the boundaries.
\hfill $\triangle$
\end{definition}

We sometimes forget the identifications and treat
$M_-,M_+$ as submanifolds of $\del W$,
but it will be important for the $S$-matrix computation
(see \xmpref{x:solid-torus-S-matrix-topology})
to keep this identification explicit.

One may treat a relative cobordism as a cornered cobordism
by taking, say, the outgoing corner
as the corner for the cornered cobordism.
An application of this is to use
the definition of extended Crane-Yetter theory
on cornered cobordisms (\defref{d:zcy-elementary})
to define maps associated to relative cobordisms
(see \eqnref{e:zrt-zcy}, \xmpref{x:cy-Yprod-2},
\figref{f:CY-RT}).

Another application is to help formulate the following
useful fact about negative type half-handles,
which intuitively says that
attaching negative type half-handles
does not change the ``bulk'' of a relative cobordism:

\begin{lemma}
\label{l:negative-half-handle-identity}
Let $M$ be a relative cobordism obtained from
attaching only negative type half-handles to
an identity relative cobordism.
Then, treating $M$ as a cornered cobordism
with corner $\del_+^2 M$,
$M$ is an identity cornered cobordism.
\end{lemma}

\begin{proof}
This follows easily by dualizing
(i.e. taking the dual decomposition of)
the fact that attaching positive type half-handles
does not change the diffeomorphism type of the
relative cobordism
(because, as stated before, the attaching region for
positive type half-handles is a ball).
One can also prove this directly by
applying a deformation retract of
a negative half-handle onto the vertical handle within it.
\end{proof}

By considering dual handles, we see that
attaching a positive half-handle can be interpreted
as the identity cornered cobordism with $\del_-^2$
as corner.

\subsection{Half-handle Decomposition for Y-product}
\label{s:half-handle-decomp-Yprod}

Now let us consider the Y-product construction
from the perspective of (half-)handle decompositions.
Recall that the Y-product on $Q = M \cup_N \ov{M}$,
where $M$ is a manifold with boundary,
is obtained by gluing three copies of $M \times [-1,1]$
to $N \times \wdtld{\YY}$
(see \eqnref{e:Y-thick}).
Following that construction,
if we only glue on two copies of $M \times [-1,1]$,
say to $N \times I_1$ and $N \times I_2$,
the resulting manifold has two boundary components,
$M \cup_N N \times I_3 \cup_N \ov{M}$
and $M \cup_N \ov{M}$;
it is clearly the identity relative cobordism on $Q$.
Thus, most of the interesting topology happens
in $M \times I_3$,
here treated as a relative cobordism
from $\ov{M} \sqcup M$ to $N \times I$.
See \figref{f:Yprod-decomp} below.

\begin{figure}[h] 
\centering
\begin{tikzpicture}
\begin{scope}[shift={(-1.5,2)}]
\draw (-1,0) .. controls +(50:0.4cm) and +(180:0.6cm) ..
	(0.2,0.3) .. controls +(0:0.6cm) and +(50:0.4cm) ..
	(1,0) .. controls +(-130:0.4cm) and +(0:0.6cm) ..
	(-0.2,-0.3) .. controls +(180:0.6cm) and +(-130:0.4cm) ..
	(-1,0);
\draw[line width=0.2mm] (-0.3,-0.3) -- (0.1,0.3);
\draw[line width=0.2mm] (-0.1,-0.3) -- (0.3,0.3);
\end{scope}
\begin{scope}[shift={(1.5,2)}]
\draw (-1,0) .. controls +(50:0.4cm) and +(180:0.6cm) ..
	(0.2,0.3) .. controls +(0:0.6cm) and +(50:0.4cm) ..
	(1,0) .. controls +(-130:0.4cm) and +(0:0.6cm) ..
	(-0.2,-0.3) .. controls +(180:0.6cm) and +(-130:0.4cm) ..
	(-1,0);
\draw[line width=0.2mm] (-0.3,-0.3) -- (0.1,0.3);
\draw[line width=0.2mm] (-0.1,-0.3) -- (0.3,0.3);
\end{scope}
\begin{scope}[shift={(0,-2)}]
\draw (-1,0) .. controls +(50:0.4cm) and +(180:0.6cm) ..
	(0.2,0.3) .. controls +(0:0.6cm) and +(50:0.4cm) ..
	(1,0) .. controls +(-130:0.4cm) and +(0:0.6cm) ..
	(-0.2,-0.3) .. controls +(180:0.6cm) and +(-130:0.4cm) ..
	(-1,0);
\draw[line width=0.1mm] (-0.3,-0.3) -- (0.1,0.3);
\draw[line width=0.1mm] (-0.1,-0.3) -- (0.3,0.3);
\end{scope}
\begin{scope}[shift={(-1.5,0)}]
\draw[line width=0.1mm]
	(-1,0) .. controls +(50:0.4cm) and +(180:0.6cm) ..
	(0.2,0.3) --
	(3.2,0.3) .. controls +(0:0.6cm) and +(50:0.4cm) ..
	(4,0) .. controls +(-130:0.4cm) and +(0:0.6cm) ..
	(2.8,-0.3) --
	(-0.2,-0.3) .. controls +(180:0.6cm) and +(-130:0.4cm) ..
	(-1,0);
\end{scope}
\draw (-2.58,1.85) -- (-2.58,-0.15)
	.. controls +(-90:1cm) and +(90:1cm) ..
	(-1.08,-2.15);
\draw (2.58,2.15) -- (2.58,0.15)
	.. controls +(-90:1cm) and +(90:1cm) ..
	(1.08,-1.85);
\draw (-0.42,2.15) .. controls +(-90:0.5cm) and +(150:0.2cm) ..
	(0,1.2);
\draw (0.42,1.85) .. controls +(-90:0.4cm) and +(30:0.2cm) ..
	(0,1.2);
\begin{scope}[line width=0.2mm,shift={(-0.2,-0.3)}]
\draw (-1.4,2) -- (-1.4,0);
\draw (-1.6,2) -- (-1.6,0)
	.. controls +(-90:0.2cm) and +(180:0.2cm) ..
	(-1.4,-0.2) --
	(-0.1,-0.2) --
	(-0.1,-2);
\draw (1.4,2) -- (1.4,0);
\draw (1.6,2) -- (1.6,0)
	.. controls +(-90:0.2cm) and +(0:0.2cm) ..
	(1.4,-0.2) --
	(0.1,-0.2) --
	(0.1,-2);
\end{scope}
\begin{scope}[line width=0.1mm,shift={(0.2,0.3)}]
\draw (-1.4,2) -- (-1.4,0);
\draw (-1.6,2) -- (-1.6,0)
	.. controls +(-90:0.2cm) and +(180:0.2cm) ..
	(-1.4,-0.2) --
	(-0.1,-0.2) --
	(-0.1,-2);
\draw (1.4,2) -- (1.4,0);
\draw (1.6,2) -- (1.6,0)
	.. controls +(-90:0.2cm) and +(0:0.2cm) ..
	(1.4,-0.2) --
	(0.1,-0.2) --
	(0.1,-2);
\end{scope}
\node at (-3.5,0) {\small $M \times I_1$};
\node at (3.5,0) {\small $M \times I_2$};
\node at (0,2.5) {\small $M \times I_3$};
\draw[line width=0.1mm] (0,2.3) -- (-0.4,1.3);
\node at (3,-1.5) {\small $N \times \widetilde{\mathbf{Y}}$};
\draw[line width=0.1mm] (2.3,-1.3) -- (1.3,-0.4);
\end{tikzpicture}
\caption{}
\label{f:Yprod-decomp}
\end{figure}

Suppose we are given some half-handle decomposition
of $M$ as a relative cobordism from $\emptyset$ to $N$;
let the handles, in order of attachment,
be $H_1, H_2, \ldots, H_l$.
Lay out $M \times I_3$ like a ``folding fan'':
if $M$ is embedded in some half-space
$\RR_{\geq 0} \times \RR^l$ away from the boundary,
then sweeping the half-space through an extra dimension
will trace out
\[
M \times I_3 \simeq
\{(\cos \theta \cdot x, y, \sin \theta \cdot x) |
	(x,y) \in M \}
	\subset \RR \times \RR^l \times \RR
\]
Then it is clear that
$H_1 \times I_3, H_2 \times I_3, \ldots, H_l \times I_3$
give a half-handle decomposition of $M \times I_3$,
where each $H_i \times I_3$ is exactly one index higher
than $H_i$.

A half-handle decomposition for the counit
can be obtained by similar methods.
This time, we think of $M$ as a relative cobordism
from $N$ to $\emptyset$,
and properly embed $M$ in $\RR_{\geq 0} \times \RR^l$,
with $N$ in $\{0\} \times \RR^l$.
By sweeping through an extra dimension as before,
we will get the counit relative cobordism,
and each (half-)handle sweeps out a (half-)handle.

\subsection{Solid tori}
\label{s:topology-examples-torus}

We describe two examples of Y-products,
both are Y-products on the solid torus;
these will serve as the topological basis of
our main result.

Before that, let us fix some notation/conventions -
the large number of automorphisms of the solid torus
makes it necessary to fix particular choices
of curves/bases to avoid confusion;
for example, the seemingly innocuous automorphism
of ``flipping a donut'' actually implements
the ``charge conjugation'' operation on the Verlinde algebra
(see \eqnref{e:verlinde-solid-torus}).

Let $X$ be the solid torus $X = S^1 \times \Dk{2}$.
Choose $\vec{l} := S^1 \times \{*\}$,
$\vec{m} := \{*\} \times \del \Dk{2}$
to be the longitude and meridian, respectively.
orientations are chosen so that
$(\vec{n},\vec{l},\vec{m})$ is the orientation of $X$,
where $\vec{n}$ is the outward normal vector at the boundary.

Given a diagram of another solid torus,
we may represent an identification of that solid torus
with $X$ by drawing and labeling two simple closed curves
on that solid torus's boundary,
indicating which curve goes to $\vec{l}$ and $\vec{m}$.

\begin{example}
\label{x:solid-torus-S-matrix-topology}
Define $X'$ as $X$ after applying surgery along
the core circle with 0-framing
(trivial with respect to the product structure of $X$,
or equivalently,
$\vec{l}$ provides the framing);
by definition, there is a cornered cobordism
$\cH_2: X \to X'$, corresponding to the attachment
of a 2-handle to the core circle.

Now $X'$ is clearly another solid torus.
We define $\Phi: X' \simeq X$ to be the diffeomorphism
that sends $\vec{m}$ and $\vec{l}$
(as curves in $\del X' = \del X$)
to $\vec{l}$ and $-\vec{m}$ respectively.
Similarly,
we define $\ov{\Phi}: X' \simeq X$ to be the diffeomorphism
that sends $\vec{m}$ and $\vec{l}$
to $-\vec{l}$ and $\vec{m}$ respectively.

We denote the cornered cobordisms resulting from
composing $\Phi,\ov{\Phi}$ with $\cH_2$ by
$\Psi := \Phi \circ \cH_2 : X \to X$,
$\ov{\Psi} := \ov{\Phi} \circ \cH_2 : X \to X$
respectively.

Note that $\ov{\Psi}$ is the dual cobordism of $\Psi$,
and vice versa
(with appropriate choice of identification $\ov{X} \simeq X$).

If we identify the boundary $\del X$
with $\RR^2$,
such that $\vec{l} = (1\;\;0)^T$
and $\vec{m} = (0\;\;1)^T$,
then $\Psi|_{\del X}$ can be identified
with a clockwise $\pi/2$ rotation,
which has two fixed points $(0,0)$ and $(1/2,1/2)$,
and swaps $(1/2,0)$ and $(0,1/2)$.
Similarly, $\ov{\Psi}|_{\del X}$ can be identified
with an anti-clockwise $\pi/2$ rotation.
This will be useful for extending $\Psi$
to higher genus solid handlebodies
(see \secref{s:topology-examples-handlebody}).
\hfill $\triangle$
\end{example}

Now let us describe the two Y-products on $X$.

\begin{example}[Solid torus I]
\label{x:solid-torus-1}
Consider $M = \Dk{2} \times [0,1]$ as a relative cobordism
\[
(\Dk{2} \times [0,1], \del \Dk{2} \times [0,1]) :
(\emptyset, \emptyset) \to
(\Dk{2} \times \{0,1\}, \del \Dk{2} \times \{0,1\})
\]
which is half-way to building $X$;
we give it the following half-handle decomposition:
start with the $0^-$-half-handle,
followed by a $1^-$-half-handle,
and finally a $2$-handle
\footnote{Note that the last two steps can in fact
be replaced with a single $1^+$-half-handle;
we choose this half-handle decomposition
for use in computations later.}
(see \figref{f:solid-torus-half-1}).

To be more specific \footnotemark,
the $0^-$-half-handle
starts at the $(1/4,0)$ point on $\del X$.
Then the $1^-$-half-handle grows out of the $0^-$-half-handle
and closes up into a meridian.
Finally, the $2$-handle fills in the meridian.

\begin{figure}[h] 
\centering
\input{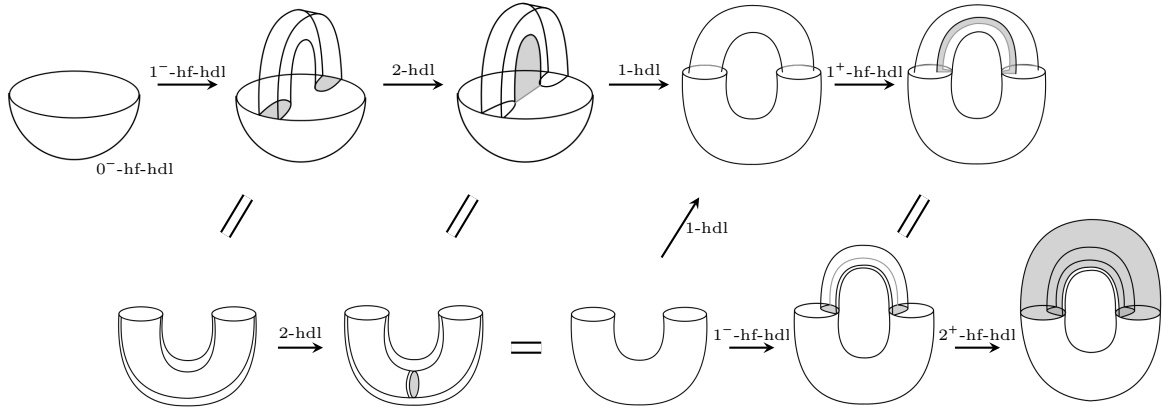}
\caption{Half-handle decomposition for $X$,
for $\YY_X^{(1)}$}
\label{f:solid-torus-half-1}
\end{figure}

This should define a Y-product on $X$,
which we denote $\YY_X^{(1)}$,
with the corresponding half-handle decomposition
on the Y-product on the solid torus $X$
as shown in \figref{f:solid-torus-Yprod-1} below.

\begin{figure}[h] 
\centering
\input{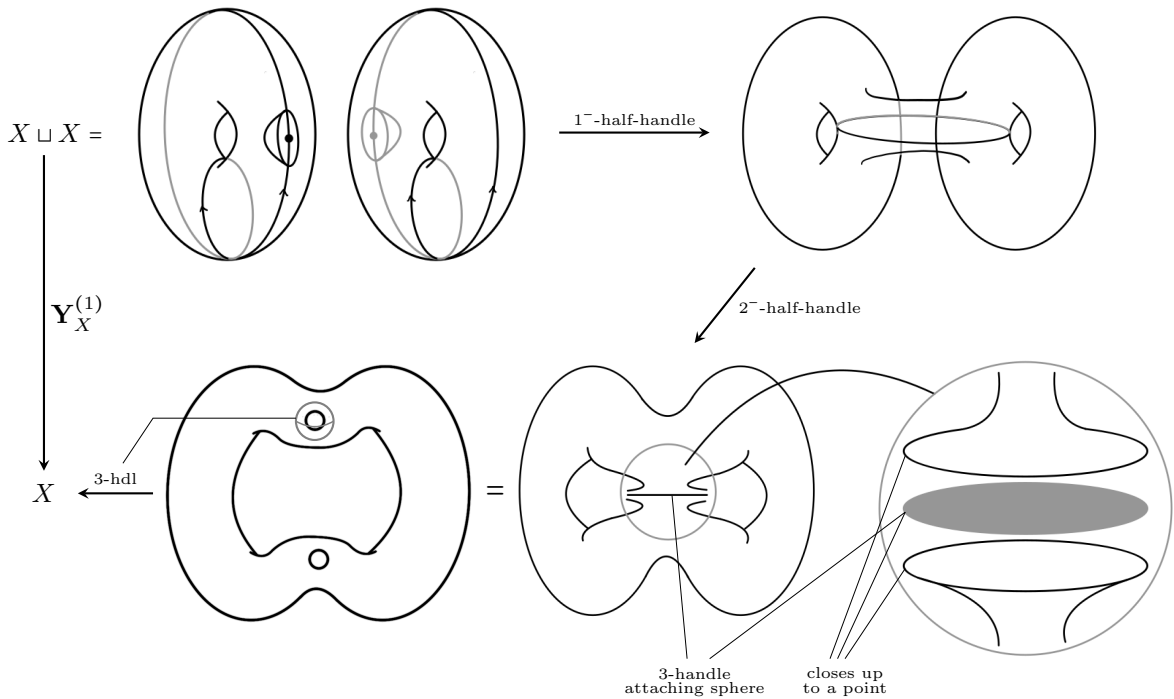}
\caption{
$\YY_X^{(1)}$ half-handle decomposition.
Note that the overall effect of the last two steps
($2^-$-half-handle and $3$-handle)
is to remove a solid cylinder,
or more precisely,
move the solid cylinder into the vertical boundary;
this is consistent with the idea that
the $2^-$- and $3$-(half)-handles
can be combined into one $2^+$-half-handle.
}
\label{f:solid-torus-Yprod-1}
\end{figure}


Here's a more wordy description of the half-handle decomposition
of $\YY_X^{(1)}$,
in terms of a surgery picture,
i.e. we keep track of just what happens to the 3-manifold
as we add half-handles.
Start with $X \sqcup X$.
The $1^-$-half-handle
cuts out two half 3-disks $\hDk{3}$ from the boundaries of $X$'s
(to be precise \footnotemark[\value{footnote}],
we take out a $\hDk{3}$ around
the $(1/4,0)$ point from the left $X$, and
a $\hDk{3}$ around the $(3/4,0)$ point from the right $X$),
and attaches a solid cylinder.
The boundary of this new 3-manifold $M_1$
is a genus two surface.

\footnotetext{This may seem overly pedantic,
as the boundary of $X$ is so symmetric
as to have no special points,
but these choices matter in the interaction between
$\Psi/\ov{\Psi}$ and the Y-products;
see \prpref{p:Psi-Yprod-equiv}.
It becomes particularly important when discussing
higher genus handlebodies in
\secref{s:topology-examples-handlebody},
as the Y-products are no longer symmetric.
}

Next,
the attaching sphere of the $2^-$-half-handle
is a circle on the boundary of $M_1$ (the genus two surface)
that travels along the meridians of the $X$'s
and goes back and forth along the solid cylinder just attached.
A neighborhood of the attaching sphere is removed,
and replaced with a thickened 2-disk,
obtaining $M_2$.

In terms of the boundary, $\del M_2$ is now already a torus.
Indeed, just focusing on what happens to the boundary surfaces,
we see that we have performed a Y-product on the torus.
However, there is one more handle, a 3-handle corresponding
to the 2-handle of $M$.
What is the attaching sphere of this 3-handle?
Observe that the attaching sphere of the previous
$2^-$-half-handle also bounds a 2-disk in $M_1$
(it goes through the solid cylinder).
The attaching sphere of the 3-handle is the union of
this 2-disk and the core of the $2^-$-half-handle.

Another way to view this Y-product is that it is simply
the usual pair of pants times $\Dk{2}$.
This is apparent from the construction
as described in \secref{s:basic-construction}.
In particular, we see that the vertical boundary
is the usual pair of pants times $S^1$.

A half-handle decomposition for the
counit for the corresponding Y-coproduct is given
in \figref{f:solid-torus-counit-1}.
\hfill $\triangle$
\end{example}

\begin{figure}[h] 
\centering
\begin{tikzpicture}
\node at (-3.5,0) {$X$};
\draw[->, line width=0.3mm] (-3.0,0) -- (-1.7,0);
\node at (-2.4,0.25) {\scriptsize $2^-$-hf-hdl};
\draw (0,0) ellipse (1.5cm and 1cm);
\draw (-0.4,0.1) to[out=-60,in=-120] (0.4,0.1);
\draw (-0.29,0) to[out=60,in=120] (0.29,0);
\draw[fill=black,opacity=0.2] (0,0) ellipse (0.3cm and 0.2cm);
\draw[line width=0.1mm] (-2,0.5) -- (-1.8,0.8) -- (0,0);
\node at (3.8,0.1) {$= \mathcal{D}^3
\xrightarrow{3^- \text{-hf-hdl}}
S^3
\xrightarrow{4 \text{-hdl}}
\emptyset
$};
\draw[->, line width=0.3mm] (2.3,-0.3) to[out=-30,in=-150] (5.5,-0.3);
\node at (3.9,-0.95) {\scriptsize $3^+$-hf-hdl};
\end{tikzpicture}
\caption{
Half-handle decomposition for counit for $\YY_X^{(1)}$.
The $2^-$-half-handle corresponds to the
$1^-$-half-handle in the bottom row in
\figref{f:solid-torus-half-1},
and the $3^+$-half-handle,
here shown to be broken into a $3^-$-half-handle
followed by a $4$-handle,
corresponds to the final $2^+$-half-handle
in \figref{f:solid-torus-half-1}.
}
\label{f:solid-torus-counit-1}
\end{figure}
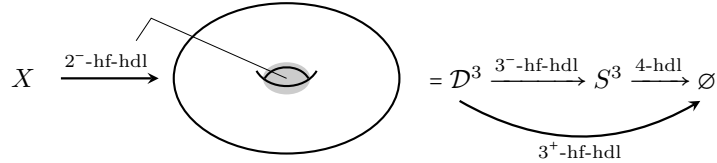

\begin{example}[Solid torus II]
\label{x:solid-torus-2}
Now let us consider the other way to have the solid torus
be the double of a manifold with corners.
Slicing the solid torus like a bagel,
we see that it is the double of
$\hDk{2} \times S^1$
(with corner $\{-1,1\} \times S^1$,
where $\{-1,1\}$ are the corners of $\hDk{2}$),
so that the two halves glue along an annulus.

This half-handle decomposition is simpler than
the previous one.
Indeed, it is clear from \figref{f:solid-torus-half-2}
that just taking the $0^-$- and $1^-$-half-handle
gives us a half-handle decomposition of $\hDk{2} \times S^1$.
Again, to be precise,
the $0^-$-half-handle starts at the $(0,3/4)$ point on $\del X$.
Then the $1^-$-half-handle grows out of the $0^-$-half-handle
and closes up into a longitude.

\begin{figure}[h] 
\centering
\input{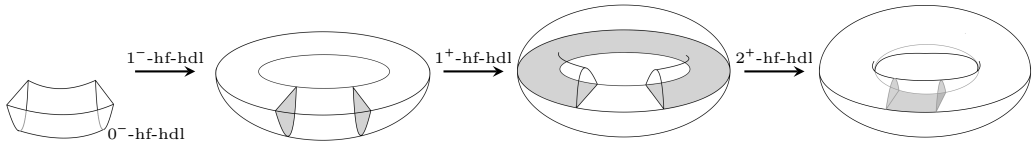}
\caption{Half-handle decomposition for $X$,
for $\YY_X^{(2)}$.
Note that in the third diagram,
the attaching region of the $1^+$-half-handle
is the shaded region,
the outgoing boundary is made of the two half-disks
facing each other,
and the vertical boundary is the top curved surface.
}
\label{f:solid-torus-half-2}
\end{figure}

The corresponding half-handle decomposition for $\YY_X^{(2)}$
is shown in \figref{f:solid-torus-Yprod-2}.
In terms of surgery on 3-manifolds,
for the $1^-$-half-handle,
we take out a $\hDk{3}$ around
the $(0,3/4)$ point from the left $X$, and
a $\hDk{3}$ around the $(0,1/4)$ point from the right $X$,
and attaches a solid cylinder.
Next, the attaching sphere of the $2^-$-half-handle
is a circle on the boundary
that travels along the longitudes of the $X$'s
and goes back and forth along the solid cylinder just attached.
A neighborhood of the attaching sphere is removed,
and replaced with a thickened 2-disk.

\begin{figure}[h] 
\centering
\input{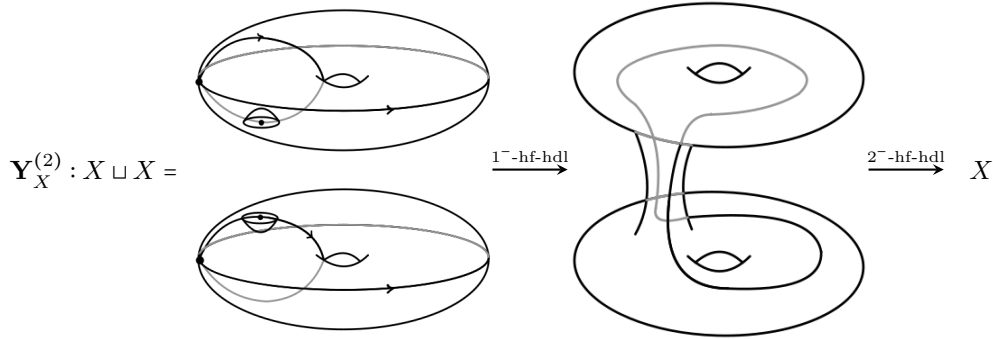}
\caption{$\YY_X^{(2)}$ half-handle decomposition}
\label{f:solid-torus-Yprod-2}
\end{figure}

The resulting Y-product is different from the previous example.
Indeed, we may view the solid torus as $I \times \text{Ann}$.
Then this Y-product is (Y-product on $I$) times Ann.
The Y-product on $I$ is homotopy equivalent to a point,
and thus this Y-product $\YY_X^{(2)}$
is homotopy equivalent to a circle,
whereas the previous Y-product $\YY_X^{(1)}$
is homotopy equivalent to a usual pair of pants.

Note that the vertical boundary of this Y-product
is also diffeomorphic to the usual pair of pants times $S^1$.

A half-handle decomposition for the
counit for the corresponding Y-coproduct is given
in \figref{f:solid-torus-counit-2} below.
\hfill $\triangle$
\end{example}

\begin{figure}[h] 
\centering
\begin{tikzpicture}
\node at (-2.2,0) {$X =$};
\draw (0,0) ellipse (1.5cm and 1cm);
\draw (-0.4,0.1) to[out=-60,in=-120] (0.4,0.1);
\draw (-0.29,0) to[out=60,in=120] (0.29,0);
\draw[fill=black,opacity=0.2] (0,1)
	.. controls +(-150:0.1cm) and +(90:0.2cm) .. (-0.1,0.55)
	.. controls +(-90:0.2cm) and +(150:0.1cm) .. (0,0.15)
	-- (0.1,0.15)
	.. controls +(30:0.1cm) and +(-90:0.2cm) .. (0.2,0.55)
	.. controls +(90:0.2cm) and +(-30:0.1cm) .. (0.1,1);
\draw[opacity=0.2] (0.1,0.15)
	.. controls +(150:0.1cm) and +(-90:0.2cm) .. (0,0.55)
	.. controls +(90:0.2cm) and +(-150:0.1cm) .. (0.1,1);
\node at (3.5,0) {
$\xrightarrow{2^+ \text{-hf-hdl}}
\mathcal{D}^3
\xrightarrow{3^+ \text{-hf-hdl}}
\emptyset
$};
\draw[line width=0.1mm] (0.1,0.55) -- (1.6,0.8) -- (1.8,0.4);
\end{tikzpicture}
\caption{
Half-handle decomposition for counit for $\YY_X^{(2)}$.
The $2^+$- and $3^+$-half-handles correspond to the
$1^+$- and $2^+$-half-handles in
\figref{f:solid-torus-half-2}.
}
\label{f:solid-torus-counit-2}
\end{figure}
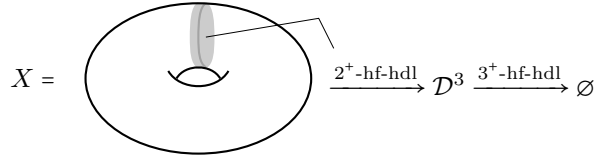

\begin{remark}
\label{r:comment-two-decompositions}
We note that, as one might expect,
the two half-handle decompositions for $X$
can be related by handle cancellations and sliding.
As mentioned before, the $1^-$-half-handle and $2$-handle
in the first example can be merged/canceled
to get a $1^+$-half-handle,
so that the overall half-handle decomposition of the
solid torus consists of half-handles of index
$0^-, 1^+, 1^-, 2^+$ in that order.
In the second example, the half-handle decomposition consists of
half-handles of index $0^-, 1^-, 1^+, 2^+$ in that order.
It is easy to check that the $1^-$- and $1^+$-half-handles
can slide off each other
and be swapped in the order of handle attachments.
But this changes the resulting Y-products because
the manifold obtain from the first half of the handles
is changed.
\end{remark}

\begin{proposition}
\label{p:Psi-Yprod-equiv}
As relative cobordisms $X \sqcup X \to X$,
\begin{align*}
\YY_X^{(1)} \circ (\ov{\Psi} \sqcup \ov{\Psi})
&=
\ov{\Psi} \circ \YY_X^{(2)}
\\
\YY_X^{(1)} \circ (\Psi \sqcup \Psi)
&=
\Psi \circ \YY_X^{(2)} \circ P
\end{align*}
where $P$ swaps the two copies of $X$.
\end{proposition}
\begin{proof}
\figref{f:Psi-Yprod-equiv}
proves the first equality.
They are equal as relative cobordisms by a simple
2-3 handle cancellation.
We just note that in the left diagram,
the attaching region for the
$1^-$- and $2^-$-half-handles of $\YY_X^{(1)}$
looks like that of $\YY_X^{(2)}$
because of $\ov{\Psi}$.

The second equality needs an extra $P$
because $\Psi$ turns $\del X$ the other way
compared to $\ov{\Psi}$,
so the diagram for $\YY_X^{(1)} \circ (\Psi \sqcup \Psi)$
would look like the left diagram
in \figref{f:Psi-Yprod-equiv},
but the second $X$ would be stacked on top of the first.

\begin{figure}[h] 
\centering
\input{diagram-Psi-Y1-Y2.tikz}
\caption{
$\YY_X^{(1)} \circ (\ov{\Psi} \sqcup \ov{\Psi})
= \ov{\Psi} \circ \YY_X^{(2)}$
}
\label{f:Psi-Yprod-equiv}
\end{figure}

\end{proof}

\begin{proposition}
\label{p:Psi-Yprod-equiv-rev}
Let $K$ be the cornered cobordism $K: X \to X$
given by attaching two 2-handles as in
\figref{f:solid-torus-K}.
Then as relative cobordisms $X \sqcup X \to X$,
\begin{align*}
\YY_X^{(2)} \circ (\Psi \sqcup \Psi)
&=
\Psi \circ \YY_X^{(1)} \circ (K \sqcup \id_X)
=
\Psi \circ \YY_X^{(1)} \circ (\id_X \sqcup K)
\\
\YY_X^{(2)} \circ (\ov{\Psi} \sqcup \ov{\Psi}) \circ P
&=
\ov{\Psi} \circ \YY_X^{(1)} \circ (K \sqcup \id_X)
=
\ov{\Psi} \circ \YY_X^{(1)} \circ (\id_X \sqcup K)
\end{align*}
\end{proposition}

\begin{proof}
In Figures \ref{f:solid-torus-K-Yprod-1},
\ref{f:solid-torus-K-Yprod-2},
half-handles are only added in the second step
(at $\YY_X^{(1)},\YY_X^{(2)}$),
and the same half-handles are added in both figures.
Thus, from the final diagrams,
it is clear that performing an extra $\Psi$
in \figref{f:solid-torus-K-Yprod-1}
results in \figref{f:solid-torus-K-Yprod-2}.

\begin{figure}[h] 
\centering
\begin{tikzpicture}
\node at (-3.5,0) {$K := \Psi \circ \overline{\Psi} =
\overline{\Psi} \circ \Psi =$};
\draw (0,0) ellipse (1.5cm and 1cm);
\draw (-0.2,0.1) to[out=-60,in=-120] (0.2,0.1);
\draw (-0.1,0) to[out=60,in=120] (0.1,0);
\draw (0,0) ellipse (0.8cm and 0.5cm);
\draw[overline] (-0.8,0) circle (0.2cm);
\draw[overline] (0.8,0) arc (0:-180:0.8cm and 0.5cm);
\node at (-2,0.8) {\tiny 2-handle};
\draw[line width=0.1mm] (-1.5,0.6) -- (-1,0.1);
\draw[line width=0.1mm] (-1.4,0.8) -- (-0.5,0.4);
\end{tikzpicture}
\caption{
Cornered cobordism $K: X \to X$;
as a 4-manifold, $K$ is $\Dk{2} \times S^2$.
}
\label{f:solid-torus-K}
\end{figure}
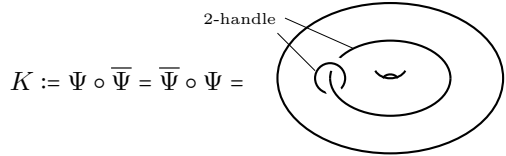

\begin{figure}[h] 
\centering
\input{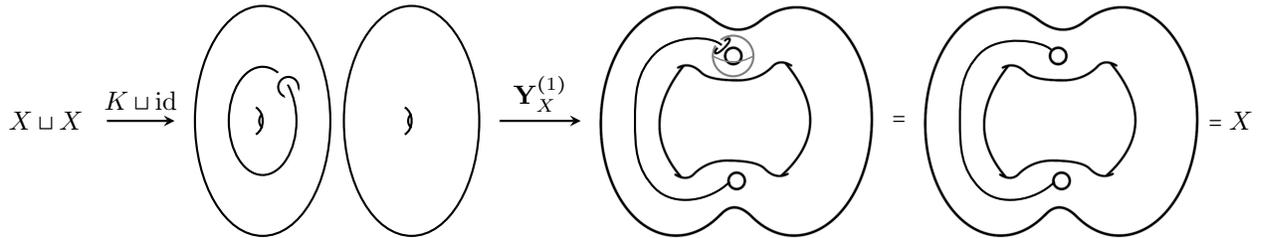}
\caption{
$\YY_X^{(1)} \circ (K \sqcup \id_X)$.
In the penultimate equality,
the 2- and 3-handles cancel.
}
\label{f:solid-torus-K-Yprod-1}
\end{figure}

\begin{figure}[h] 
\centering
\input{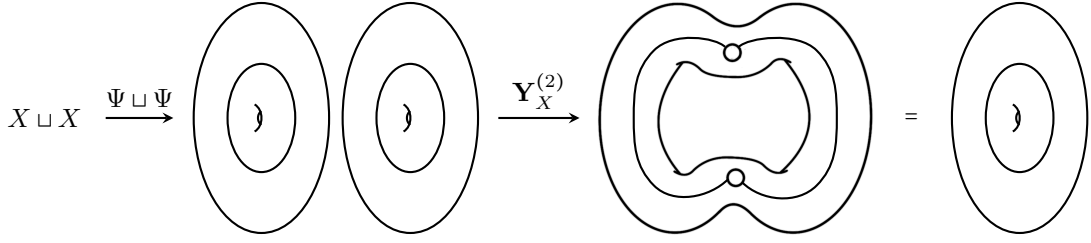}
\caption{
$\YY_X^{(2)} \circ (\Psi \sqcup \Psi)$.
In the last equality, the 1-handle is canceled by
one of the 2-handles.
}
\label{f:solid-torus-K-Yprod-2}
\end{figure}

\end{proof}

We also consider similar results for the coproducts,
which follow simply from dualizing the picture:

\begin{proposition}
\label{p:Psi-Ycoprod-equiv}
As relative cobordisms $X \to X \sqcup X$,
\begin{align*}
(\Psi \sqcup \Psi) \circ \coYY_X^{(1)} 
&=
\coYY_X^{(2)} \circ \Psi
\\
(\ov{\Psi} \sqcup \ov{\Psi}) \circ \coYY_X^{(1)} 
&=
P \circ \coYY_X^{(2)} \circ \ov{\Psi}
\end{align*}
where $P$ swaps the two copies of $X$.
\end{proposition}

\begin{proposition}
\label{p:Psi-Ycoprod-equiv-rev}
Let $K$ be the cornered cobordism $K: X \to X$
given by attaching two 2-handles as in
\figref{f:solid-torus-K}.
Then as relative cobordisms $X \sqcup X \to X$,
\begin{align*}
(\ov{\Psi} \sqcup \ov{\Psi}) \circ \coYY_X^{(2)}
&=
(K \sqcup \id_X) \circ \coYY_X^{(1)} \circ \ov{\Psi} 
=
(\id_X \sqcup K) \circ \coYY_X^{(1)} \circ \ov{\Psi} 
\\
P \circ (\Psi \sqcup \Psi) \circ \coYY_X^{(2)}
&=
(K \sqcup \id_X) \circ \coYY_X^{(1)} \circ \Psi
=
(\id_X \sqcup K) \circ \coYY_X^{(1)} \circ \Psi
\end{align*}
where $P$ swaps the two copies of $X$.
\end{proposition}

\begin{remark}
\label{r:frobenius-topology}
\prpref{p:Psi-Yprod-equiv-rev} essentially formally
follows from \prpref{p:Psi-Ycoprod-equiv}
(which in turn essentially formally follows from
\prpref{p:Psi-Yprod-equiv}).
Represent the first equation in \prpref{p:Psi-Ycoprod-equiv}
as follows:
\begin{equation}
\begin{tikzpicture}
\node[dotnode] (a) at (0,0) {};
\node at (0.125,0.1) {\tiny $1$};
\draw (a) -- (0,0.5);
\draw (a) -- (-0.4, -0.3);
\draw (a) -- (0.4, -0.3);
\node[dotnode] at (-0.2,-0.15) {};
\node[dotnode] at (0.2,-0.15) {};
\node at (-0.3,-0.05) {\tiny $\Psi$};
\node at (0.3,-0.05) {\tiny $\Psi$};
\end{tikzpicture}
=
\begin{tikzpicture}
\node[dotnode] (a) at (0,0) {};
\node at (0.15,0.05) {\tiny $2$};
\draw (a) -- (0,0.5);
\draw (a) -- (-0.4, -0.3);
\draw (a) -- (0.4, -0.3);
\node[dotnode] at (0,0.25) {};
\node at (0.15,0.3) {\tiny $\Psi$};
\end{tikzpicture}
\end{equation}

Then, we have
\begin{equation}
\label{e:K-Psi-Y1-Y2}
\begin{tikzpicture}
\node[dotnode] (a) at (0,0) {};
\draw (a) -- (0,-0.5);
\draw (a) -- (-0.4,0.3);
\draw (a) -- (0.4,0.3);
\node[dotnode] at (-0.2,0.15) {};
\node[dotnode] at (0,-0.35) {};
\node at (0.13,-0.1) {\tiny 1};
\node at (-0.25,0) {\tiny $K$};
\node at (0.15,-0.35) {\tiny $\Psi$};
\end{tikzpicture}
=
\begin{tikzpicture}
\node[dotnode] (a) at (0,0) {};
\draw (a) -- (0,0.3);
\draw (a) .. controls +(-150:0.3cm) and +(90:0.2cm) ..
	(-0.3,-0.3) .. controls +(-90:0.3cm) and +(-90:0.3cm) ..
	(-0.6,-0.3) -- (-0.6,0.3);
\draw (a)  .. controls +(-30:0.3cm) and +(90:0.2cm) ..
	(0.3,-0.3) -- (0.3,-0.6);
\node[dotnode] at (-0.6,0.15) {};
\node[dotnode] at (-0.6,-0.15) {};
\node[dotnode] at (0.3,-0.5) {};
\node at (-0.75,0.15) {\tiny $\Psi$};
\node at (-0.75,-0.15) {\tiny $\ov{\Psi}$};
\node at (0.45,-0.5) {\tiny $\Psi$};
\node at (0.1,0.1) {\tiny 1};
\end{tikzpicture}
=
\begin{tikzpicture}
\node[dotnode] (a) at (0,0) {};
\draw (a) -- (0,0.3);
\draw (a) .. controls +(-150:0.3cm) and +(90:0.2cm) ..
	(-0.3,-0.3) .. controls +(-90:0.3cm) and +(-90:0.3cm) ..
	(-0.6,-0.3) -- (-0.6,0.3);
\draw (a)  .. controls +(-30:0.3cm) and +(90:0.2cm) ..
	(0.3,-0.3) -- (0.3,-0.6);
\node[dotnode] at (-0.6,0.15) {};
\node[dotnode] at (-0.2,-0.1) {};
\node[dotnode] at (0.2,-0.1) {};
\node at (-0.75,0.15) {\tiny $\Psi$};
\node at (-0.3,0) {\tiny $\Psi$};
\node at (0.15,-0.25) {\tiny $\Psi$};
\node at (0.1,0.1) {\tiny 1};
\end{tikzpicture}
=
\begin{tikzpicture}
\node[dotnode] (a) at (0,0) {};
\draw (a) -- (0,0.3);
\draw (a) .. controls +(-150:0.3cm) and +(90:0.2cm) ..
	(-0.3,-0.3) .. controls +(-90:0.3cm) and +(-90:0.3cm) ..
	(-0.6,-0.3) -- (-0.6,0.3);
\draw (a)  .. controls +(-30:0.3cm) and +(90:0.2cm) ..
	(0.3,-0.3) -- (0.3,-0.6);
\node[dotnode] at (-0.6,0.15) {};
\node[dotnode] at (0,0.15) {};
\node at (-0.75,0.15) {\tiny $\Psi$};
\node at (0.15,0.15) {\tiny $\Psi$};
\node at (0,-0.15) {\tiny 2};
\end{tikzpicture}
=
\begin{tikzpicture}
\node[dotnode] (a) at (0,0) {};
\draw (a) -- (0,-0.5);
\draw (a) -- (-0.4,0.3);
\draw (a) -- (0.4,0.3);
\node[dotnode] at (-0.2,0.15) {};
\node[dotnode] at (0.2,0.15) {};
\node at (0.1,-0.1) {\tiny 2};
\node at (-0.35,0.1) {\tiny $\Psi$};
\node at (0.35,0.1) {\tiny $\Psi$};
\end{tikzpicture}
\end{equation}

\end{remark}


\subsection{Solid handlebodies}
\label{s:topology-examples-handlebody}

We can generalize the discussion above to
handlebodies of higher genus.
We fix a specific genus $g$ solid handlebody,
denoted $X_g$,
constructed from $g$ copies of $X$,
with a 1-handle joining the
$(1/2,1/2)$ point on the $i$-th $\del X$
to the $(0,0)$ point on the $(i+1)$-st $\del X$.
This is a little awkward to draw
(see \figref{f:solid-handlebody-skein-basis})
but the point is that $\Psi$ naturally extends to $X_g$.

Indeed, applying $\Psi$ to the $i$-th $X$
is compatible with applying $\ov{\Psi}$ to the $(i+1)$-st $X$.
We thus define $\Psi_g$ to be the cornered cobordism
that results from alternately applying $\Psi$
and $\ov{\Psi}$ to the copies of $X$ in $X_g$
(starting with $\Psi$ on the first copy);
we define $\ov{\Psi}_g$ to be the same except
it starts with $\ov{\Psi}$ on the first copy of $X$.

For application in \xmpref{x:Psi-S-matrix-handlebody},
we also consider slight variants
$\Psi_g', \ov{\Psi}_g'$ of $\Psi_g, \ov{\Psi}_g$.
The only alteration made is near the $(0,0)$ point
of the boundary of the first solid torus,
where the identifications $\Phi, \ov{\Phi}: X' \simeq X$
are twisted so that they also fix a neighborhood of $(0,0)$.
(In the diagram below,
the segments labeled $\vec{m},\vec{l}$ in the right diagram
are the images of $\vec{m},\vec{l}$ in the left diagram.)

\begin{equation}
\Phi' = \Phi
\textrm{ except in a neighborhood of }(0,0):
\begin{tikzpicture}
\draw[midarrow={0.8}] (0,-1) -- (0,1);
\draw[midarrow={0.8}] (-1,0) -- (1,0);
\node at (-0.2,0.6) {\tiny $\vec{m}$};
\node at (0.6,0.2) {\tiny $\vec{l}$};
\node[dotnode] at (0,0) {};
\draw (0,0) circle (1cm);
\end{tikzpicture}
\mapsto
\begin{tikzpicture}
\foreach \x in {0,270}
\draw[rotate=\x, midarrow={0.6}] (0,0) --
  (0,0.05) .. controls +(up:0.5cm) and +(180:0.4cm) .. (1,0);
\foreach \x in {90,180}
\draw[rotate=\x] (0,0) --
  (0,0.05) .. controls +(up:0.5cm) and +(180:0.4cm) .. (1,0);
\draw (0,0) circle (1cm);
\node[dotnode] at (0,0) {};
\node at (0.4,0.35) {\tiny $\vec{m}$};
\node at (0.35,-0.4) {\tiny $\vec{l}$};
\end{tikzpicture}
\end{equation}
\begin{equation}
\ov{\Phi}' = \ov{\Phi}
\textrm{ except in a neighborhood of }(0,0):
\begin{tikzpicture}
\draw[midarrow={0.8}] (0,-1) -- (0,1);
\draw[midarrow={0.8}] (-1,0) -- (1,0);
\node at (-0.2,0.6) {\tiny $\vec{m}$};
\node at (0.6,0.2) {\tiny $\vec{l}$};
\node[dotnode] at (0,0) {};
\draw (0,0) circle (1cm);
\end{tikzpicture}
\mapsto
\begin{tikzpicture}
\foreach \x in {0,270}
\draw[rotate=\x, midarrow={0.6}] (0,0) --
  (0,0.05) .. controls +(up:0.5cm) and +(0:0.4cm) .. (-1,0);
\foreach \x in {90,180}
\draw[rotate=\x] (0,0) --
  (0,0.05) .. controls +(up:0.5cm) and +(0:0.4cm) .. (-1,0);
\draw (0,0) circle (1cm);
\node[dotnode] at (0,0) {};
\node at (-0.4,0.35) {\tiny $\vec{m}$};
\node at (0.35,0.4) {\tiny $\vec{l}$};
\end{tikzpicture}
\end{equation}




We describe the analogs of the two Y-products
from the previous section.
We first need good half-handle decompositions for $X_g$.
In our definition of $X_g$ above,
it suggests that we build $g$ copies of $X$'s first,
and then add a 1-handle between each pair of adjacent $X$'s.
This ruins the symmetry and is not suitable for
constructing Y-products,
so we take a different approach.

Let us build $g$ copies of half of $X$
in the manner that would yield $\YY_X^{(1)}$
(as in the left half of \figref{f:solid-torus-half-1}).
Now add a $1^-$-half-handle wherever the 1-handles
between $X$'s should be added.
Observe that upon adding the dual $1^+$-half-handles,
these make up those 1-handles.
From here, it is easy to see that by continuing with
the dual handles on each $X$
(as in the right half of \figref{f:solid-torus-half-1}),
we arrive at $X_g$.

We denote the Y-product that results from the above
half-handle decomposition of $X_g$
by $\YY_{X_g}^{(1)}$,
since it is closely related to $\YY_X^{(1)}$.
Indeed, it is not hard to see that $\YY_{X_g}^{(1)}$
can also be described as follows:
starting with $X_g \sqcup X_g$,
``apply'' $\YY_X^{(1)}$ at every position,
now we have $g$ copies of $X$,
but between adjacent $X$'s, we have not one but two 1-handles;
apply a $2^-$-half-handle to eliminate one of them.

We define $\YY_{X_g}^{(2)}$ the same way,
except using the half-handle decomposition of $X$
that would yield $\YY_X^{(2)}$ instead of $\YY_X^{(1)}$.

A half-handle decomposition for the
counit for the corresponding Y-coproduct
follows directly from the solid torus case;
apply $2^+$-half-handles that cut the connection
between solid tori,
and then apply the counit corresponding to $\YY_X^{(1)}$
or $\YY_X^{(2)}$ to each solid torus.

\begin{proposition}
\label{p:Psi-Yprod-equiv-handlebody}
As relative cobordisms $X_g \sqcup X_g \to X_g$,
\begin{align*}
\YY_{X_g}^{(1)} \circ (\ov{\Psi}_g \sqcup \ov{\Psi}_g)
&=
\ov{\Psi}_g \circ \YY_{X_g}^{(2)}
\\
\YY_{X_g}^{(1)} \circ (\Psi_g \sqcup \Psi_g)
&=
\Psi_g \circ \YY_{X_g}^{(2)} \circ P
\end{align*}
where $P$ swaps the two copies of $X_g$.
\end{proposition}

\begin{proof}
Follows easily from the solid torus case
(\prpref{p:Psi-Yprod-equiv}).
\end{proof}

\section{Verlinde Formula}

The main application we have for this construction
is a topological explanation,
via Crane-Yetter theory,
for the Verlinde formula \cite{VERLINDE1988360},
as stated in \prpref{p:verlinde-formula}
in the context of premodular categories.

As is well known, the skein module of a solid torus
is the Verlinde algebra.
We prove that the Y-products
considered in \secref{s:topology-examples-torus}
give rise to the fusion and convolution products
on the Verlinde algebra,
and the cornered cobordism $\Psi$
gives rise to the $S$-matrix.
Thus, \prpref{p:Psi-Yprod-equiv} gives a topological
proof of the Verlinde formula.

We discuss implications of applying the same techniques
to the higher genus solid handlebodies
(\secref{s:topology-examples-handlebody}),
in particular we derive a generalized $S$-matrix.

See the discussion in \cite{mathoverflowVerlinde}
on the Verlinde formula.

\subsection{Skeins, Witten-Reshetikhin-Turaev}

Skeins, or skein relations,
have been around, at least implicitly,
since Alexander introduced his
polynomial invariant of knots and links
\cite{alexander1928topological},
and brought to greater attention by Conway
\cite{conway1970enumeration}.
After Jones's discovery of his polynomial invariant
\cite{jones1985polynomial},
there was an explosion of activity on skeins,
in particular,
Kauffman found a skein formulation \cite{kauffman1987state},
and a common generalization, the HOMFLY-PT polynomial
was constructed
\cite{freyd1985new},\cite{przytycki1988invariants}.

Skeins are, in a sense, 1-chains with non-abelian coefficients.
A skein in a 3-manifold $M$
is an embedded graph with extra data,
known as a ``coloring'',
attached to its edges and vertices.
More generally,
a skein is a linear combination of such graphs.
There should be a well-defined notion of the
evaluation of a skein in a ball;
two skeins are considered equivalent if they
agree outside a ball and have the same evaluation
in the ball.
In general, the graphs are allowed to intersect the boundary
of $M$ transversely at a set of points $\VV$,
known as a ``boundary value'',
which inherit the coloring from the graph.
The skein module of $M$ with boundary value $\VV$,
denoted $\Skein(M; \VV)$,
is the space of skeins with a fixed boundary value $\VV$.

From here on,
we will assume that the reader is familiar
with the basics of skein theory.
We give a quick summary of the
notations and conventions about premodular categories,
as presented in \cite{tham-reduced}
that are relevant for our applications in this paper;
see also \cite{kirillov-stringnet}.

We fix a semisimple premodular
(i.e. ribbon fusion) category $\cA$
over an algebraically closed field $\kk$.
Here are some notation:
\begin{itemize}
\item $\one$: unit simple object;

\item $X^*$: the (left) dual of $X$;
	we suppress associativity, unit, and pivotal morphisms
	$X \simeq X^{**}$;

\item $X_i$ for $i \in \Irr(\cA)$:
	representative for each isomorphism class of simple object,
	such that $X_{i^*} = X_i^*$;

\item $d_i$: categorical dimension of $X_i$;

\item $\sqrt{d_i}$: choice of square root such that
	$\sqrt{d_\one} = 1$ and $\sqrt{d_i} = \sqrt{d_{i^*}}$;

\item $\cD = \sum_{i \in \Irr(\cA)} d_i^2$:
	the global dimension;
	$\cD\neq 0$ by \cite{ENO2005-fusion};

\item $c_{X,Y}$: braiding of $X$ strand over $Y$ strand:
\begin{equation}
c_{X,Y} = 
\begin{tikzpicture}
\begin{scope}[scale={0.6}]
\draw[->] (1,0.5) to[out=-90,in=90] (0,-0.5);
\draw[->, overline] (0,0.5) to[out=-90,in=90] (1,-0.5);
\node at (0,0.7) {\tiny $X$};
\node at (1,0.7) {\tiny $Y$};
\end{scope}
\end{tikzpicture}
\label{e:braiding-convention}
\end{equation}

\item $\theta_X$: the twist operator,
	given by a left-hand twist:
\begin{equation}
\theta_X =
\begin{tikzpicture}
 \draw[->]
  (0.2,0) .. controls +(up:0.1cm) and +(up:0.4cm) ..
  (0,-0.3);
\draw[thin_overline={1.5}]
  (0,0.3) .. controls +(down:0.4cm) and +(down:0.1cm) ..
  (0.2,0);
\node at (0,0.4) {\tiny $X$};
\end{tikzpicture}
\end{equation}

\item $s_{ij} = \trace c_{j,i^*} \circ c_{i^*,j}$:
	the $S$-matrix; we also define
$\ov{s}_{ij} = \trace c_{i^*,j}^\inv \circ c_{j,i^*}^\inv
= s_{i^* j}$,
which we refer to as the $\ov{S}$-matrix; equivalently,
\begin{equation} \label{e:S-matrix}
\begin{tikzpicture}
\draw (0,0) -- (0,0.6);
\node at (0.15,-0.4) {\tiny $j$};
\draw[overline={1.5}] (0,0.1) circle (0.25cm);
\draw[midarrow] (-0.25,0.05) -- (-0.25,0.04);
\node at (-0.5,0.1) {\tiny $i$};
\draw[overline={1.5},midarrow={0.9}] (0,0) -- (0,-0.5);
\end{tikzpicture}
=
\frac{s_{ij}}{d_j} \cdot \id_{X_j}
\;\; ; \;\;
\begin{tikzpicture}
\draw[midarrow={0.9}] (0,0) -- (0,-0.5);
\node at (0.15,-0.4) {\tiny $j$};
\draw[overline={1.5}] (0,0.1) circle (0.25cm);
\draw[midarrow] (-0.25,0.05) -- (-0.25,0.04);
\node at (-0.5,0.1) {\tiny $i$};
\draw[overline={1.5}] (0,0) -- (0,0.6);
\end{tikzpicture}
=
\frac{\ov{s}_{ij}}{d_j} \cdot \id_{X_j}
\end{equation}
(Note $s_{ij}$ as defined here is denoted $\tilde{s}_{ij}$
in \cite{BakK}).
\end{itemize}

Concatenation of objects will mean tensor product;
this will always be the ``standard'' one,
e.g. the tensor product in $\cA$
or tensor product of vector spaces.
Most sums over simple objects, typically indexed
with a lower case Latin alphabet (usually $i,j,k$)
will implicitly be over $\Irr(\cA)$.

We define the functor
$\eval{} : \cA^{\boxtimes n} \to \textrm{Vec}$
by
\begin{equation} \label{e:vev}
\eval{V_1,\ldots,V_n} = \Hom_\cA(\one,V_1\cdots V_n)
\;\;.
\end{equation}
The pivotal structure gives functorial isomorphisms
\begin{equation} \label{e:cyclic}
z:\eval{V_1,\ldots,V_n} \simeq
  \eval{V_n,V_1,\ldots,V_{n-1}}
\end{equation}
such that $z^n = \id$,
so up to canonical isomorphism,
$\eval{V_1,\ldots,V_n}$ only depends on the cyclic
order of $V_1,\ldots,V_n$.

There is a non-degenerate pairing
$\mathbf{\text{ev}}: \eval{V_1,\ldots,V_n} \tnsr
\eval{V_n^*,\ldots,V_1^*} \to \kk$
obtained by post-composing with evaluation maps.
When two nodes in a graph are labeled by
the same Greek letter, say $\alpha$,
it stands for a summation over
a pair of dual bases:
\begin{equation}\label{e:summation_convention}
\begin{tikzpicture}
\begin{scope}[scale={0.7}]
\node[small_morphism] (ph) at (0,0) {$\alpha$};
\draw[->] (ph)-- +(240:1cm) node[pos=0.7, left] {\tiny $V^{*}_1$} ;
\draw[->] (ph)-- +(180:1cm);
\draw[->] (ph)-- +(120:1cm);
\draw[->] (ph)-- +(60:1cm);
\draw[->] (ph)-- +(0:1cm);
\draw[->] (ph)-- +(-60:1cm) node[pos=0.7, right] {\tiny $V^{*}_n$};
\node[small_morphism] (ph') at (3,0) {$\alpha$};
\draw[->] (ph')-- +(240:1cm) node[pos=0.7, left] {\tiny $V_n$} ;
\draw[->] (ph')-- +(180:1cm);
\draw[->] (ph')-- +(120:1cm);
\draw[->] (ph')-- +(60:1cm);
\draw[->] (ph')-- +(0:1cm);
\draw[->] (ph')-- +(-60:1cm) node[pos=0.7, right] {\tiny $V_1$};
\end{scope}
\end{tikzpicture}
\quad = \sum_\alpha\quad
\begin{tikzpicture}
\begin{scope}[scale={0.7}]
\node[small_morphism] (ph) at (0,0) {$\vphi_\alpha$};
\draw[->] (ph)-- +(240:1cm) node[pos=0.7, left] {\tiny $V^{*}_1$} ;
\draw[->] (ph)-- +(180:1cm);
\draw[->] (ph)-- +(120:1cm);
\draw[->] (ph)-- +(60:1cm);
\draw[->] (ph)-- +(0:1cm);
\draw[->] (ph)-- +(-60:1cm) node[pos=0.7, right] {\tiny $V^{*}_n$};
\node[small_morphism] (ph') at (3,0) {$\vphi^\alpha$};
\draw[->] (ph')-- +(240:1cm) node[pos=0.7, left] {\tiny $V_n$} ;
\draw[->] (ph')-- +(180:1cm);
\draw[->] (ph')-- +(120:1cm);
\draw[->] (ph')-- +(60:1cm);
\draw[->] (ph')-- +(0:1cm);
\draw[->] (ph')-- +(-60:1cm) node[pos=0.7, right] {\tiny $V_1$};
\end{scope}
\end{tikzpicture}
\end{equation}
where $\vphi_\alpha\in \eval{V_1,\dots, V_n}$,
$\vphi^\alpha\in \eval{V_n^*,\dots, V_1^*}$
are dual bases with respect to the pairing 
$\mathbf{\text{ev}}$.

A dashed line stands for the regular coloring,
i.e. the sum of all colorings
by simple objects $i$, each taken with coefficient $d_i$:
\begin{equation} \label{e:regular_color}
\begin{tikzpicture}
\draw[regular] (0,0.3) -- (0,-0.3);
\end{tikzpicture}
=\sum_{i\in \Irr(\cA)} d_i \quad 
\begin{tikzpicture}
\draw[midarrow={0.55}] (0,0.3)--(0,-0.3)
  node[pos=0.5, right] {$i$};
\end{tikzpicture}
\;\;.
\end{equation}
An oriented edge labeled $X$ is the same as
the oppositely oriented edge labeled $X^*$.

Note that although we will be discussing 
various pivotal multifusion categories,
the morphisms are described in terms of morphisms
in $\cA$, in particular
all morphisms depicted graphically are
of morphisms in $\cA$,
unless specified otherwise.

Finally, we record some facts and lemmas that are useful
for computations.
We refer readers to \cite{kirillov-stringnet} for proofs.

\begin{equation} \label{e:combine}
\begin{tikzpicture}
\node[small_morphism] (T) at (0,-0.4) {$\alpha$};
\node[small_morphism] (B) at (0,0.4) {$\alpha$};
\node at (0,1) {$\dots$};
\node at (0,-1) {$\dots$};
\draw[midarrow] (T)-- +(-120:0.8cm)
  node[pos=0.5,left] {\small $V_1$};
\draw[midarrow] (T)-- +(-60:0.8cm)
  node[pos=0.5,right] {\small $V_n$};
\draw[midarrow_rev] (B)-- +(120:0.8cm)
  node[pos=0.5,left] {\small $V_1$};
\draw[midarrow_rev] (B)-- +(60:0.8cm)
  node[pos=0.5,right] {\small $V_n$};
\draw[regular] (T) -- (B); 
\end{tikzpicture}
=
\begin{tikzpicture}
\node at (0,0) {$\dots$};
\draw[midarrow] (-0.3,1)-- (-0.3,-1)
  node[pos=0.5,left] {\small $V_1$};
\draw[midarrow] (0.3,1)-- (0.3,-1)
  node[pos=0.5,right] {\small $V_n$};
\end{tikzpicture}
\end{equation}
\begin{equation}
\label{e:combine-alt}
\begin{tikzpicture}
\node[small_morphism] (B) at (0,-0.5) {$\alpha$};
\node[small_morphism] (T) at (0,0.5) {$\alpha$};
\draw[midarrow_rev={0.6}] (T) -- (0,1);
\draw[midarrow={0.6}] (B) -- (0,-1);
\node at (0,0) {$\dots$};
\draw[midarrow] (T) .. controls +(-150:0.5cm) and +(150:0.5cm) .. (B);
\draw[midarrow] (T) .. controls +(-30:0.5cm) and +(30:0.5cm) .. (B);
\node at (-0.7,0) {\small $V_1$};
\node at (0.7,0) {\small $V_n$};
\node at (0.3,0.8) {\small $X_k$};
\node at (0.3,-0.8) {\small $X_k$};
\end{tikzpicture}
=
\frac{\dim (\Hom(X_k, V_1 \tnsr \cdots \tnsr V_n))}{d_k}
\;
\cdot
\;
\begin{tikzpicture}
\draw[midarrow] (0,1)-- (0,-1)
  node[pos=0.5,right] {\tiny $k$};
\end{tikzpicture}
\end{equation}
\begin{equation} \label{e:sliding}
\textrm{Sliding lemma:}\;\;\;\;\;\;
\begin{tikzpicture}
\draw[regular] (0,0) circle(0.4cm);
\path[subgraph] (0,0) circle(0.3cm); 
\draw (0,1.2)..controls +(-90:0.8cm) and +(90:0.5cm) ..
  (-0.55, 0) ..controls +(-90:0.5cm) and +(90:0.8cm) ..
  (0,-1.2);
 \end{tikzpicture}
\quad=\quad
\begin{tikzpicture}
\draw[regular] (0,0) circle(0.4cm);
\node[small_morphism] (top) at (0,0.45) {\tiny $\alpha$};
\node[small_morphism] (bot) at (0,-0.45) {\tiny $\alpha$};
\path[subgraph] circle(0.3cm); 
\draw (0,1.2)--(top) (bot)--(0,-1.2);    
\end{tikzpicture}
\quad=\quad 
 \begin{tikzpicture}
 \draw[regular] (0,0) circle(0.4cm);
 \path[subgraph] (0,0) circle(0.3cm); 
 \draw (0,1.2)..controls +(-90:0.8cm) and +(90:0.5cm) ..
   (0.55, 0) ..controls +(-90:0.5cm) and +(90:0.8cm) ..
   (0,-1.2);
 \end{tikzpicture}
\end{equation}
where the shaded region can contain anything.

\begin{equation}
\label{e:charge-conservation}
\frac{1}{\cD}
\begin{tikzpicture}
\draw[midarrow={0.9}] (0,0) -- (0,-0.5);
\node at (0.15,-0.4) {\tiny $i$};
\draw[overline={1.5},regular] (0,0.1) circle (0.25cm);
\draw[overline={1.5}] (0,0) -- (0,0.5);
\end{tikzpicture}
=
\delta_{i,1} \id_{X_i}
\;\;
\text{(when $\cA$ is modular)}
\end{equation}

More generally,
from \cite{muger2003structure},
when $\cA$ is just premodular,
if $J \subset \Irr(\cA)$ is the subset of
transparent simple objects,
then
\begin{equation}
\label{e:charge-conservation-premodular}
\frac{1}{\cD}
\begin{tikzpicture}
\draw[midarrow={0.9}] (0,0) -- (0,-0.5);
\node at (0.15,-0.4) {\tiny $i$};
\draw[overline={1.5},regular] (0,0.1) circle (0.25cm);
\draw[overline={1.5}] (0,0) -- (0,0.5);
\end{tikzpicture}
=
\begin{cases}
	\id_{X_i} & \textrm{ if } i \in J \\
	0 & \textrm{ else}
\end{cases}
\end{equation}

The Verlinde algebra is the Grothendieck ring
$V = K(\cA) \tnsr_{\mathbb{Z}} \kk$.
It has a natural basis $x_i := [X_i]$.
The multiplication is given by the fusion rules:
\begin{equation}
x_i x_j = \sum_{k \in \Irr(\cA)} N_{ij}^k x_k
\end{equation}
where $N_{ij}^k$ is the multiplicity of $X_k$
in the direct sum decomposition of $X_i \tnsr X_j$.
We refer to this as the \emph{fusion product}.

We can define another algebra structure $*$ on $V$,
which we refer to as the \emph{convolution product},
given by:
\begin{equation}
x_i * x_j = (\delta_{ij} / d_i) \cdot x_i
\end{equation}

The $S$-matrix diagonalizes the fusion rules
(see \cite{muger2003structure}*{Lemma 2.4 iii},
\cite{BakK}):

\begin{proposition}[Verlinde Formula]
\label{p:verlinde-formula}
Let $s : V \simeq V$ be the linear operator that,
in the basis $\{x_i\}$,
is given by the $S$-matrix; that is,
\[
s(x_i) = \sum_j s_{ij} x_j
\]
Then
\[
s(xy) = s(x) * s(y)
\]
\end{proposition}

When $s$ is invertible, setting $x = x_i, y = x_j$,
we get the more familiar form
\[
N_{ij}^k = \sum_l \frac{ s_{jl} s_{il} (s^\inv)_{lk} }{s_{0l}}
\]

\subsection{Crane-Yetter}
In the second author's PhD thesis
\cite{tham-thesis},
we developed an extended Crane-Yetter theory
in terms of skeins.


\begin{definition}
The state spaces associated to a 3-manifold with boundary
$M$, depending on choice of boundary value $\VV$ on $\del M$,
is defined to be the skein module $\ZCY(M;\VV) = \Skein(M;\VV)$.
\hfill $\triangle$
\end{definition}

Before we give the value of Crane-Yetter on
a cornered cobordism, let us briefly discuss
the relation to WRT.
For this discussion $\cA$ is modular.
We showed in \cite{tham-thesis} that,
as widely expected (see e.g. \cite{barrett-etal-2007}),
WRT theory is a boundary theory of the Crane-Yetter theory,
which we formulate as follows.
Given a cobordism between closed 2-manifolds
$M : N_- \to N_+$,
and a skein $\vphi$ in $M$ with boundary values
$\VV_-,\VV_+$ at $N_-,N_+$ respectively,
WRT assigns vector spaces $\ZRT(N_-;\VV_-), \ZRT(N_+;\VV_+)$
and a linear map between them
\begin{equation}
\label{e:zrt}
\ZRT(M,\vphi) : \ZRT(N_-;\VV_-) \to \ZRT(N_+;\VV_+)
\end{equation}
Now given a relative cobordism between 3-manifolds with
boundary which restricts to $M: N_- \to N_+$
on the vertical boundary, say
$(W, M) : (M_-, N_-) \to (M_+, N_+)$,
Crane-Yetter theory gives us a map
\begin{align}
\label{e:zrt-zcy}
\begin{split}
\ZCY(W,M,\vphi) : \ZCY(M_-;\VV_-) &\to \ZCY(M_+;\VV_+)
\\
\vphi_- &\mapsto \ZCY(W;N_+)(\vphi_- \cup \vphi)
\end{split}
\end{align}
that is, we glue a skein $\vphi_-$ in $M_-$ to $\vphi$
along the boundary value $\VV_-$,
treat it as a skein in $M_- \cup_{N_-} M$,
and apply the Crane-Yetter map associated to the
4-manifold $W$ with corner $N_+$
(see \figref{f:CY-RT} below).

\begin{figure}[h] 
\centering
\input{diagram-CY-RT.tikz}
\caption{$\ZRT$ in terms of $\ZCY$}
\label{f:CY-RT}
\end{figure}

We showed in \cite{tham-thesis} that
Equations \ref{e:zrt} and \ref{e:zrt-zcy}
agree up to a scalar depending only on
$\cA$ and the signature $\sigma$
and Euler characteristic $\chi$ of $W$:
$\ZCY(W) = \kappa^{\sigma(W)} \cD^{\chi(W)/2} \ZRT(\del_v W)$.
\footnote{
Note here $\ZCY$ is $Z^{\text{sk}}_{CY}$ in \cite{tham-thesis}.
}

We define the value of $\ZCY$ on a cornered cobordism $(W;N)$
by first defining it on handles,
or more precisely elementary (cornered) cobordisms,
and then, by decomposing $(W;N)$ into elementary ones,
we define $\ZCY(W;N)$ to be the composition;
we then show that the result is independent
of the choice of decomposition.
We will only need to compute cobordisms with
handles of index $2$ and $3$,
but we the value of $\ZCY$ on handles of all indices
for completeness:

\begin{definition}
\label{d:zcy-elementary}
Let $W : M \to_N M'$
be an elementary cornered cobordism of index $k$,
that is,
an identity cornered cobordism on $M$
composed with a $k$-handle attachment $\cH_k$.
For a boundary value $\VV \in \ZCY(N)$,
we define
\[
\ZCY(W;N) : \ZCY(M;\VV) \to \ZCY(M';\VV)
\]
case-by-case: given a colored ribbon graph
$\Gamma \in \ZCY(M;\VV)$,
$\ZCY(W;N)(\Gamma)$ is constructed as follows:
\begin{itemize}
\item $k=0$: $\cH_0$ adds a new $S^3$ component to $M$;
	we define
	\[
		\ZCY(W;N)(\Gamma) :=
	\cD \cdot \Gamma \cup \emptyskein{S^3}
	\]
	where $\emptyskein{S^3}$ is the empty skein in $S^3$;
\item $k=4$: $\cH_4$ kills off an $S^3$ component of $M$;
	writing $\Gamma = \Gamma' \sqcup \Gamma''$
	with $\Gamma'$ in the $S^3$ component
	and $\Gamma''$ in the other component of $M$,
	\[
		\ZCY(W;N)(\Gamma) :=
		\ZRT(\Gamma') \cdot \Gamma''
	\]
\item $k=1$: the attaching region of $\cH_1$ is a pair of balls;
	by an isotopy, we may arrange that $\Gamma$ is disjoint
	from the attaching region,
	and regard $\Gamma$ as a graph in $M'$,
	and define
	\[
		\ZCY(W;N)(\Gamma) := \cD^\inv \cdot \Gamma
	\]
\item $k=2$: similar to the $k=1$ case,
	arrange $\Gamma$ to be disjoint from the attaching region.
	The belt region of $\cH_2$ (and its neighborhood)
	defines an embedding of the solid torus
	$\Dk{2} \times \del \Dk{2} \to M'$.
	Let $\gamma = [-\veps,\veps] \times \del \Dk{2}
		\subset \Dk{2} \times \del \Dk{2}$
	be the core of the solid torus, trivially framed,
	and let $\Gamma' = (\gamma, \text{regular})$
	be the colored ribbon graph obtained by applying the regular coloring
	to $\gamma$.
	Then, sending $\Gamma'$ to $M'$ under the embedding,
	\[
		\ZCY(W;N)(\Gamma) = \Gamma \cup \Gamma'
	\]
\item $k=3$: first suppose that $\Gamma$ intersects
	the co-core of $\cH_3$ transversely and in exactly one ribbon,
	and suppose the label of this ribbon is a simple object $i$.
	If $i=\one$, we may ignore this ribbon and isotope
	$\Gamma$ to be disjoint from $\cH_3$;
	then we may define
	\[
		\ZCY(W;N)(\Gamma) = \delta_{i,\one} \cdot \Gamma
	\]
	In general, by isotopy, we may arrange $\Gamma$ to be transverse
	to the co-core of $\cH_3$, and then apply
	\eqnref{e:combine} 
	to get $\Gamma = \sum_j \Gamma_j$,
	where each $\Gamma_j$ satisfies the previous assumptions.
	Then we extend to this case by linearity.
\end{itemize}
\hfill $\triangle$
\end{definition}

\begin{proposition}[\cite{tham-thesis}*{Prop. 5.13}]
\label{p:zcy-indep}
Let $W : M \to_N M'$ be a cornered cobordism,
and let $W = W_l \circ \cdots \circ W_1$
be a decomposition of $W$ into a composition of
elementary cornered cobordisms.
We define
\[
\ZCY(W;N) = \ZCY(W_l;N) \circ \cdots \circ \ZCY(W_1;N)
\]

Then $\ZCY(W;N)$ is independent of handle decomposition of $W$.
\end{proposition}

The proof boils down to checking that
modifying the decomposition (into elementary cobordisms)
by handle slides and handle pair creation/annihilation
does not affect $\ZCY(W;N)$,
and is a simple exercise in skein theory.
Note that in \cite{tham-thesis} we work in PL topology,
but since we work in low dimensions,
it is the same.

\subsection{Solid tori}

Recall that we let $X = \Dk{2} \times S^1$ denote the solid torus,
Let $C$ denote the core circle $\{0\} \times S^1$ in $X$,
framed trivially with respect to $X$ as a product
(i.e. the normal vector is a fixed vector $v$
in every slice $\Dk{2} \times \{\theta\}$),
and fixed the orientation on $C$ agreeing with $\vec{l}$.

The skein module of $X$ with empty boundary
condition has basis $\{\vphi_i\}_{i \in \Irr(\cA)}$,
where $\vphi_i$ is the skein with underlying graph being $C$
and labeled by the simple object $X_i$
(this is stated more generally in \lemref{l:V-ZCY-Xg};
see e.g. \cite{tham-thesis}*{Eqn. 9.15}).
Then $\ZCY(X)$ can be identified with
the Verlinde algebra:
\begin{equation}
\label{e:verlinde-solid-torus}
\ZCY(X) \simeq V \;\;,\;\;
\vphi_i \mapsto x_i
\end{equation}

Then, as mentioned at the beginning of
\secref{s:topology-examples-torus},
the ``flipping a donut'' automorphism translate into
charge conjugation $i \mapsto i^*$
in the Verlinde algebra.

\begin{example}
\label{x:Psi-S-matrix}
The cornered cobordisms $\Psi,\ov{\Psi}$
implement the $S$,$\ov{S}$-matrices on $\ZCY(X)$,
i.e.
\begin{align*}
\ZCY(\Psi)(\vphi_j)
&= s(\vphi_j)
= \sum_i s_{ij} \vphi_i
\\
\ZCY(\ov{\Psi})(\vphi_j)
&= \ov{s}(\vphi_j)
= \sum_i \ov{s}_{ij} \vphi_i
\end{align*}
This follows from a direct computation as in
\figref{f:solid-torus-S-matrix-skein} below.
\hfill $\triangle$
\end{example}

\begin{figure}[h] 
\centering
\begin{tikzpicture}
\node at (-2.1,0) {$\varphi_j =$};
\draw (0,0) ellipse (1.5cm and 1cm);
\draw (-0.2,0.1) to[out=-60,in=-120] (0.2,0.1);
\draw (-0.1,0) to[out=60,in=120] (0.1,0);
\draw[midarrow={0.8}] (0,0) ellipse (0.8cm and 0.5cm);
\node at (0.3,-0.6) {\tiny $j$};
\node at (1.3,-0.9) {$X$};
\draw[|->] (1.8,0) -- (3.2,0);
\node at (2.5,0.3) {\small 2-handle};
\begin{scope}[shift={(5,0)}]
\draw (0,0) ellipse (1.5cm and 1cm);
\draw (-0.2,0.1) to[out=-60,in=-120] (0.2,0.1);
\draw (-0.1,0) to[out=60,in=120] (0.1,0);
\draw (0,0) ellipse (0.6cm and 0.4cm);
\draw[thin_overline={2},dashed] (-0.6,0) circle (0.2cm);
\draw[thin_overline={1}] (-0.6,0) arc (180:360:0.6cm and 0.4cm);
\draw[midarrow={0.8}] (0,0) ellipse (1.2cm and 0.8cm);
\node at (0.1,-0.65) {\tiny $j$};
\node at (1.3,-0.9) {$X'$};
\draw[line width=0.1mm] (-1.9,0.5) -- (-1.3,0.9) -- (-0.2,0.38);
\end{scope}
\draw[|->] (6.8,0) -- (8.2,0);
\node at (7.65,0.7) {$m \mapsto l$};
\node at (7.8,0.3) {$l \mapsto -m$};
\node at (6.8,0.5) {$\Psi:$};
\begin{scope}[shift={(10,0)}]
\draw (0,0) ellipse (1.5cm and 1cm);
\draw (-0.2,0.1) to[out=-60,in=-120] (0.2,0.1);
\draw (-0.1,0) to[out=60,in=120] (0.1,0);
\draw[dashed] (0,0) ellipse (0.8cm and 0.5cm);
\draw[thin_overline={1}] (-0.8,0) circle (0.2cm);
\draw[dashed,thin_overline={2}] (-0.8,0) arc (180:360:0.8cm and 0.5cm);
\node at (-1.2,0) {\tiny $j$};
\draw[->] (-1,0.05) -- (-1,-0.05);
\node at (1.3,-0.9) {$X$};
\end{scope}
\node at (12.5,0) {$= \sum s_{ij} \varphi_i$};
\end{tikzpicture}
\caption{Crane-Yetter applied to $\Psi$
results in the $S$-matrix on $V$;
computation for $\ov{S}$-matrix is similar.
}
\label{f:solid-torus-S-matrix-skein}
\end{figure}
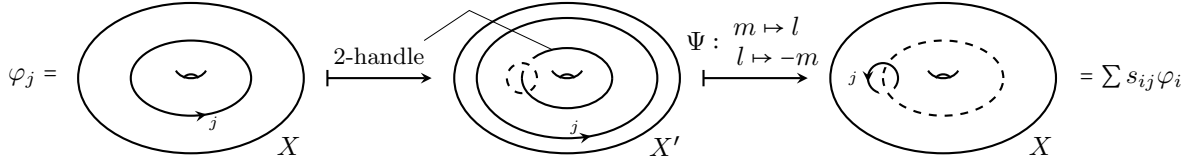

\begin{example}
\label{x:cy-Yprod-2}
Consider the Y-product $\YY_X^{(2)}$ on $X$
as described in \xmpref{x:solid-torus-2},
where we worked out a half-handle decomposition.
This is a relative cobordism
\[
(\YY_X^{(2)}, \YY_{\del X}^{(2)}) :
(X \sqcup X, \del X \sqcup \del X) \to (X, \del X)
\]

In order to apply Crane-Yetter theory,
we interpret this relative cobordism
as a cornered cobordisms
by taking the outgoing corner $\del_+^2 = \del X$
as the corner:
\[
(\YY_X^{(2)}; \del X):
\YY_{\del X}^{(2)} \cup (X \sqcup X) \to X
\]

By \lemref{l:negative-half-handle-identity},
this cornered cobordism is in fact the identity cobordism,
since only half-handles
of the negative type are used to build $\YY_X^{(2)}$.

Let us put the empty skein in the
vertical boundary $\YY_{\del X}^{(2)}$,
$\vphi_i$ in one $X$, and $\vphi_j$ in the other $X$.
It is easy to see that under
$\YY_{\del X}^{(2)} \cup (X \sqcup X) \simeq X$,
the two skeins $\vphi_i, \vphi_j$ simply stack
on top of each other
(see \figref{f:solid-torus-Yprod-2-CY});
as skeins, this is equal to the linear combination
$\sum_k N_{ij}^k \vphi_k$
(use \eqnref{e:combine-alt}).


\begin{figure}[h] 
\centering
\begin{tikzpicture}
\node at (-3.6,0) {$Z_{CY}(\mathbf{Y}_X^{(2)})(\varphi_i \otimes \varphi_j) =$};
\draw (0,0) ellipse (1.5cm and 1cm);
\draw (-0.2,0.1) to[out=-60,in=-120] (0.2,0.1);
\draw (-0.1,0) to[out=60,in=120] (0.1,0);
\draw[midarrow={0.7}] (0,-0.2) ellipse (0.8cm and 0.5cm);
\draw[thin_overline={2},midarrow={0.7}] (0,0.2) ellipse (0.8cm and 0.5cm);
\node at (-0.2,-0.4) {\tiny $i$};
\node at (-0.2,-0.82) {\tiny $j$};
\node at (1.8,0) {$=$};
\begin{scope}[shift={(3.6,0)}]
\draw (0,0) ellipse (1.5cm and 1cm);
\draw (-0.2,0.1) to[out=-60,in=-120] (0.2,0.1);
\draw (-0.1,0) to[out=60,in=120] (0.1,0);
\draw (-1.1,0) arc (180:40:1.1cm and 0.7cm);
\draw[midarrow={0.5}] (-1.1,0) arc (180:320:1.1cm and 0.7cm);
\draw (-0.8,0) arc (180:40:0.8cm and 0.5cm);
\draw[midarrow={0.4}] (-0.8,0) arc (180:320:0.8cm and 0.5cm);
\node at (-0.2,-0.36) {\tiny $j$};
\node at (-0.2,-0.82) {\tiny $i$};
\node[dotnode] (a1) at (0.9,0.2) {};
\node[dotnode] (a2) at (0.9,-0.2) {};
\node at (1.1,0.2) {\tiny $\alpha$};
\node at (1.1,-0.2) {\tiny $\alpha$};
\draw (a1) to[out=90,in=-30] (0.83,0.46);
\draw (a1) to[out=150,in=-30] (0.6,0.33);
\draw (a2) to[out=-90,in=30] (0.83,-0.46);
\draw (a2) to[out=-150,in=30] (0.6,-0.33);
\draw[dotted] (a1) to[out=-60,in=60] (a2);
\end{scope}
\node at (5.4,0) {$=$};
\begin{scope}[shift={(7.2,0)}]
\draw (0,0) ellipse (1.5cm and 1cm);
\draw (-0.2,0.1) to[out=-60,in=-120] (0.2,0.1);
\draw (-0.1,0) to[out=60,in=120] (0.1,0);
\draw[dashed] (0.8,0) arc (0:140:0.8cm and 0.5cm);
\draw[dashed] (0.8,0) arc (0:-140:0.8cm and 0.5cm);
\draw[white, line width=0.8mm] (0.78,0) -- (0.82,0);
\node[dotnode] (a3) at (-0.6,0.33) {};
\node[dotnode] (a4) at (-0.6,-0.33) {};
\node at (-0.74,0.4) {\tiny $\alpha$};
\node at (-0.74,-0.4) {\tiny $\alpha$};
\draw[midarrow] (a3) to[out=-150,in=150] (a4);
\draw[midarrow] (a3) to[out=-70,in=70] (a4);
\node at (-0.9,0) {\tiny $i$};
\node at (-0.4,0) {\tiny $j$};
\end{scope}
\node at (9.7,0) {$= \sum N_{ij}^k \varphi_k$};
\end{tikzpicture}
\caption{}
\label{f:solid-torus-Yprod-2-CY}
\end{figure}

Thus, $\YY_X^{(2)}$ gives the fusion product
on the Verlinde algebra:
\begin{equation}
\ZCY(\YY_X^{(2)})(\vphi_i \tnsr \vphi_j) =
\sum_k N_{ij}^k \vphi_k
\end{equation}
and so \eqnref{e:verlinde-solid-torus}
is an isomorphism of algebras.
\hfill $\triangle$
\end{example}

\begin{example}
\label{x:cy-Yprod-1}
Now consider the Y-product $\YY_X^{(1)}$ from
\xmpref{x:solid-torus-1}.
The (half-)handles of $\YY_X^{(1)}$ are
of index $1^-, 2^-, 3$ in that order.
Since the $3$-handle does not affect the vertical boundary,
we see that $\YY_X^{(1)}$, treated again as a
cornered cobordism with $\del_+^2$ as corner,
is the identity cornered cobordism
(from the $1^-$-,$2^-$-half-handles)
composed with a $3$-handle:
\[
(\YY_X^{(1)}; \del X):
\YY_{\del X}^{(1)} \cup (X \sqcup X) \xrightarrow{3\text{-handle}} X
\]
\begin{figure}[h] 
\centering
\input{diagram-Y1-skein.tikz}
\caption{}
\label{f:solid-torus-Yprod-1-CY}
\end{figure}

From the computation in \figref{f:solid-torus-Yprod-1-CY},
we see that $\YY_X^{(1)}$ gives the convolution product
on the Verlinde algebra
\begin{equation}
\ZCY(\YY_X^{(1)})(\vphi_i \tnsr \vphi_j) =
1/d_i \cdot \delta_{ij} \cdot \vphi_i
\end{equation}
\hfill $\triangle$
\end{example}

Finally, we put everything together
to give an alternative proof of the Verlinde formula
based on Crane-Yetter theory:

\begin{theorem}[Verlinde Formula]
\label{t:main}
Under an identification of the Verlinde algebra $V$
with the skein module $\ZCY(X)$,
the fusion product, convolution product,
and $S$,$\ov{S}$-matrix on $V$
are the results of evaluating the Crane-Yetter functor
on the relative cobordisms
$\YY_X^{(2)},\YY_X^{(1)}: X \sqcup X \to X$,
and $\Psi,\ov{\Psi}: X \to X$, respectively.

Thus, since the Y-products are related by
$\Psi$ (resp. $\ov{\Psi}$),
the $\ov{S}$-matrix (resp. $S$-matrix)
is an algebra morphism from
$V_g$ with the fusion product to
$V_g$ with the convolution product, i.e.

\begin{align}
\label{e:main}
\begin{split}
s(x) * s(y) &= s(xy)
\\
\ov{s}(x) * \ov{s}(y) &= \ov{s}(xy)
\end{split}
\end{align}
for $x,y \in V$.

\end{theorem}

\begin{proof}
Examples \ref{x:cy-Yprod-1}, \ref{x:cy-Yprod-2}
show that under the identification $x_i \mapsto \vphi_i$,
the Y-products $\YY_X^{(2)}, \YY_X^{(1)}$
give the fusion and convolution products, respectively.
\xmpref{x:Psi-S-matrix} shows that $\Psi$
gives the $S$-matrix operator.

\prpref{p:Psi-Yprod-equiv}
shows that the Y-products are related by $\Psi,\ov{\Psi}$;
by evaluating under Crane-Yetter, we get
\begin{align*}
\ZCY(\YY_X^{(1)}) \circ (\ZCY(\ov{\Psi}) \tnsr \ZCY(\ov{\Psi}))
&=
\ZCY(\ov{\Psi}) \circ \ZCY(\YY_X^{(2)})
\\
\ZCY(\YY_X^{(1)}) \circ (\ZCY(\Psi) \tnsr \ZCY(\Psi))
&=
\ZCY(\Psi) \circ \ZCY(\YY_X^{(2)}) \circ P
=
\ZCY(\Psi) \circ \ZCY(\YY_X^{(2)})
\end{align*}
since both products are commutative,
and we have \eqnref{e:main};
here $P$ is the swapping of factors for tensor product
of vector spaces.

\end{proof}

For the other way, that is,
$S,\ov{S}$-matrices as algebra maps
from convolution to fusion, we have:

\begin{theorem}
\label{t:verlinde-formula-reverse}
Let $J \subset \Irr(\cA)$ be the subset of
transparent simple objects,
and let $\vphi_J = \sum_{i \in J} d_i \vphi_i$.
Then 
\begin{align}
s((\cD \vphi_J \cdot x) * y)
= s(x * (\cD \vphi_J \cdot y))
= s(x) \cdot s(y)
\\
\ov{s}((\cD \vphi_J \cdot x) * y)
= \ov{s}(x * (\cD \vphi_J \cdot y))
= \ov{s}(x) \cdot \ov{s}(y)
.
\end{align}
\end{theorem}

\begin{proof}
Similar to \thmref{t:main},
this follows from \prpref{p:Psi-Yprod-equiv-rev}
and computing $\ZCY(K)$.

It follows from \defref{d:zcy-elementary} that,
given $\vphi_i$, its image under $\ZCY(K)$,
is obtained by adding $\Xi$ to $\vphi_i$,
where $\Xi$ is the skein consisting
of the two 2-handle attaching circles that define $K$
(in \figref{f:solid-torus-K}),
both labeled with $\Omega$.
In other words,
$\Xi = s(\sum_{i \in \Irr(\cA)} d_i \vphi_i)$,
and $\ZCY(K)(\vphi_i) = \vphi_i \cdot \Xi$.

By \eqnref{e:charge-conservation},
when $\cA$ is modular, we have
$\Xi = \cD \cdot \vphi_{\one} = \cD \cdot \emptyskein{X}$;
more generally, by \eqnref{e:charge-conservation-premodular},
for $\cA$ premodular, we have
$\Xi = \cD \cdot (\sum_{i \in J} d_i \vphi_i) = \cD \vphi_J$.
\end{proof}

\begin{remark}
We may also prove \thmref{t:verlinde-formula-reverse}
for modular $\cA$
by appealing to the more general result that,
as mentioned at the beginning of this section
(in the discussion relating Crane-Yetter to Reshetikhin-Turaev),
$\ZRT$ and $\ZCY$ should only differ by a factor
$\kappa^{\sigma} \cD^{\chi / 2}$
which only depends on $\cA$
and on the signature $\sigma$ and Euler characteristic $\chi$
of the 4-manifold in question.
Here, as a 4-manifold, $K$ is $S^2 \times \Dk{2}$,
which has $\sigma = 0, \chi = 2$,
while the identity cobordism on $X$, as a 4-manifold,
is $S^1 \times \Dk{2} \times \Dk{2}$
and has $\sigma = 0, \chi = 0$.
This contributes a factor of $\cD^{\chi / 2} = \cD$
to $\ZCY(K)$ as compared to $\ZCY(\id_X)$,
which agrees with out computation above.
\end{remark}

\subsection{Y-coproducts, units, counits}

\begin{example}
\label{x:torus-frobenius-1}
[Y-coproduct etc for $\YY_X^{(1)}$,
continuation of \xmpref{x:cy-Yprod-1}]
To compute $\coYY_X^{(1)}$,
we use the half-handle decomposition
dual to \figref{f:solid-torus-Yprod-1}.

We start with $\vphi_i$.
The dual to the 3-handle is a 1-handle,
which essentially does not affect $\vphi_i$
except give a factor of $1/\cD$.
Next, the dual to the $2^-$-half-handle
is a $1^+$-half-handle,
which we present as a $1^-$-half-handle
followed by a $2$-handle.
The $1^-$-half-handle only changes the boundary,
and the $2$-handle's belt circle is an unknot
that is contractible, so it gives a factor of $\cD$.
Finally, the dual to the $1^-$-half-handle
is a $2^+$-half-handle,
which we present as a $2^-$-half-handle
followed by a $3$-handle.
Again, the $2^-$-half-handle only changes the boundary,
and the $3$-handle cuts $\vphi_i$ twice.
The overall composition gives
\[
\ZCY(\coYY_X^{(1)})(\vphi_i) =
\frac{1}{d_i} \cdot \vphi_i \tnsr \vphi_i
\]

The counit, as given in \figref{f:solid-torus-counit-1},
when evaluated under $\ZCY$,
simply amounts to embedding $X$ in $S^3$ in a standard way
($\vec{l}$ must give 0-framing),
and evaluating the skein under $\ZRT$.
In particular, $\vphi_i \mapsto d_i$.
It is clear that this is a counit for $\ZCY(\coYY_X^{(1)})$,
as expected.

By performing a similar analysis to the
dual of the counit,
we see that the main interesting step for the unit
is the attachment of a $2$-handle
(part of the $1^+$-half-handle that
``punches out the donut hole''),
which gives us the skein consisting of
the core circle labeled by $\Omega$,
i.e. $\sum d_i \vphi_i$,
which is clearly the unit for the convolution product.

To summarize, we have the following structures,
which form a Frobenius algebra
(see discussion around \eqnref{e:frobenius-topology},
or check by direction computation):
\begin{align}
\label{e:frob-1}
\begin{split}
\vphi_i \tnsr \vphi_j &\mapsto
\frac{\delta_{ij}}{d_i} \vphi_i
\\
1 &\mapsto \sum d_i \vphi_i
\\
\vphi_i &\mapsto \frac{1}{d_i} \vphi_i \tnsr \vphi_i
\\
\vphi_i &\mapsto d_i
\end{split}
\end{align}
\hfill $\triangle$
\end{example}

\begin{example}
\label{x:torus-frobenius-2}
Similar to \xmpref{x:torus-frobenius-1},
we use the half-handle decomposition dual to
\figref{f:solid-torus-Yprod-2}.

We start with $\vphi_i$.
The dual to the $2^-$-half-handle is a $1^+$-half-handle,
which effectively punches out a hole in the side of the donut,
while adding a $\Omega$-labeled circle around the hole.
Enlarge the hole so that our 3-manifold resembles
the second diagram in \figref{f:solid-torus-Yprod-2}.
The dual to the $1^-$-half-handle is a $2^+$-half-handle,
which effectively cuts of the $1^-$-half-handle that
connects the two solid tori.
The result is $\sum_k \vphi_i \vphi_k \tnsr \vphi_{k^*}$,
or if we chose to push $\vphi_i$ to the bottom solid torus,
$\sum_k \vphi_k \tnsr \vphi_{k^*} \vphi_i$,
(the $d_k$ factor from $\Omega$
is canceled by the $\alpha$ that arises from the
$2^+$-half-handle, as we saw in \xmpref{x:cy-Yprod-1}).

The counit, as given in \figref{f:solid-torus-counit-2},
when evaluated under $\ZCY$,
is given by $\vphi_i \mapsto \delta_{i0}$:
the $2^+$-half-handle denies any $\vphi_i$
except for $i=\one$.
It is clear from the dual to this half-handle decomposition
that the unit is given by $\vphi_0$.

To summarize, we have the following structures,
which form a Frobenius algebra
(see discussion around \eqnref{e:frobenius-topology},
or check by direction computation):
\begin{align}
\label{e:frob-2}
\begin{split}
\vphi_i \tnsr \vphi_j &\mapsto
\sum_k N_{ij}^k \vphi_k
\\
1 &\mapsto \vphi_0
\\
\vphi_i &\mapsto \sum_k \vphi_i \vphi_k \tnsr \vphi_{k^*}
= \sum_k \vphi_k \tnsr \vphi_{k^*} \vphi_i
\\
\vphi_i &\mapsto \delta_{i0}
\end{split}
\end{align}
\hfill $\triangle$
\end{example}

In \cite{Moore1989}, they give a generalization
of the Verlinde formula to the
``$n$-point function characters at genus $g$''
(here $S_{ij} = s_{ij}/\cD^{1/2}$):
\begin{equation}
\label{e:verlinde-genus-g}
\dim V(g,i_1,\ldots,i_n) =
\sum_p \frac{S_{i_1 p} \cdots S_{i_n p}}{S_{0p} \cdots S_{0p}}
(\frac{1}{S_{0p}})^{2g - 2}
\end{equation}

We may arrive at \eqnref{e:verlinde-genus-g} as follows.
Consider the following computation,
each diagram representing a cobordism
$X^{\sqcup n} \to \emptyset$
(see \eqnref{e:K-Psi-Y1-Y2}),
here shown for $n=3$ inputs and $g=1$ loop
(the unlabeled dots mean $\Psi$,
and the numbered dots refer to the Y-(co)product or counit):
\begin{equation}
\label{e:verlinde-higher-genus}
\begin{tikzpicture}
\begin{scope}[shift={(0,0)}]
\node[dotnode] (a1) at (0,0.5) {};
\node[dotnode] (a2) at (0,0) {};
\node[dotnode] (a3) at (0,-0.5) {};
\node[dotnode] (a4) at (0,-1) {};
\draw (a1) -- (a2);
\draw (a2) .. controls +(-150:0.3cm) and +(150:0.3cm) .. (a3);
\draw (a2) .. controls +(-30:0.3cm) and +(30:0.3cm) .. (a3);
\draw (a3) -- (a4);
\draw (a1) -- (-0.5,1);
\draw (a1) -- (0,1);
\draw (a1) -- (0.5,1);
\node[dotnode] at (-0.21,-0.25) {}; 
\node at (0.15,0.45) {\tiny 1};
\node at (0.15,0.05) {\tiny 1};
\node at (0.15,-0.55) {\tiny 1};
\node at (0.15,-1) {\tiny 1};
\node at (-0.35,-0.25) {\tiny $K$};
\node[dotnode] at (-0.25,0.75) {};
\node[dotnode] at (0,0.75) {};
\node[dotnode] at (0.25,0.75) {};
\end{scope}
\node at (1,0) {$=$};
\begin{scope}[shift={(2,0)}]
\node[dotnode] (a1) at (0,0.5) {};
\node[dotnode] (a2) at (0,0) {};
\node[dotnode] (a3) at (0,-0.5) {};
\node[dotnode] (a4) at (0,-1) {};
\draw (a1) -- (a2);
\draw (a2) .. controls +(-150:0.3cm) and +(150:0.3cm) .. (a3);
\draw (a2) .. controls +(-30:0.3cm) and +(30:0.3cm) .. (a3);
\draw (a3) -- (a4);
\draw (a1) -- (-0.5,1);
\draw (a1) -- (0,1);
\draw (a1) -- (0.5,1);
\node[dotnode] at (-0.21,-0.25) {}; 
\node at (0.15,0.45) {\tiny 2};
\node at (0.15,0.05) {\tiny 1};
\node at (0.15,-0.55) {\tiny 1};
\node at (0.15,-1) {\tiny 1};
\node at (-0.35,-0.25) {\tiny $K$};
\node[dotnode] at (0,0.25) {};
\end{scope}
\node at (3,0) {$=$};
\begin{scope}[shift={(4,0)}]
\node[dotnode] (a1) at (0,0.5) {};
\node[dotnode] (a2) at (0,0) {};
\node[dotnode] (a3) at (0,-0.5) {};
\node[dotnode] (a4) at (0,-1) {};
\draw (a1) -- (a2);
\draw (a2) .. controls +(-150:0.3cm) and +(150:0.3cm) .. (a3);
\draw (a2) .. controls +(-30:0.3cm) and +(30:0.3cm) .. (a3);
\draw (a3) -- (a4);
\draw (a1) -- (-0.5,1);
\draw (a1) -- (0,1);
\draw (a1) -- (0.5,1);
\node[dotnode] at (-0.21,-0.25) {}; 
\node at (0.15,0.45) {\tiny 2};
\node at (0.15,0.05) {\tiny 2};
\node at (0.15,-0.55) {\tiny 1};
\node at (0.15,-1) {\tiny 1};
\node[dotnode] at (0.21,-0.25) {};
\end{scope}
\node at (5,0) {$=$};
\begin{scope}[shift={(6,0)}]
\node[dotnode] (a1) at (0,0.5) {};
\node[dotnode] (a2) at (0,0) {};
\node[dotnode] (a3) at (0,-0.5) {};
\node[dotnode] (a4) at (0,-1) {};
\draw (a1) -- (a2);
\draw (a2) .. controls +(-150:0.3cm) and +(150:0.3cm) .. (a3);
\draw (a2) .. controls +(-30:0.3cm) and +(30:0.3cm) .. (a3);
\draw (a3) -- (a4);
\draw (a1) -- (-0.5,1);
\draw (a1) -- (0,1);
\draw (a1) -- (0.5,1);
\node at (0.15,0.45) {\tiny 2};
\node at (0.15,0.05) {\tiny 2};
\node at (0.15,-0.55) {\tiny 2};
\node at (0.15,-1) {\tiny 1};
\node[dotnode] at (0,-0.75) {};
\end{scope}
\node at (7,0) {$\approx$};
\begin{scope}[shift={(8,0)}]
\node[dotnode] (a1) at (0,0.5) {};
\node[dotnode] (a2) at (0,0) {};
\node[dotnode] (a3) at (0,-0.5) {};
\node[dotnode] (a4) at (0,-1) {};
\draw (a1) -- (a2);
\draw (a2) .. controls +(-150:0.3cm) and +(150:0.3cm) .. (a3);
\draw (a2) .. controls +(-30:0.3cm) and +(30:0.3cm) .. (a3);
\draw (a3) -- (a4);
\draw (a1) -- (-0.5,1);
\draw (a1) -- (0,1);
\draw (a1) -- (0.5,1);
\node at (0.15,0.45) {\tiny 2};
\node at (0.15,0.05) {\tiny 2};
\node at (0.15,-0.55) {\tiny 2};
\node at (0.15,-1) {\tiny 2};
\end{scope}
\end{tikzpicture}
\end{equation}

The last `$\approx$' is not an equality because
$\coii_1 \circ \Psi \neq \coii_2$,
they differ by an additional 2-handle.
Under Crane-Yetter, the difference
is only a factor of $\cD$:
the $\Omega$ loop around the attaching circle of $\Psi$
is contractible because of $\coii_1$,
thus contributes a factor of $\cD$.

Putting $\vphi_{i_1},\ldots,\vphi_{i_n}$ as input,
the left most diagram gives
\begin{align}
\begin{split}
\vphi_{i_1} \tnsr \cdots \tnsr \vphi_{i_n}
\mapsto
\sum_p \frac{s_{i_1 p} \cdots s_{i_n p}}{d_p^{n-1}}
	\cdot \vphi_p
\mapsto
\sum_p \frac{s_{i_1 p} \cdots s_{i_n p}}{d_p^{n}}
	\cdot \vphi_p \tnsr \vphi_p
\mapsto
\sum_p \cD \cdot \frac{s_{i_1 p} \cdots s_{i_n p}}{d_p^{n+1}}
	\cdot \vphi_p
\mapsto
\cdots
\\
\cdots
\mapsto
\sum_p \cD^g \cdot
	\frac{s_{i_1 p} \cdots s_{i_n p}}{d_p^{n-1+2g}}
	\cdot \vphi_p
\mapsto
\sum_p \cD^g \cdot
	\frac{s_{i_1 p} \cdots s_{i_n p}}{d_p^{n+2g-2}}
\end{split}
\end{align}
which is equal to the right hand side of
\eqnref{e:verlinde-genus-g}.
The right most diagram gives
\begin{align}
\begin{split}
\vphi_{i_1} \tnsr \cdots \tnsr \vphi_{i_n}
\mapsto
\vphi_{i_1} \cdots \vphi_{i_n}
\mapsto
\sum_{k_1} \vphi_{k_1} \tnsr
	\vphi_{k_1^*} \vphi_{i_1} \cdots \vphi_{i_n}
\mapsto
\sum_{k_1} \vphi_{k_1}
	\vphi_{k_1^*} \vphi_{i_1} \cdots \vphi_{i_n}
\mapsto
\cdots
\\
\cdots
\mapsto
\sum_{k_1,\ldots,k_g} \vphi_{k_1} \vphi_{k_1^*}
	\cdots
	\vphi_{k_g} \vphi_{k_g^*}
	\vphi_{i_1} \cdots \vphi_{i_n}
\mapsto
\dim V(g,i_1,\ldots,i_n)
\end{split}
\end{align}

Note that the genus $g$ here has nothing to do with
the genus in $X_g$ in the following section,
and \eqnref{e:verlinde-genus-g} is not the generalization
of the Verlinde formula that we promised in the abstract.

\subsection{Higher genus handlebodies}

It is easy to see that skeins of the form shown
in \figref{f:solid-handlebody-skein-basis}
span $\ZCY(X_g;A)$
(use \eqnref{e:combine} and isotope to
look like that, then use skein relations in a ball
around $\vphi$).

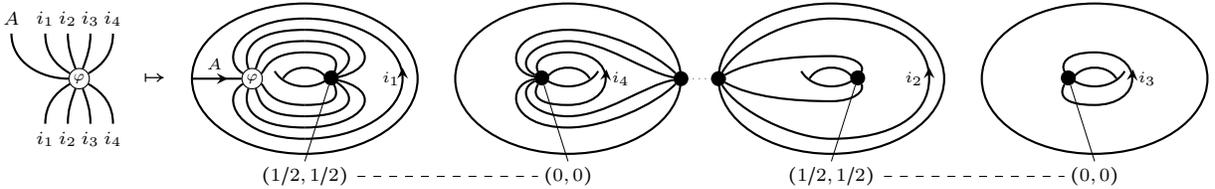
\begin{figure}[h] 
\centering
\begin{tikzpicture}
\begin{scope}[shift={(-3,0)}]
\node[small_morphism,minimum size=5pt] (ph) at (0,0) {\scriptsize $\varphi$};
\draw (ph) to[out=150,in=-90] (-0.45,0.6);
\draw (ph) to[out=-150,in=90] (-0.45,-0.6);
\draw (ph) to[out=110,in=-90] (-0.15,0.6);
\draw (ph) to[out=-110,in=90] (-0.15,-0.6);
\draw (ph) to[out=70,in=-90] (0.15,0.6);
\draw (ph) to[out=-70,in=90] (0.15,-0.6);
\draw (ph) to[out=30,in=-90] (0.45,0.6);
\draw (ph) to[out=-30,in=90] (0.45,-0.6);
\draw (ph) to[out=180,in=-90] (-0.9,0.6);
\node at (-0.9,0.8) {\scriptsize $A$};
\node at (-0.45,0.8) {\scriptsize $i_1$};
\node at (-0.45,-0.8) {\scriptsize $i_1$};
\node at (-0.15,0.8) {\scriptsize $i_2$};
\node at (-0.15,-0.8) {\scriptsize $i_2$};
\node at (0.15,0.8) {\scriptsize $i_3$};
\node at (0.15,-0.8) {\scriptsize $i_3$};
\node at (0.45,0.8) {\scriptsize $i_4$};
\node at (0.45,-0.8) {\scriptsize $i_4$};
\end{scope}
\node at (-2,0) {$\mapsto$};
\draw (0,0) ellipse (1.5cm and 1cm);
\draw (-0.4,0.1) to[out=-60,in=-120] (0.4,0.1);
\draw (-0.29,0) to[out=60,in=120] (0.29,0);
\node[fill=black,circle,inner sep=2pt,outer sep=0pt] (d1) at (0.35,0.01) {};
\node[small_morphism,minimum size=5pt] (p1) at (-0.7,0) {\scriptsize $\varphi$};
\draw[midarrow_rev] (p1) -- (-1.5,0);
\node at (-1.2,0.2) {\tiny $A$};
\node at (0,-1.3) {\scriptsize $(1/2,1/2)$};
\draw[line width=0.1mm] (0,-1.1) -- (d1);
\draw[line width=0.1mm,dashed] (0.7,-1.3) -- (3.1,-1.3);
\draw (p1) .. controls +(150:0.5cm) and +(180:0.9cm) ..
	(0,0.8);
\draw (p1) .. controls +(-150:0.5cm) and +(180:0.9cm) ..
	(0,-0.8);
\draw[midarrow_rev] (0,0.8) arc (90:-90:1.3cm and 0.8cm);
\node at (1.15,0) {\tiny $i_1$};
\draw (p1) .. controls +(110:0.4cm) and +(180:0.7cm) ..
	(0,0.65);
\draw (p1) .. controls +(-110:0.4cm) and +(180:0.7cm) ..
	(0,-0.65);
\draw (0.35,0) .. controls +(10:0.8cm) and +(0:0.8cm) ..
	(0,0.65);
\draw (0.35,0) .. controls +(-10:0.8cm) and +(0:0.8cm) ..
	(0,-0.65);
\draw (p1) .. controls +(70:0.3cm) and +(180:0.6cm) ..
	(0,0.5);
\draw (p1) .. controls +(-70:0.3cm) and +(180:0.6cm) ..
	(0,-0.5);
\draw (0.35,0) .. controls +(30:0.5cm) and +(0:0.6cm) ..
	(0,0.5);
\draw (0.35,0) .. controls +(-30:0.5cm) and +(0:0.6cm) ..
	(0,-0.5);
\draw (p1) .. controls +(30:0.2cm) and +(180:0.5cm) ..
	(0,0.35);
\draw (p1) .. controls +(-30:0.2cm) and +(180:0.5cm) ..
	(0,-0.35);
\draw (0.35,0) .. controls +(60:0.3cm) and +(0:0.4cm) ..
	(0,0.35);
\draw (0.35,0) .. controls +(-60:0.3cm) and +(0:0.4cm) ..
	(0,-0.35);
\begin{scope}[shift={(3.5,0)}]
\draw (0,0) ellipse (1.5cm and 1cm);
\draw (-0.4,0.1) to[out=-60,in=-120] (0.4,0.1);
\draw (-0.29,0) to[out=60,in=120] (0.29,0);
\node[fill=black,circle,inner sep=2pt,outer sep=0pt] (d2) at (-0.35,0.01) {};
\node[fill=black,circle,inner sep=2pt,outer sep=0pt] (d2p) at (1.5,0) {};
\node at (0,-1.3) {\scriptsize $(0,0)$};
\draw[line width=0.1mm] (0,-1.1) -- (d2);
\draw[line width=0.1mm,dotted] (1.5,0) -- (2,0);
\draw[midarrow_rev={0.5}]
	(d2) .. controls +(120:0.2cm) and +(180:0.4cm) ..
	(0,0.35) .. controls +(0:0.3cm) and +(90:0.2cm) ..
	(0.5,0) .. controls +(-90:0.2cm) and +(0:0.3cm) ..
	(0,-0.35) .. controls +(180:0.5cm) and +(-120:0.2cm) ..
	(d2);
\node at (0.7,0) {\tiny $i_4$};
\draw (d2) .. controls +(155:0.4cm) and +(180:0.5cm) ..
	(0,0.5) .. controls +(0:0.5cm) and +(170:0.5cm) ..
	(d2p);
\draw (d2) .. controls +(-155:0.4cm) and +(180:0.5cm) ..
	(0,-0.5) .. controls +(0:0.5cm) and +(-170:0.5cm) ..
	(d2p);
\draw (d2) .. controls +(170:0.5cm) and +(180:0.8cm) ..
	(0,0.65) .. controls +(0:0.5cm) and +(135:0.5cm) ..
	(d2p);
\draw (d2) .. controls +(-170:0.5cm) and +(180:0.8cm) ..
	(0,-0.65) .. controls +(0:0.5cm) and +(-135:0.5cm) ..
	(d2p);
\end{scope}
\begin{scope}[shift={(7,0)}]
\draw (0,0) ellipse (1.5cm and 1cm);
\draw (-0.4,0.1) to[out=-60,in=-120] (0.4,0.1);
\draw (-0.29,0) to[out=60,in=120] (0.29,0);
\node[fill=black,circle,inner sep=2pt,outer sep=0pt] (d3) at (-1.5,0) {};
\node[fill=black,circle,inner sep=2pt,outer sep=0pt] (d3p) at (0.35,0.01) {};
\node at (0,-1.3) {\scriptsize $(1/2,1/2)$};
\draw[line width=0.1mm] (0,-1.1) -- (d3p);
\draw[line width=0.1mm,dashed] (0.7,-1.3) -- (3.1,-1.3);
\draw[midarrow_rev={0.5}]
	(d3) .. controls +(60:0.5cm) and +(180:0.6cm) ..
	(0,0.8) .. controls +(0:0.8cm) and +(90:0.5cm) ..
	(1.3,0) .. controls +(-90:0.5cm) and +(0:0.8cm) ..
	(0,-0.8) .. controls +(180:0.5cm) and +(-60:0.5cm) ..
	(d3);
\node at (1.1,0) {\tiny $i_2$};
\draw (d3) .. controls +(30:0.5cm) and +(180:0.5cm) ..
	(0,0.3) .. controls +(0:0.3cm) and +(60:0.2cm) ..
	(d3p);
\draw (d3) .. controls +(-30:0.5cm) and +(180:0.5cm) ..
	(0,-0.3) .. controls +(0:0.3cm) and +(-60:0.2cm) ..
	(d3p);
\end{scope}
\begin{scope}[shift={(10.5,0)}]
\draw (0,0) ellipse (1.5cm and 1cm);
\draw (-0.4,0.1) to[out=-60,in=-120] (0.4,0.1);
\draw (-0.29,0) to[out=60,in=120] (0.29,0);
\node[fill=black,circle,inner sep=2pt,outer sep=0pt] (d4) at (-0.35,0.01) {};
\node at (0,-1.3) {\scriptsize $(0,0)$};
\draw[line width=0.1mm] (0,-1.1) -- (d4);
\draw[midarrow_rev={0.5}]
	(d4) .. controls +(120:0.2cm) and +(180:0.4cm) ..
	(0,0.35) .. controls +(0:0.3cm) and +(90:0.2cm) ..
	(0.5,0) .. controls +(-90:0.2cm) and +(0:0.3cm) ..
	(0,-0.35) .. controls +(180:0.5cm) and +(-120:0.2cm) ..
	(d4);
\node at (0.7,0) {\tiny $i_3$};
\end{scope}
\end{tikzpicture}
\caption{
We depict for genus 4, but the general case is clear:
At each $X$, going from left to right,
one strand is ``left behind'',
alternately the leftmost and rightmost strand,
so the label on the ``left behind'' strand would be
$i_1, i_g, i_2, i_{g-1}, \ldots, i_{\lfloor g/2 \rfloor + 1}$.
Note the labels $(0,0)$ and $(1/2,1/2)$ indicate
how $X$ is identified with that solid torus.
}
\label{f:solid-handlebody-skein-basis}
\end{figure}


Define the \emph{genus $g$ Verlinde algebra} to be
$V_g = \sum_{i_1,\ldots,i_g} V_{(i_1,\ldots,i_g)}$,
where
$V_{(i_1,\ldots,i_g)} :=
\End_\cA(X_{i_1} \tnsr \cdots \tnsr X_{i_g})$,
and the direct sum is over all $g$-tuples of simple objects.

We define the
\emph{internal genus $g$ Verlinde algebra}
(see \cite{gunningham2019finiteness})
to be the functor
\begin{align}
\begin{split}
V_g^{(-)}: \cA^{op} &\to \textrm{Vec}
\\
A &\mapsto
V_g^{(A)}
:=
\dirsum_{i_1,\ldots,i_g}
V_{(i_1,\ldots,i_g)}^{(A)}
:=
\dirsum_{i_1,\ldots,i_g}
\Hom_{\cA}(A X_{i_1} \cdots X_{i_g},
X_{i_1} \cdots X_{i_g})
\\
f: A \to B &\mapsto
- \circ (f \tnsr \id)
\end{split}
\end{align}

By \figref{f:solid-handlebody-skein-basis},
we have linear maps
$V_{(i_1,\ldots,i_g)}^{(A)} \to \ZCY(X_g;A)$.
Note that the marked point on the boundary,
labeled with object $A\in \cA$,
is placed at a fixed point of $\Psi'$ (and of $\Psi_g'$).

By \lemref{l:V-ZCY-Xg} below,
we may unambiguously denote skeins in $\ZCY(X_g;A)$
by an element in $V_g^{(A)}$.

\begin{lemma}
\label{l:V-ZCY-Xg}
The linear maps
$V_{(i_1,\ldots,i_g)}^{(A)} \to \ZCY(X_g;A)$
are isomorphisms onto their images;
moreover, their images give a direct sum decomposition
of $\ZCY(X_g;A)$.
Thus, we have a natural isomorphism
$V_g^{(-)} \simeq \ZCY(X_g;-): \cA^{op} \to \textrm{Vec}$.
\end{lemma}
\begin{proof}
Follows from standard skein-theoretic techniques.
\end{proof}

We may refine $V_{(i_1,\ldots,i_g)}^{(A)}$
by the simple object that the morphism passes through:
we define $V_{(i_1,\ldots,i_g);k}^{(A)} :=
\mathrm{span} (\Hom(X_k, X_{i_1} \cdots X_{i_g})
\circ \Hom(A X_{i_1} \cdots X_{i_g}, X_k))$,
and we have
$V_{(i_1,\ldots,i_g)}^{(A)} =
\dirsum_k V_{(i_1,\ldots,i_g);k}^{(A)}$

The generalization of the convolution product
for the genus $g$ Verlinde algebra
is given by composition in each component
$V_{(i_1,\ldots,i_g);k}$
(see \defref{d:generalized-convolution-product},
\xmpref{x:cy-handlebody-Yprod-1}),
while the generalization of the fusion product
is not so obvious
(see \defref{d:generalized-fusion-product},
\xmpref{x:cy-handlebody-Yprod-2}).

The operations defined below should be thought of as
natural transformations $V_g^{(-)} \xrightarrow{.} V_g^{(-)}$
or $V_g^{(-)} \tnsr V_g^{(-)}
\xrightarrow{.} V_g^{(- \tnsr^{op} -)}$;
when restricted to $\one$,
we get linear maps $V_g \to V_g$ or
$V_g \tnsr V_g \to V_g$.
Again by \lemref{l:V-ZCY-Xg},
we may regard the operations defined below as
natural transformations
$\ZCY(X_g;-) \xrightarrow{.} \ZCY(X_g;-)$
or
$\ZCY(X_g;-) \tnsr \ZCY(X_g;-)
\xrightarrow{.} \ZCY(X_g;- \tnsr^{op} -)$.

\begin{definition}
\label{d:generalized-S-matrix}
The \emph{generalized $S$-matrix}
and \emph{generalized $\ov{S}$-matrix}
are given by, for
$\vphi \in V_{(i_1,\ldots,i_g)}^{(A)}$,

\begin{equation*}
\label{e:solid-handlebody-S-matrix-skein-result}
s(\vphi) = 
\sum_{k_1,\ldots,k_g}
d_{k_1} \cdots d_{k_g}
\begin{tikzpicture}
\node[dotnode] (ph) at (0,0) {};
\draw (0.5,0) -- (0.5,1);
\draw (1,0) -- (1,1);
\draw[thin_overline={1.5}, midarrow_rev={0.1}] (ph)
	.. controls +(60:0.4cm) and +(90:0.2cm) .. (0.6,0)
	.. controls +(-90:0.2cm) and +(-60:0.4cm) .. (ph);
\draw[thin_overline={1.5}, midarrow_rev={0.1}] (ph)
	.. controls +(120:0.8cm) and +(90:0.8cm) .. (1.1,0)
	.. controls +(-90:0.8cm) and +(-120:0.8cm) .. (ph);
\draw[thin_overline={1.5}, midarrow={0.8}] (0.5,0) -- (0.5,-1);
\draw[thin_overline={1.5}, midarrow={0.8}] (1,0) -- (1,-1);
\draw (ph) .. controls +(180:0.3cm) and +(-90:0.5cm) .. (-0.5,1);
\node at (-0.15,-0.1) {\tiny $\vphi$};
\node at (-0.2,0.4) {\tiny $i_1$};
\node at (0.25,0) {\tiny $i_g$};
\node at (0.3,-0.8) {\tiny $k_1$};
\node at (1.2,-0.8) {\tiny $k_g$};
\node at (0.8,0) {\tiny $\ldots$};
\node at (0.75,-0.8) {\tiny $\ldots$};
\node at (0.75,0.8) {\tiny $\ldots$};
\node at (-0.5,1.1) {\tiny $A$};
\end{tikzpicture}
\in V_{(i_1,\ldots,i_g)}^{(A)}
\;\; ; \;\;
\ov{s}(\vphi) = 
\sum_{k_1,\ldots,k_g}
d_{k_1} \cdots d_{k_g}
\begin{tikzpicture}
\node[dotnode] (ph) at (0,0) {};
\draw[midarrow={0.8}] (0.5,0) -- (0.5,-1);
\draw[midarrow={0.8}] (1,0) -- (1,-1);
\draw[thin_overline={1.5}, midarrow_rev={0.1}] (ph)
	.. controls +(60:0.4cm) and +(90:0.2cm) .. (0.6,0)
	.. controls +(-90:0.2cm) and +(-60:0.4cm) .. (ph);
\draw[thin_overline={1.5}, midarrow_rev={0.1}] (ph)
	.. controls +(120:0.8cm) and +(90:0.8cm) .. (1.1,0)
	.. controls +(-90:0.8cm) and +(-120:0.8cm) .. (ph);
\draw[thin_overline={1.5}] (0.5,0) -- (0.5,1);
\draw[thin_overline={1.5}] (1,0) -- (1,1);
\draw (ph) .. controls +(180:0.3cm) and +(-90:0.5cm) .. (-0.5,1);
\node at (-0.15,-0.1) {\tiny $\vphi$};
\node at (-0.2,0.4) {\tiny $i_1$};
\node at (0.25,0) {\tiny $i_g$};
\node at (0.3,-0.8) {\tiny $k_1$};
\node at (1.2,-0.8) {\tiny $k_g$};
\node at (0.8,0) {\tiny $\ldots$};
\node at (0.75,-0.8) {\tiny $\ldots$};
\node at (0.75,0.8) {\tiny $\ldots$};
\node at (-0.5,1.1) {\tiny $A$};
\end{tikzpicture}
\in V_{(i_1,\ldots,i_g)}^{(A)}
.
\end{equation*}


\hfill $\triangle$
\end{definition}

\begin{definition}
\label{d:generalized-convolution-product}
For $\vphi \in V_{(i_1,\ldots,i_g)}^{(A)},
\vphi' \in V_{(j_1,\ldots,j_g)}^{(B)}$,
we define their
\emph{generalized convolution product}
to be
\begin{equation}
\vphi * \vphi'
= \frac{\delta_{i_1 j_1}}{d_{i_1}} 
\cdot \ldots \cdot
\frac{\delta_{i_g j_g}}{d_{i_g}}
\cdot
\vphi' \circ (\id_B \tnsr \vphi)
\begin{tikzpicture}
\node[dotnode] (ph) at (0,0.3) {};
\node[dotnode] (ph') at (0,-0.3) {};
\draw (-0.3,0.6)
	.. controls +(-90:0.15cm) and +(150:0.15cm) .. (ph)
	.. controls +(-150:0.15cm) and +(90:0.15cm) .. (-0.3,0)
	.. controls +(-90:0.15cm) and +(150:0.15cm) .. (ph')
	.. controls +(-150:0.15cm) and +(90:0.15cm) .. (-0.3,-0.6);
\draw (0.3,0.6)
	.. controls +(-90:0.15cm) and +(30:0.15cm) .. (ph)
	.. controls +(-30:0.15cm) and +(90:0.15cm) .. (0.3,0)
	.. controls +(-90:0.15cm) and +(30:0.15cm) .. (ph')
	.. controls +(-30:0.15cm) and +(90:0.15cm) .. (0.3,-0.6);
\draw (ph)
	.. controls +(180:0.3cm) and +(-90:0.2cm) .. (-0.6,0.6);
\draw (ph')
	.. controls +(180:0.5cm) and +(-90:0.4cm) .. (-0.8,0.6);
\node at (0,0) {\small $\ldots$};
\node at (0,0.5) {\small $\ldots$};
\node at (0,-0.55) {\small $\ldots$};
\node at (-0.6,0.7) {\tiny $A$};
\node at (-0.8,0.7) {\tiny $B$};
\node at (0.3,0.3) {\tiny $\vphi$};
\node at (0.3,-0.28) {\tiny $\vphi'$};
\end{tikzpicture}
\in
V_{(i_1,\ldots,i_g)}^{(BA)}
.
\end{equation}
\hfill $\triangle$
\end{definition}

\begin{definition}
\label{d:generalized-fusion-product}
For $\vphi \in V_{(i_1,\ldots,i_g)}^{(A)},
\vphi' \in V_{(j_1,\ldots,j_g)}^{(B)}$,
we define their
\emph{generalized fusion product}
to be given

\begin{equation}
\vphi \cdot \vphi' :=
\sum_{k_1,\ldots,k_g}
d_{k_1} \cdots d_{k_g}
\label{e:solid-handlebody-Yprod-1-CY-alt}
\begin{tikzpicture}
\node[dotnode] (ph) at (0,0.1) {};
\node[dotnode] (ph') at (0.5, -0.1) {};
\node[dotnode] (a1) at (-0.2,0.5) {};
\node[dotnode] (a1') at (-0.2,-0.5) {};
\node[dotnode] (ag) at (0.7,0.5) {};
\node[dotnode] (ag') at (0.7,-0.5) {};
\draw (ph)
	.. controls +(180:0.3cm) and +(down:0.8cm) .. (-0.8,1);
\draw (ph')
	.. controls +(180:0.8cm) and +(down:1cm) .. (-1.2,1);
\draw (ph') -- (a1);
\draw (ph') -- (ag);
\draw (ph') -- (a1');
\draw (ph) -- (a1);
\draw[thin_overline={1}] (ph) -- (ag);
\draw[thin_overline={1}] (ph) -- (a1');
\draw[thin_overline={1}] (ph) -- (ag');
\draw (ph') -- (ag');
\draw (a1) -- +(up:0.5cm);
\draw (ag) -- +(up:0.5cm);
\draw (a1') -- +(down:0.5cm);
\draw (ag') -- +(down:0.5cm);
\node at (-0.2,1.1) {\tiny $k_1$};
\node at (-0.2,-1.1) {\tiny $k_1$};
\node at (0.7,1.1) {\tiny $k_g$};
\node at (0.7,-1.1) {\tiny $k_g$};
\node at (-0.8,1.1) {\tiny $A$};
\node at (-1.2,1.1) {\tiny $B$};
\node at (0.25,0.9) {$\ldots$};
\node at (0.25,0.5) {$\ldots$};
\node at (0.25,-0.5) {$\ldots$};
\node at (0.25,-0.9) {$\ldots$};
\node at (-0.38,0.5) {\tiny $\alpha_1$};
\node at (-0.38,-0.5) {\tiny $\alpha_1$};
\node at (0.9,0.5) {\tiny $\alpha_g$};
\node at (0.9,-0.5) {\tiny $\alpha_g$};
\node at (0.7,-0.1) {\tiny $\vphi'$};
\node at (-0.2,0.2) {\tiny $\vphi$};
\end{tikzpicture}
\in V^{(BA)}.
\end{equation}


\hfill $\triangle$
\end{definition}

Now we apply the Crane-Yetter functor
to the relative cobordisms
$\YY_{X_g}^{(1)}, \YY_{X_g}^{(2)}, \Psi', \ov{\Psi}'$
from \secref{s:topology-examples-handlebody}.
Since we are considering skeins with boundary values,
we cannot simply put the empty skein
in the vertical boundaries as we did before.
In the examples below,
our choice of skein in the vertical boundary
will be very simple, essentially just the ``identity''.

\begin{example}
\label{x:Psi-S-matrix-handlebody}
Recall that $\Psi_g', \ov{\Psi}_g'$
are designed to be the identity in a neighborhood
of the marked point on the boundary,
so $\ZCY(\Psi_g'), \ZCY(\ov{\Psi}_g')$
are well-defined as linear maps from $\ZCY(X_g;A)$ to itself.
From \figref{f:solid-handlebody-S-matrix-skein},
we can see that,
under Crane-Yetter, the cornered cobordisms
$\Psi_g', \ov{\Psi}_g'$
realize the generalized $S$,$\ov{S}$-matrices on
$\ZCY(X_g;A)$;
that is,
$\ZCY(\Psi_g')(\vphi) = s(\vphi)$,
$\ZCY(\ov{\Psi}_g')(\vphi) = \ov{s}(\vphi)$.

In relation to $\Psi_g,\ov{\Psi}_g$,
recall that the marked point at $p:=(0,0)$
should come with a framing, say tangent to $\vec{l}$;
then $\Psi_g$
defines a linear map from $\ZCY(X_g;(A,p,\vec{l}))$
to $\ZCY(X_g;(A,p,-\vec{m}))$,
and $\ov{\Psi}_g$ defines a linear map from
$\ZCY(X_g;(A,p,\vec{l}))$
to $\ZCY(X_g;(A,p,\vec{m}))$.
Then we may get back to $\ZCY(X_g;(A,p,\vec{l}))$
by applying a diffeomorphism from $X_g$ to itself
that untwists the framing on $p$.
Alternatively,
we may compose with a relative cobordism
which is trivial topologically,
but on the vertical boundary,
we have a skein that untwists the marked point.

\begin{figure}[h] 
\centering
\includegraphics[width=15cm]{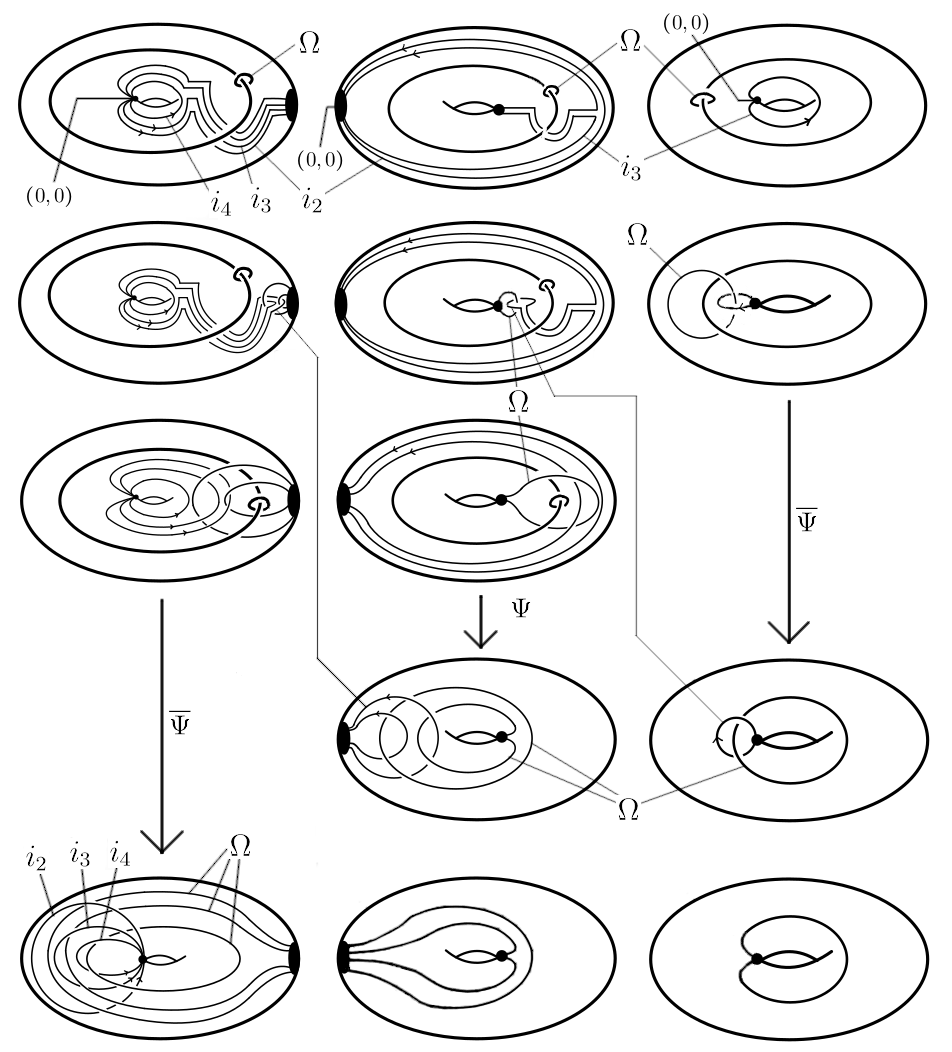}
\caption{
Computation of $\ZCY(\Psi)$ for genus $g=4$;
for odd genus the argument is essentially the same.
The top row represents the skeins after adding the 2-handles.
In the bottom row,
we see that the skeins labeled with $i_*$'s
criss-cross with the skeins labeled with $\Omega$.
This is used inductively
(we proceed from right to left) - for the torus to its left,
we know that when we pull the skeins in,
we get the criss-cross pattern,
as we see in the second row.
}
\label{f:solid-handlebody-S-matrix-skein}
\end{figure}
\hfill $\triangle$
\end{example}

\begin{example}
\label{x:cy-handlebody-Yprod-1}
Similar to \xmpref{x:cy-Yprod-1},
the half-handles that make up $\YY_{X_g}^{(1)}$
are of negative type,
so the most interesting part for skeins is
the 3-handles.
In the last copy of $X$ in $X_g$,
the 3-handle cuts a pair of skeins labeled
$i_k$ and $j_k$, where $k=\lfloor g/2 \rfloor + 1$,
which forces $i_k = j_k$
and gives a factor of $1/d_{i_k}$.
Upon simplifying the skeins,
we see that in the penultimate copy of $X$,
the 3-handle again cuts only a pair of skeins,
labeled $i_{k'}$ and $j_{k'}$ for some $k'$
($= k+1$ or $k-1$ depending on parity of $g$),
which forces $i_{k'} = j_{k'}$
and gives a factor of $1/d_{i_{k'}}$.
Repeat this procedure,
we see that we end up with
$\ZCY(\YY_{X_g}^{(1)})(\vphi \tnsr \vphi')
= \vphi * \vphi'$.
In other words, under Crane-Yetter,
$\YY_{X_g}^{(1)}$ realizes the generalized convolution product.
\hfill $\triangle$
\end{example}

%
%

\begin{example}
\label{x:cy-handlebody-Yprod-2}
Similar to \xmpref{x:cy-Yprod-2},
it is easy to see that $\YY_{X_g}^{(2)}$
stacks one $X_g$ on top of another,
and so $\ZCY(\YY_{X_g}^{(2)})(\vphi \tnsr \vphi')$
simply stacks $\vphi$ on top of $\vphi'$.
Using methods similar to the computation of
$\ZCY(\YY_X^{(2)})$
(see \figref{f:solid-torus-Yprod-2-CY}),
and adding an appropriate identity morphism
in the vertical boundary
to combine marked points into one,
we see that, under Crane-Yetter,
$\YY_{X_g}^{(2)}$ realizes the generalized fusion product:
$\ZCY(\YY_{X_g}^{(2)})(\vphi \tnsr \vphi')
= \vphi \cdot \vphi'$.
\hfill $\triangle$
\end{example}

%

\begin{theorem}
\label{t:main-handlebody}
Under the identification of the
internal genus $g$ Verlinde algebra $V_g^{(-)}$
with the skein module functor $\ZCY(X_g;-)$,
the generalized fusion product,
generalized convolution product,
and generalized $S,\ov{S}$-matrices on $V_g^{(-)}$
are the results of evaluating the Crane-Yetter functor
on the relative cobordisms
$\YY_{X_g}^{(2)},\YY_{X_g}^{(1)}: X_g \sqcup X_g \to X_g$,
and $\Psi_g',\ov{\Psi}_g': X_g \to X_g$, respectively.

Thus, since the Y-products are related by
$\ov{\Psi}_g'$ (resp. $\Psi_g'$),
the generalized $\ov{S}$-matrix (resp. $S$-matrix)
is an algebra morphism from
$V_g$ with the (resp. opposite) generalized fusion product to
$V_g$ with the generalized convolution product.
\end{theorem}

\begin{proof}
Examples \ref{x:cy-handlebody-Yprod-2},
\ref{x:cy-handlebody-Yprod-1},
\ref{x:Psi-S-matrix-handlebody}
prove the first part of the statement.

\prpref{p:Psi-Yprod-equiv-handlebody} shows that
the Y-products are related by $\Psi_g, \ov{\Psi}_g$;
by evaluating under Crane-Yetter
for empty boundary conditions
(i.e. $A = B = \one$), we get
\begin{align*}
\ZCY(\YY_{X_g}^{(1)}) \circ
(\ZCY(\ov{\Psi}_g) \tnsr \ZCY(\ov{\Psi}_g))
&=
\ZCY(\ov{\Psi}_g) \circ \ZCY(\YY_{X_g}^{(2)})
\\
\ZCY(\YY_X^{(1)}) \circ (\ZCY(\Psi_g) \tnsr \ZCY(\Psi_g))
&=
\ZCY(\Psi_g) \circ \ZCY(\YY_X^{(2)}) \circ P
\end{align*}
where here $P$ is the swapping of factors for tensor product
of vector spaces;
thus, for $x, y \in V_g$,
\begin{align*}
s(x) * s(y) &= s(x \cdot y)
\\
\ov{s}(x) * \ov{s}(y) &= \ov{s}(y \cdot x) .
\end{align*}

Now for non-empty boundary condition,
we need to use $\Psi_g', \ov{\Psi}_g'$
instead of $\Psi_g, \ov{\Psi}_g$.
As discussed in \xmpref{x:Psi-S-matrix-handlebody},
the former is obtain from the latter by
untwisting the marked point on the boundary.
It is clear by direct inspection that
for $x \in V_g^{(A)}, y \in V_g^{(B)}$,
\begin{align*}
s(x) * s(y) &= s(x \cdot y)
\\
\ov{s}(x) * \ov{s}(y) &= c_{B,A} \circ \ov{s}(y \cdot x) .
\end{align*}
\end{proof}


\section{Relation to Yetter's Handle as Hopf Algebra}
\label{s:yetter-portrait}

We would like to say a few words about
the relation of our work to
Yetter's work on the handle as a Hopf algebra
\cite{yetter1997portrait}.

In \cite{yetter1997portrait},
Yetter constructs a formal Hopf algebra structure
on the torus with one boundary component
(which we refer to as $H$ for the rest of this discussion).
He considers the category $S^1$,
where the objects are surfaces with boundary identified
with $S^1$,
and morphisms are cornered cobordisms.
The tensor product of two surfaces $N_1 \tnsr N_2$
is $(N_1 \sqcup N_2) \cup POP$ (the pair of pants).
Then the multiplication and comultiplication on $H$
are given by cornered cobordisms
$m: H \tnsr H \to H$ and $\Delta: H \to H \tnsr H$
(see \figref{f:yetter-portrait-fig7}).

\begin{figure}[h] 
\centering
\includegraphics[width=8cm]{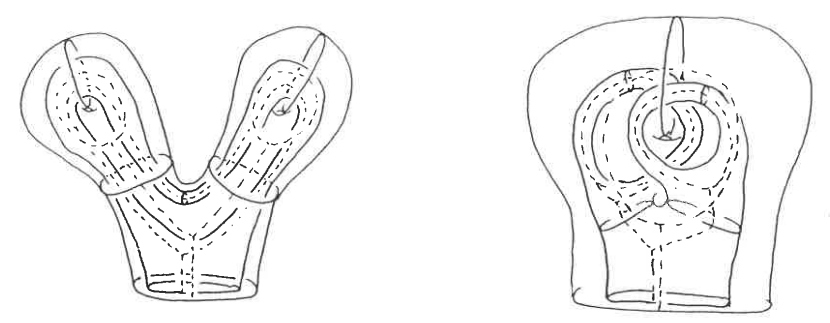}
\caption{Multiplication and comultiplication on $H$,
from \cite{yetter1997portrait}*{Figure 7}.
}
\label{f:yetter-portrait-fig7}
\end{figure}

We would like to describe $H$ as a bialgebra
in terms of the Y-product.
In order to do so, we will need to modify
the notion of Y-product slightly.
Before we describe the variant,
notice that the multiplication $m$
in \figref{f:yetter-portrait-fig7}
is similar to $\YY_X^{(1)}$.
Indeed, it is simply $\YY_H$
considered as a cornered cobordism with corner given by
the outgoing corner.
Likewise, the coproduct on $H$ looks similar to $\YY_X^{(2)}$,
and can be given by $\YY_H$ considered as
a cornered cobordism with incoming corner of $\YY_H$ as corner.

Given $Q = M \cup_N \ov{M}$,
we define the \emph{reduced Y-product},
denoted $\YYbar_Q$,
to be $\YY_Q$ as a cornered cobordism
with corner given by the outgoing corner $\del Q$.
The incoming boundary of $\YYbar_Q$ is
$Q \# Q$,
where for $Q_1,Q_2$ with identifications of their boundaries
with $\del Q$,
$Q_1 \# Q_2 := (Q_1 \sqcup Q_2) \cup \YY_{\del Q}$
(here $\YY_{\del Q}$ is the vertical boundary of $\YY_Q$;
$\#$ would be $\tnsr$ in Yetter's context).

In terms of half-handle decomposition,
suppose we are given a half-handle decomposition of $M$
as a relative cobordism
$(M, N') : \emptyset \to (N,P=\del N)$
such that all half-handles are of negative type
and come before regular handles.
Such half-handle decomposition always exists:
any handle decomposition for the vertical boundary $N'$
as a cobordism $\emptyset \to P$
can be upgraded to a half-handle decomposition for
the relative cobordism
$(N' \times I, N') : \emptyset \to (N', P)$
by changing each $k$-handle to a $k^-$-half-handle,
then any handle decomposition for the cornered cobordism
$(M;P) : N' \to N$ will complete the
desired half-handle decomposition.

We can build, as in \secref{s:half-handle-decomp-Yprod},
a half-handle decomposition for $Q$,
which we may breakdown to a sequence of cobordisms
$\emptyset \to N' \to N \to N' \to \emptyset$.
Obviously, $Q \sqcup Q$ can be given the
half-handle decomposition which is the
concatenation of two copies of that sequence.
By attaching $\YY_{\del Q}$ to $Q \sqcup Q$,
it is not hard to see that the middle $N' \to \emptyset \to N'$
is canceled out, leaving the sequence
$\emptyset \to N' \to N \to N' \to N \to N' \to \emptyset$
as a half-handle decomposition for $\YYbar_Q$.

This operation of attaching $\YY_{\del Q}$ to $Q \sqcup Q$
is a relative cobordism $Q \sqcup Q \to \YYbar_Q$;
it is clear from the construction of the full Y-product
and our choice of half-handle decomposition for $M$
that it is contained in
the full Y-product $\YY_Q: Q \sqcup Q \to Q$.
In other words, we may apply the same construction
for $\YY_Q$ to get a half-handle decomposition
for $\YYbar_Q$,
by simply restricting to adding handles for the latter half
of $M: \emptyset \to N' \to N$.

As a concrete example,
we can build $H$ from the sequence of (half-)handles
$0^-, 1, 1, 1^+$.
The resulting $H \# H$ has (half-)handles
$0^-,1,1,1,1,1^+$.
Then $\YYbar_H$ is given by one 2-handle
that eliminates the middle two 1-handles of $H \# H$.

The \emph{reduced Y-coproduct, unit, and counit},
denoted $\coYYbar_Q, \iibar_Q$, and $\coiibar_Q$
respectively, are also given by similar considerations.
Note that the unit object for $\#$ is not $\emptyset$,
but $N' \times I$,
which has boundary $N' \cup_P \ov{N'}$.
The unit $\ii_Q$ is a cornered cobordism
from $N' \times I$ to $Q$.

The following is a generalization
of the bialgebra structure on $H$,
in particular, it applies to any connected surface
with one boundary component:

\begin{proposition}
\label{p:Q-bialg}
Let $Q$ be a connected $n$-manifold
with boundary a sphere $S^{n-1}$,
and suppose that $Q$ can be presented as the double
of a manifold in two ways, say
$Q = M_1 \cup_{N_1} \ov{M_1} = M_2 \cup_{N_2} \ov{M_2}$,
such that $N_1$ and $N_2$ intersect transversely.
Furthermore, suppose that $M_0 := M_1 \cap M_2$
is diffeomorphic to a ball (after smoothing corners),
and $M_0 \cap \del Q, M_1 \cap \del Q, M_2 \cap Q$
are $(n-1)$-balls (after smoothing corners).
Then the Y-product from $M_1$, $\YYbar_Q^{(1)}$,
and Y-coproduct from $M_2$, $\coYYbar_Q^{(2)}$,
obey the bialgebra relation.
\end{proposition}

\begin{proof}
Denote by $N_1' := M_1 \cap \del Q, N_2' := M_2 \cap \del Q$,
or equivalently they are the
closures of $M_1 \backslash N_1, M_2 \backslash N_2$,
and denote by $N_0' := M_0 \cap \del Q$;
by hypothesis, they must be $(n-1)$-balls.

Denote $P := N_1 \cap N_2$;
it separates $N_1$ into two pieces,
each a mirror copy of the other,
and we denote one by $N_{10}$.
Similarly, $P$ separates $N_2$ into $N_{20}$ and $\ov{N_{20}}$.

We denote by $P_1, P_2$ the closures of
$\del N_{10} \backslash P, \del N_{20} \backslash P$.

We represent $Q$ schematically as follows:
\begin{equation}
\label{e:diagram-bialg-Q}
\begin{tikzpicture}
\draw (-0.7,0) -- (0.7,0);
\draw (0,0.7) -- (0,-0.7);
\draw (0,0.7)
	.. controls +(180:0.6cm) and +(90:0.6cm) .. (-0.7,0);
\draw (0,0.7)
	.. controls +(0:0.6cm) and +(90:0.6cm) .. (0.7,0);
\draw (0,-0.7)
	.. controls +(180:0.6cm) and +(-90:0.6cm) .. (-0.7,0);
\draw (0,-0.7)
	.. controls +(0:0.6cm) and +(-90:0.6cm) .. (0.7,0);
\node at (-0.35,0.35) {\tiny $\overline{M_0}$};
\node at (0.35,0.35) {\tiny $M_0$};
\node at (-0.35,-0.35) {\tiny $M_0$};
\node at (0.38,-0.35) {\tiny $\overline{M_0}$};
\draw[decorate,decoration={brace,mirror,amplitude=2pt}]
	(-0.7,-0.8) -- (-0.1,-0.8);
\draw[decorate,decoration={brace,mirror,amplitude=2pt}]
	(0.1,-0.8) -- (0.7,-0.8);
\node at (0.4,-1.1) {\tiny $\overline{M_1}$};
\node at (-0.4,-1.1) {\tiny $M_1$};
\draw[decorate,decoration={brace,amplitude=2pt}]
	(-0.8,0.1) -- (-0.8,0.7);
\draw[decorate,decoration={brace,amplitude=2pt}]
	(-0.8,-0.7) -- (-0.8,-0.1);
\node at (-1.1,-0.4) {\tiny $M_2$};
\node at (-1.1,0.4) {\tiny $\overline{M_2}$};
\draw[line width=0.1mm]
	(0.05,-1.25) -- (0.05,-0.6) -- (0,-0.55);
\node at (0.05,-1.4) {\tiny $N_{10}$};
\draw[line width=0.1mm]
	(0.05,0.9) -- (0.05,0.6) -- (0,0.55);
\node at (0.05,1.1) {\tiny $\overline{N_{10}}$};
\draw[line width=0.1mm]
	(-0.9,-0.05) -- (-0.6,-0.05) -- (-0.55,0);
\node at (-1.2,0) {\tiny $N_{20}$};
\draw[line width=0.1mm]
	(0.9,-0.05) -- (0.6,-0.05) -- (0.55,0);
\node at (1.2,0) {\tiny $\overline{N_{20}}$};
\node[dotnode] at (0,0) {};
\draw[line width=0.1mm]
	(0,0) -- (1,0.4);
\node at (1.1,0.4) {\tiny $P$};
\end{tikzpicture}
\;\;\;
;
\;\;\;
\begin{tikzpicture}
\begin{scope}[shift={(0.35,0.35)}]
\draw (-0.7,0) -- (0,0) -- (0,-0.7);
\draw (-0.7,0) .. controls +(-90:0.6cm) and +(180:0.6cm) ..
	(0,-0.7);
\draw (-0.53,-0.53) -- (-0.63,-0.63);
\node at (-0.35,-0.35) {\tiny $M_0$};
\node at (-1,-0.3) {\tiny $\mathcal{D}^{n-1}$};
\node at (-0.3,-0.95) {\tiny $\mathcal{D}^{n-1}$};
\node at (0.1,0.1) {\tiny $P$};
\node at (-0.35,0.15) {\tiny $N_{20}$};
\node at (0.25,-0.35) {\tiny $N_{10}$};
\node at (-0.85,0.1) {\tiny $P_2$};
\node at (0.15,-0.8) {\tiny $P_1$};
\node at (-0.7,-0.7) {\tiny $P_0$};
\end{scope}
\end{tikzpicture}
\begin{tikzpicture}
\draw (0,-0.6) -- (0,0.6);
\draw (0,-0.6) .. controls +(-120:0.5cm) and +(-120:0.5cm) ..
	(0,0.6);
\draw (0,0.6) arc (90:180:0.6);
\draw[line width=0.2mm] (-0.6,0) arc (180:270:0.6);
\draw (-0.6,0) .. controls +(-90:0.35cm) and +(170:0.3cm) ..
	(-0.1,-0.7);
\draw[line width=0.1mm] (0,0.6)
	.. controls +(-160:0.5cm) and +(90:0.1cm) ..
	(-0.45,-0.25)
	.. controls +(-90:0.1cm) and +(150:0.1cm) ..
	(-0.3,-0.62)
	.. controls +(-30:0.1cm) and +(-150:0.1cm) ..
	(0,-0.6);
\node at (0.3,0) {\tiny $P$};
\draw[line width=0.1mm] (0,0) -- (0.2,0);
\node at (0.45,-0.45) {\tiny $N_{10}$};
\draw[line width=0.1mm] (-0.1,-0.3) -- (0.2,-0.4);
\node at (0.35,0.45) {\tiny $N_{20}$};
\draw[line width=0.1mm] (-0.05,0.2) -- (0.1,0.4);
\node at (-0.8,-0.6) {\tiny $\mathcal{D}^{n-1}$};
\draw[line width=0.1mm] (-0.6,-0.7) -- (-0.4,-0.7) -- (-0.3,-0.3);
\draw[line width=0.1mm] (-0.9,-0.5) -- (-0.7,0) -- (-0.5,0);
\node at (-0.8,0.6) {\tiny $P_2$};
\draw[line width=0.1mm] (-0.7,0.5) -- (-0.51,0.3);
\node at (-0.2,0.8) {\tiny $P_1$};
\draw[line width=0.1mm] (-0.2,0.7) -- (-0.2,0.35) -- (-0.11,0.3);
\node at (-0.5,0.7) {\tiny $P_0$};
\draw[line width=0.1mm] (-0.5,0.6) -- (-0.35,0.3);
\end{tikzpicture}
\;\;\;
;
\;\;\;
\begin{tikzpicture}
\begin{scope}[shift={(0,-0.5)}]
\draw (0,1) -- (0,0) -- (1,0);
\draw (1,0) arc (0:90:1);
\begin{scope}[line width=0.2mm]
\draw (0.25,0) to[out=90,in=-112.5] (0.383,0.924);
\draw (0.5,0) to[out=90,in=-135] (0.707,0.707);
\draw (0.75,0) to[out=90,in=-157.5] (0.924,0.383);
\end{scope}
\node at (-0.1,-0.1) {\tiny $P_0$};
\node at (1.15,-0.1) {\tiny $P_1$};
\node at (-0.15,1.1) {\tiny $P_2$};
\node at (0.4,1.3) {\tiny $N_{20}$};
\node at (-0.35,0.5) {\tiny $\mathcal{D}^{n-1}$};
\node at (0.5,-0.2) {\tiny $\mathcal{D}^{n-1}$};
\draw[line width=0.1mm] (0.25,1.2) -- (0.2,0.97);
\draw[line width=0.1mm] (0.5,1.15) -- (0.85,0.8) -- (0.85,0.55);
\node at (1,1.2) {\tiny $P$};
\draw[line width=0.1mm] (0.9,1.2) -- (0.383,0.924);
\draw[line width=0.1mm] (0.92,1.1) -- (0.707,0.707);
\draw[line width=0.1mm] (1.0,1.1) -- (0.924,0.383);
\node at (1.4,0.8) {\tiny $N_{10}$};
\draw[line width=0.1mm] (1.3,0.7) -- (0.98,0.2);
\draw[line width=0.1mm] (1.2,0.8) -- (0.75,1) -- (0.6,0.8);
\end{scope}
\end{tikzpicture}
\end{equation}

We see how $N_0'$ is split into two $\Dk{n-1}$
along a $\Dk{n-2}$, which we denote by $P_0$.
Note that these are schematic diagrams,
$N_{10}$ and $N_{20}$ are not necessarily $(n-1)$-balls.

Since $M_0$ is an $n$-ball by hypothesis,
it can be thought of as a trivial cobordism
from $N_{20} \cup_P N_{10}$ to $\Dk{n-1}$
(technically, this would be a cobordism that is a
mix between cornered cobordism and relative cobordism;
the corner is $P_1$, and the vertical boundary is
the $\Dk{n-1}$ between $P_0$ and $P_2$;
if we double $M_0$ along the vertical boundary,
then we get a trivial cornered cobordism).
The rightmost diagram above shows $M_0$
for Yetter's case, which is when $Q$
is a torus with one boundary component.

$Q \#^{(1)} Q$ and $Q \#^{(2)} Q$
are schematically represented as follows:
\footnote{
Here the superscript in $\#^{(i)}$
is used to indicate the Y-product from which the $\#$ operation
is for.}

\begin{equation}
\label{e:diagram-bialg-QQ}
Q \#^{(1)} Q =
\begin{tikzpicture}
\draw (-0.7,0) -- (1.9,0);
\draw (0,0.7) -- (0,-0.7);
\draw (1.2,0.7) -- (1.2,-0.7);
\draw (0,0.7)
	.. controls +(180:0.6cm) and +(90:0.6cm) .. (-0.7,0);
\draw (1.2,0.7)
	.. controls +(0:0.6cm) and +(90:0.6cm) .. (1.9,0);
\draw (0,-0.7)
	.. controls +(180:0.6cm) and +(-90:0.6cm) .. (-0.7,0);
\draw (1.2,-0.7)
	.. controls +(0:0.6cm) and +(-90:0.6cm) .. (1.9,0);
\draw (0,0.7)
	.. controls +(0:0.6cm) and +(110:0.1cm) ..
	(0.55,0.5) .. controls +(-70:0.05cm) and +(-110:0.05cm) ..
	(0.65,0.5) .. controls +(70:0.1cm) and +(180:0.6cm) ..
	(1.2,0.7);
\draw (0,-0.7)
	.. controls +(0:0.6cm) and +(-110:0.1cm) ..
	(0.55,-0.5) .. controls +(70:0.05cm) and +(110:0.05cm) ..
	(0.65,-0.5) .. controls +(-70:0.1cm) and +(180:0.5cm) ..
	(1.2,-0.7);
\node at (0.6,0.28) {\tiny $M_0 \cup \overline{M_0}$};
\node at (0.6,-0.28) {\tiny $\overline{M_0} \cup M_0$};
\end{tikzpicture}
\;\;\;
;
\;\;\;
Q \#^{(2)} Q =
\begin{tikzpicture}
\begin{scope}[shift={(0,-0.6)}]
\draw (-0.7,0) -- (0.7,0);
\draw (-0.7,1.2) -- (0.7,1.2);
\draw (0,1.9) -- (0,-0.7);
\draw (0,1.9)
	.. controls +(180:0.6cm) and +(90:0.6cm) .. (-0.7,1.2);
\draw (0,1.9)
	.. controls +(0:0.6cm) and +(90:0.6cm) .. (0.7,1.2);
\draw (0,-0.7)
	.. controls +(180:0.6cm) and +(-90:0.6cm) .. (-0.7,0);
\draw (0,-0.7)
	.. controls +(0:0.6cm) and +(-90:0.6cm) .. (0.7,0);
\draw (-0.7,1.2)
	.. controls +(-90:0.6cm) and +(160:0.1cm) ..
	(-0.5,0.65) .. controls +(-20:0.05cm) and +(20:0.05cm) ..
	(-0.5,0.55) .. controls +(-160:0.1cm) and +(90:0.6cm) ..
	(-0.7,0);
\draw (0.7,1.2)
	.. controls +(-90:0.6cm) and +(20:0.1cm) ..
	(0.5,0.65) .. controls +(-160:0.05cm) and +(160:0.05cm) ..
	(0.5,0.55) .. controls +(-20:0.1cm) and +(90:0.6cm) ..
	(0.7,0);
\end{scope}
\end{tikzpicture}
\end{equation}

In each middle piece $M_0 \cup \ov{M_0}$,
the union is taken over one of the $\Dk{n-1}$'s
in $N_0'$.

Finally $(Q \#^{(1)} Q) \#^{(2)} (Q \#^{(1)} Q)
= (Q \#^{(2)} Q) \#^{(1)} (Q \#^{(2)} Q)$
is schematically represented as follows:

\begin{equation}
\label{e:diagram-bialg-QQQ}
\begin{tikzpicture}
\draw (-0.7,0) -- (1.9,0);
\draw (-0.7,1.2) -- (1.9,1.2);
\draw (0,1.9) -- (0,-0.7);
\draw (1.2,1.9) -- (1.2,-0.7);
\draw (0,-0.7)
	.. controls +(180:0.6cm) and +(-90:0.6cm) .. (-0.7,0);
\draw (1.2,-0.7)
	.. controls +(0:0.6cm) and +(-90:0.6cm) .. (1.9,0);
\draw (0,1.9)
	.. controls +(180:0.6cm) and +(90:0.6cm) .. (-0.7,1.2);
\draw (1.2,1.9)
	.. controls +(0:0.5cm) and +(90:0.5cm) .. (1.9,1.2);
\begin{scope}[shift={(0,1.2)}]
\draw (0,0.7)
	.. controls +(0:0.6cm) and +(110:0.1cm) ..
	(0.55,0.5) .. controls +(-70:0.05cm) and +(-110:0.05cm) ..
	(0.65,0.5) .. controls +(70:0.1cm) and +(180:0.6cm) ..
	(1.2,0.7);
\end{scope}
\draw (0,-0.7)
	.. controls +(0:0.6cm) and +(-110:0.1cm) ..
	(0.55,-0.5) .. controls +(70:0.05cm) and +(110:0.05cm) ..
	(0.65,-0.5) .. controls +(-70:0.1cm) and +(180:0.5cm) ..
	(1.2,-0.7);
\draw (-0.7,1.2)
	.. controls +(-90:0.6cm) and +(160:0.1cm) ..
	(-0.5,0.65) .. controls +(-20:0.05cm) and +(20:0.05cm) ..
	(-0.5,0.55) .. controls +(-160:0.1cm) and +(90:0.6cm) ..
	(-0.7,0);
\begin{scope}[shift={(1.2,0)}]
\draw (0.7,1.2)
	.. controls +(-90:0.6cm) and +(20:0.1cm) ..
	(0.5,0.65) .. controls +(-160:0.05cm) and +(160:0.05cm) ..
	(0.5,0.55) .. controls +(-20:0.1cm) and +(90:0.6cm) ..
	(0.7,0);
\end{scope}
\end{tikzpicture}
\end{equation}

We will discuss the square in the middle later.

Consider the following diagrams,
which represent $\coYYbar_Q^{(2)} \circ \YYbar_Q^{(1)}$
the same way that a fat graph represents
a relative cobordism as described in
\secref{s:basic-construction},
that is,
the diagram forms a ``spine'' onto which
copies of $M_0 \times I$ are glued
(we show this for parts of $\YYbar_Q^{(1)}$
in the right diagram below).
Note that the fact that the piece labeled (*)
is also $M_0 \times I$
relies on the fact that $N_1'$ is simply a disk.

\begin{equation}
\label{e:prod-then-coprod}
\Xi_1 :=
\coYYbar_Q^{(2)}
\circ
\YYbar_Q^{(1)}
=
\begin{tikzpicture}
\begin{scope}[shift={(0,0.3)}]
\draw (0,0.3) -- (0.9,0.3);
\draw (0,0.3) to[out=-90,in=90] (0.3,-0.3);
\draw (0.9,0.3) to[out=-90,in=90] (0.6,-0.3);
\draw (0.3,-0.3) -- (0.3,-0.6);
\draw (0.6,-0.3) -- (0.6,-0.6);
\draw (0.3,-0.3) -- (0.6,-0.3);
\draw (0.3,-0.6) to[out=-135,in=90] (0.2,-1.0);
\draw (0.3,-0.6) to[out=-45,in=90] (0.4,-0.8);
\draw (0.6,-0.6) to[out=-135,in=90] (0.5,-1.0);
\draw (0.6,-0.6) to[out=-45,in=90] (0.7,-0.8);
\draw (0.2,-1.0) -- (0.5,-1.0);
\draw (0.4,-0.8) -- (0.7,-0.8);
\draw (0.2,0.2) -- (0.4,0.4);
\draw (0.5,0.2) -- (0.7,0.4);
\draw (0.2,0.2) to[out=-90,in=135] (0.35,-0.1);
\draw (0.5,0.2) to[out=-90,in=45] (0.35,-0.1);
\draw (0.4,0.4) to[out=-90,in=135] (0.55,0.1);
\draw (0.7,0.4) to[out=-90,in=45] (0.55,0.1);
\draw (0.35,-0.1) -- (0.35,-0.4);
\draw (0.55,0.1) -- (0.55,-0.2);
\draw (0.35,-0.4) -- (0.55,-0.2);
\draw (0.55,-0.2) to[out=-90,in=90] (0.75,-0.6);
\draw (0.35,-0.4) ..controls +(-90:0.2cm) and +(90:0.5cm) ..
	(0.15,-1.2);
\draw (0.15,-1.2) -- (0.75,-0.6);
\begin{scope}[line width=0.1mm]
\draw (0.3,0.3) to[out=-90,in=135] (0.45,0);
\draw (0.6,0.3) to[out=-90,in=45] (0.45,0);
\draw (0.45,0) -- (0.45,-0.6);
\draw (0.45,-0.6) to[out=-135,in=90] (0.35,-1.0);
\draw (0.45,-0.6) to[out=-45,in=90] (0.55,-0.8);
\end{scope}
\end{scope}
\end{tikzpicture}
=
\begin{tikzpicture}
\begin{scope}[shift={(0,0.3)}]
\draw (0,0.3) -- (0.9,0.3);
\draw (0,0.3) to[out=-90,in=90] (0.3,-0.3);
\draw (0.9,0.3) to[out=-90,in=90] (0.6,-0.3);
\draw (0.3,-0.3) -- (0.3,-0.6);
\draw (0.6,-0.3) -- (0.6,-0.6);
\draw (0.3,-0.3) -- (0.6,-0.3);
\draw (0.3,-0.6) to[out=-135,in=90] (0.2,-1.0);
\draw (0.3,-0.6) to[out=-45,in=90] (0.4,-0.8);
\draw (0.6,-0.6) to[out=-135,in=90] (0.5,-1.0);
\draw (0.6,-0.6) to[out=-45,in=90] (0.7,-0.8);
\draw (0.2,-1.0) -- (0.5,-1.0);
\draw (0.4,-0.8) -- (0.7,-0.8);
\begin{scope}[line width=0.1mm]
\draw (0.3,0.3) to[out=-90,in=135] (0.45,0);
\draw (0.6,0.3) to[out=-90,in=45] (0.45,0);
\draw (0.45,0) -- (0.45,-0.6);
\draw (0.45,-0.6) to[out=-135,in=90] (0.35,-1.0);
\draw (0.45,-0.6) to[out=-45,in=90] (0.55,-0.8);
\end{scope}
\end{scope}
\end{tikzpicture}
\cup
\begin{tikzpicture}
\begin{scope}[shift={(0,0.3)}]
\draw (0.2,0.2) -- (0.4,0.4);
\draw (0.5,0.2) -- (0.7,0.4);
\draw (0.2,0.2) to[out=-90,in=135] (0.35,-0.1);
\draw (0.5,0.2) to[out=-90,in=45] (0.35,-0.1);
\draw (0.4,0.4) to[out=-90,in=135] (0.55,0.1);
\draw (0.7,0.4) to[out=-90,in=45] (0.55,0.1);
\draw (0.35,-0.1) -- (0.35,-0.4);
\draw (0.55,0.1) -- (0.55,-0.2);
\draw (0.35,-0.4) -- (0.55,-0.2);
\draw (0.55,-0.2) to[out=-90,in=90] (0.75,-0.6);
\draw (0.35,-0.4) ..controls +(-90:0.2cm) and +(90:0.5cm) ..
	(0.15,-1.2);
\draw (0.15,-1.2) -- (0.75,-0.6);
\begin{scope}[line width=0.1mm]
\draw (0.3,0.3) to[out=-90,in=135] (0.45,0);
\draw (0.6,0.3) to[out=-90,in=45] (0.45,0);
\draw (0.45,0) -- (0.45,-0.6);
\draw (0.45,-0.6) to[out=-135,in=90] (0.35,-1.0);
\draw (0.45,-0.6) to[out=-45,in=90] (0.55,-0.8);
\end{scope}
\end{scope}
\end{tikzpicture}
=
\begin{tikzpicture}
\draw (0,0.6) -- (0.9,0.6);
\draw (0.2,0.5) -- (0.4,0.7);
\draw (0.5,0.5) -- (0.7,0.7);
\draw (0.2,0) -- (0.7,0);
\draw (0.35,-0.1) -- (0.55,0.1);
\draw (0.1,-0.7) -- (0.6,-0.7);
\draw (0.3,-0.5) -- (0.8,-0.5);
\draw (0.25,-0.8) -- (0.65,-0.4);
\draw (0.2,0.5) to[out=-90,in=120] (0.35,-0.1);
\draw (0.5,0.5) to[out=-90,in=60] (0.35,-0.1);
\draw (0.4,0.7) to[out=-90,in=120] (0.55,0.1);
\draw (0.7,0.7) to[out=-90,in=60] (0.55,0.1);
\draw (0,0.6) to[out=-90,in=120] (0.2,0);
\draw (0.9,0.6) to[out=-90,in=60] (0.7,0);
\draw (0.35,-0.1) to[out=-100,in=90] (0.25,-0.8);
\draw (0.55,0.1) to[out=-70,in=90] (0.65,-0.4);
\draw (0.2,0) to[out=-100,in=90] (0.1,-0.7);
\draw (0.2,0) to[out=-70,in=90] (0.3,-0.5);
\draw (0.7,0) to[out=-100,in=90] (0.6,-0.7);
\draw (0.7,0) to[out=-70,in=90] (0.8,-0.5);
\begin{scope}[line width=0.1mm]
\draw (0.3,0.6) to[out=-90,in=120] (0.45,0);
\draw (0.6,0.6) to[out=-90,in=60] (0.45,0);
\draw (0.45,0) to[out=-100,in=90] (0.35,-0.7);
\draw (0.45,0) to[out=-70,in=90] (0.55,-0.5);
\end{scope}
\end{tikzpicture}
\;\;
;
\hspace{2cm}
\begin{tikzpicture}
\draw (0,0.3) -- (0.9,0.3);
\draw (0,0.3) to[out=-90,in=90] (0.3,-0.3);
\draw (0.9,0.3) to[out=-90,in=90] (0.6,-0.3);
\draw (0.3,-0.3) -- (0.6,-0.3);
\draw (0.2,0.2) -- (0.4,0.4);
\draw (0.5,0.2) -- (0.7,0.4);
\draw (0.2,0.2) to[out=-90,in=135] (0.35,-0.1);
\draw (0.5,0.2) to[out=-90,in=45] (0.35,-0.1);
\draw (0.4,0.4) to[out=-90,in=135] (0.55,0.1);
\draw (0.7,0.4) to[out=-90,in=45] (0.55,0.1);
\draw (0.35,-0.1) -- (0.35,-0.4);
\draw (0.55,0.1) -- (0.55,-0.2);
\draw (0.35,-0.4) -- (0.55,-0.2);
\begin{scope}[line width=0.1mm]
\draw (0.3,0.3) to[out=-90,in=135] (0.45,0);
\draw (0.6,0.3) to[out=-90,in=45] (0.45,0);
\draw (0.45,0) -- (0.45,-0.3);
\end{scope}
\draw (-0.4,0.1) -- (-0.1,0.1);
\draw (-0.2,0) -- (-0.1,0.1);
\draw (-0.4,0.1) .. controls +(-135:0.1cm) and +(180:0.2cm) ..
	(-0.2,0);
\draw (-0.25,-0.5) -- (0.05,-0.5);
\draw (-0.05,-0.6) -- (0.05,-0.5);
\draw (-0.25,-0.5) .. controls +(-135:0.1cm) and +(180:0.2cm) ..
	(-0.05,-0.6);
\draw (-0.2,0) .. controls +(-90:0.2cm) and +(90:0.2cm) ..
	(-0.05,-0.6);
\draw (-0.1,0.1) .. controls +(-90:0.2cm) and +(90:0.2cm) ..
	(0.05,-0.5);
\draw (-0.43,0.07) .. controls +(-90:0.2cm) and +(90:0.2cm) ..
	(-0.28,-0.55);
\draw[->] (0,-0.2) -- (0.2,-0.1);
\draw (1,0) -- (1.4,0);
\draw (0.9,-0.1) -- (1.1,0.1);
\draw (1.3,-0.1) -- (1.5,0.1);
\draw (0.9,-0.1) to[out=0,in=-135] (1.15,-0.05);
\draw (1.15,-0.05) to[out=-135,in=180] (1.3,-0.1);
\draw (1.1,0.1) to[out=0,in=45] (1.25,0.05);
\draw (1.25,0.05) to[out=45,in=180] (1.5,0.1);
\draw (0.9,-0.1) .. controls +(-90:0.3cm) and +(-90:0.3cm) ..
	(1.3,-0.1);
\draw (1.5,0.1) .. controls +(-90:0.2cm) and +(43:0.2cm) ..
	(1.22,-0.29);
\draw[line width=0.1mm] (1.4,0)
	.. controls +(-90:0.15cm) and +(45:0.15cm) ..
	(1.22,-0.29);
\draw[->] (1.2,0.2)
	.. controls +(90:0.4cm) and +(60:0.3cm) .. (0.55,0.4);
\begin{scope}[scale={0.2}]
\node at (9,0) {\tiny (*)};
\end{scope}
\end{tikzpicture}
\end{equation}

We denote $\Xi_1 := \coYYbar_Q^{(2)} \circ \YYbar_Q^{(1)}$
for brevity.
The first diagram is the union of
the two subsequent diagram along the thinner lines.
The last diagram (before the semicolon)
is easily seen to represent the same cornered cobordism
as $\Xi_1$.
The top half represents $\YYbar_Q^{(1)}$;
the ``
\begin{tikzpicture}
\draw (0,0) -- (0.3,0);
\draw (0.08,0.1) -- (0.08,-0.1);
\draw (0.22,0.1) -- (0.22,-0.1);
\end{tikzpicture}
''
at the top represents the incoming boundary $Q \#^{(1)} Q$
of $\YYbar_Q^{(1)}$
as in \eqnref{e:diagram-bialg-QQ},
and the ``
\begin{tikzpicture}
\draw (0,0) -- (0.2,0);
\draw (0.1,0.1) -- (0.1,-0.1);
\end{tikzpicture}
''
in the middle represents $Q$
as in \eqnref{e:diagram-bialg-Q}.

Now the other side of the bialgebra relation is
$(\YYbar_Q^{(1)} \#^{(2)} \YYbar_Q^{(1)})
\circ
(\coYYbar_Q^{(2)} \#^{(1)} \coYYbar_Q^{(2)})$,
which we denote by
$\Xi_2$ for brevity;
we depict it by the same graphical representation as follows:

\begin{equation}
\label{e:coprod-then-prod}
\Xi_2 :=
(\YYbar_Q^{(1)} \#^{(2)} \YYbar_Q^{(1)})
\circ
(\coYYbar_Q^{(2)} \#^{(1)} \coYYbar_Q^{(2)})
=
\begin{tikzpicture}
\begin{scope}[shift={(0,0.3)}]
\draw (0,0.3) -- (0.9,0.3);
\draw (0,0) -- (0,0.3);
\draw (0.9,0) -- (0.9,0.3);
\draw (0,0) to[out=-45,in=90] (0.1,-0.2);
\draw (0.9,0) to[out=-45,in=90] (1.0,-0.2);
\draw (0.1,-0.2) to[out=-90,in=90] (0.4,-0.8);
\draw (1.0,-0.2) to[out=-90,in=90] (0.7,-0.8);
\draw (0.4,-0.8) -- (0.7,-0.8);
\draw (0,0) to[out=-135,in=90] (-0.1,-0.4);
\draw (0.9,0) to[out=-135,in=90] (0.8,-0.4);
\draw (-0.1,-0.4) to[out=-90,in=90] (0.2,-1);
\draw (0.8,-0.4) to[out=-90,in=90] (0.5,-1);
\draw (0.2,-1) -- (0.5,-1);
\draw (0.2,0.2) -- (0.4,0.4);
\draw (0.2,0.2) 
  .. controls +(-90:0.2cm) and +(90:0.5cm) .. (0,-0.6);
\draw (0.4,0.4) to[out=-90,in=90] (0.6,0);
\draw (0.5,0.2) -- (0.7,0.4);
\draw (0.5,0.2) ..controls +(-90:0.2cm) and +(90:0.5cm) .. (0.3,-0.6);
\draw (0.7,0.4) to[out=-90,in=90] (0.9,0);
\draw (0,-0.6) to[out=-90,in=135] (0.15,-0.9);
\draw (0.6,0) to[out=-90,in=135] (0.75,-0.3);
\draw (0.3,-0.6) to[out=-90,in=45] (0.15,-0.9);
\draw (0.9,0) to[out=-90,in=45] (0.75,-0.3);
\draw (0.15,-0.9) -- (0.15,-1.2);
\draw (0.75,-0.3) -- (0.75,-0.6);
\draw (0.15,-1.2) -- (0.75,-0.6);
\begin{scope}[line width=0.1mm]
\draw (0.3,0) -- (0.3,0.3);
\draw (0.6,0) -- (0.6,0.3);
\draw (0.3,0) to[out=-45,in=90] (0.4,-0.2);
\draw (0.6,0) to[out=-45,in=90] (0.7,-0.2);
\draw (0.4,-0.2) to[out=-90,in=135] (0.55,-0.5);
\draw (0.7,-0.2) to[out=-90,in=45] (0.55,-0.5);
\draw (0.3,0) to[out=-135,in=90] (0.2,-0.4);
\draw (0.6,0) to[out=-135,in=90] (0.5,-0.4);
\draw (0.55,-0.5) -- (0.55,-0.8);
\draw (0.35,-0.7) -- (0.35,-1);
\draw (0.2,-0.4) to[out=-90,in=135] (0.35,-0.7);
\draw (0.5,-0.4) to[out=-90,in=45] (0.35,-0.7);
\draw (0.3,0) to[out=-135,in=90] (0.2,-0.4);
\draw (0.6,0) to[out=-135,in=90] (0.5,-0.4);
\draw (0.5,-0.4) to[out=-90,in=45] (0.35,-0.7);
\draw (0.7,-0.2) to[out=-90,in=45] (0.55,-0.5);
\draw (0.35,-0.7) -- (0.35,-1);
\draw (0.55,-0.5) -- (0.55,-0.8);
\end{scope}
\end{scope}
\end{tikzpicture}
=
\begin{tikzpicture}
\begin{scope}[shift={(0,0.3)}]
\draw (0,0.3) -- (0.9,0.3);
\draw (0,0) -- (0,0.3);
\draw (0.3,0) -- (0.3,0.3);
\draw (0.6,0) -- (0.6,0.3);
\draw (0.9,0) -- (0.9,0.3);
\draw[line width=0.1mm] (0,0) to[out=-45,in=90] (0.1,-0.2);
\draw[line width=0.1mm] (0.3,0) to[out=-45,in=90] (0.4,-0.2);
\draw[line width=0.1mm] (0.6,0) to[out=-45,in=90] (0.7,-0.2);
\draw[line width=0.1mm] (0.9,0) to[out=-45,in=90] (1.0,-0.2);
\draw[line width=0.1mm] (0.4,-0.2) to[out=-90,in=135] (0.55,-0.5);
\draw[line width=0.1mm] (0.7,-0.2) to[out=-90,in=45] (0.55,-0.5);
\draw[line width=0.1mm] (0.1,-0.2) to[out=-90,in=90] (0.4,-0.8);
\draw[line width=0.1mm] (1.0,-0.2) to[out=-90,in=90] (0.7,-0.8);
\draw[line width=0.1mm] (0.4,-0.8) -- (0.7,-0.8);
\draw[line width=0.1mm] (0.55,-0.5) -- (0.55,-0.8);
\draw (0,0) to[out=-135,in=90] (-0.1,-0.4);
\draw (0.3,0) to[out=-135,in=90] (0.2,-0.4);
\draw (0.6,0) to[out=-135,in=90] (0.5,-0.4);
\draw (0.9,0) to[out=-135,in=90] (0.8,-0.4);
\draw (0.2,-0.4) to[out=-90,in=135] (0.35,-0.7);
\draw (0.5,-0.4) to[out=-90,in=45] (0.35,-0.7);
\draw (-0.1,-0.4) to[out=-90,in=90] (0.2,-1);
\draw (0.8,-0.4) to[out=-90,in=90] (0.5,-1);
\draw (0.2,-1) -- (0.5,-1);
\draw (0.35,-0.7) -- (0.35,-1);
\end{scope}
\end{tikzpicture}
\cup
\begin{tikzpicture}
\begin{scope}[shift={(0,0.3)}]
\draw[line width=0.1mm] (0.2,0.2) -- (0.4,0.4);
\draw[line width=0.1mm] (0.2,0.2) 
  .. controls +(-90:0.2cm) and +(90:0.5cm) .. (0,-0.6);
\draw[line width=0.1mm] (0.4,0.4) to[out=-90,in=90] (0.6,0);
\draw[line width=0.1mm] (0.3,0) to[out=-135,in=90] (0.2,-0.4);
\draw[line width=0.1mm] (0.3,0) to[out=-45,in=90] (0.4,-0.2);
\draw[line width=0.1mm] (0.3,0) -- (0.3,0.3);
\draw (0.5,0.2) -- (0.7,0.4);
\draw (0.5,0.2) ..controls +(-90:0.2cm) and +(90:0.5cm) .. (0.3,-0.6);
\draw (0.7,0.4) to[out=-90,in=90] (0.9,0);
\draw (0.6,0) to[out=-135,in=90] (0.5,-0.4);
\draw (0.6,0) to[out=-45,in=90] (0.7,-0.2);
\draw (0.6,0) -- (0.6,0.3);
\draw[line width=0.1mm] (0,-0.6) to[out=-90,in=135] (0.15,-0.9);
\draw[line width=0.1mm] (0.2,-0.4) to[out=-90,in=135] (0.35,-0.7);
\draw[line width=0.1mm] (0.4,-0.2) to[out=-90,in=135] (0.55,-0.5);
\draw[line width=0.1mm] (0.6,0) to[out=-90,in=135] (0.75,-0.3);
\draw (0.3,-0.6) to[out=-90,in=45] (0.15,-0.9);
\draw (0.5,-0.4) to[out=-90,in=45] (0.35,-0.7);
\draw (0.7,-0.2) to[out=-90,in=45] (0.55,-0.5);
\draw (0.9,0) to[out=-90,in=45] (0.75,-0.3);
\draw (0.15,-0.9) -- (0.15,-1.2);
\draw (0.35,-0.7) -- (0.35,-1);
\draw (0.55,-0.5) -- (0.55,-0.8);
\draw (0.75,-0.3) -- (0.75,-0.6);
\draw (0.15,-1.2) -- (0.75,-0.6);
\end{scope}
\end{tikzpicture}
\end{equation}

We want to show that $\Xi_1$ and $\Xi_2$
are in fact the same cornered cobordism.
In the diagram below,
we observe that each piece of $M_0 \times I$ in
$\Xi_1$ has a corresponding part in $\Xi_2$:

\begin{equation}
\begin{tikzpicture}
\draw (0,0.6) -- (0.9,0.6);
\draw (0.2,0.5) -- (0.4,0.7);
\draw (0.5,0.5) -- (0.7,0.7);
\draw (0.2,0) -- (0.7,0);
\draw (0.35,-0.1) -- (0.55,0.1);
\draw (0.1,-0.7) -- (0.6,-0.7);
\draw (0.3,-0.5) -- (0.8,-0.5);
\draw (0.25,-0.8) -- (0.65,-0.4);
\draw (0.2,0.5) to[out=-90,in=120] (0.35,-0.1);
\draw (0.5,0.5) to[out=-90,in=60] (0.35,-0.1);
\draw (0.4,0.7) to[out=-90,in=120] (0.55,0.1);
\draw (0.7,0.7) to[out=-90,in=60] (0.55,0.1);
\draw (0,0.6) to[out=-90,in=120] (0.2,0);
\draw (0.9,0.6) to[out=-90,in=60] (0.7,0);
\draw (0.35,-0.1) to[out=-100,in=90] (0.25,-0.8);
\draw (0.55,0.1) to[out=-70,in=90] (0.65,-0.4);
\draw (0.2,0) to[out=-100,in=90] (0.1,-0.7);
\draw (0.2,0) to[out=-70,in=90] (0.3,-0.5);
\draw (0.7,0) to[out=-100,in=90] (0.6,-0.7);
\draw (0.7,0) to[out=-70,in=90] (0.8,-0.5);
\begin{scope}[line width=0.1mm]
\draw (0.3,0.6) to[out=-90,in=120] (0.45,0);
\draw (0.6,0.6) to[out=-90,in=60] (0.45,0);
\draw (0.45,0) to[out=-100,in=90] (0.35,-0.7);
\draw (0.45,0) to[out=-70,in=90] (0.55,-0.5);
\end{scope}
\begin{scope}[shift={(1.5,0.3)}]
\draw (0,0.3) -- (0.9,0.3);
\draw (0,0) -- (0,0.3);
\draw (0.9,0) -- (0.9,0.3);
\draw (0,0) to[out=-45,in=90] (0.1,-0.2);
\draw (0.9,0) to[out=-45,in=90] (1.0,-0.2);
\draw (0.1,-0.2) to[out=-90,in=90] (0.4,-0.8);
\draw (1.0,-0.2) to[out=-90,in=90] (0.7,-0.8);
\draw (0.4,-0.8) -- (0.7,-0.8);
\draw (0,0) to[out=-135,in=90] (-0.1,-0.4);
\draw (0.9,0) to[out=-135,in=90] (0.8,-0.4);
\draw (-0.1,-0.4) to[out=-90,in=90] (0.2,-1);
\draw (0.8,-0.4) to[out=-90,in=90] (0.5,-1);
\draw (0.2,-1) -- (0.5,-1);
\draw (0.2,0.2) -- (0.4,0.4);
\draw (0.2,0.2) 
  .. controls +(-90:0.2cm) and +(90:0.5cm) .. (0,-0.6);
\draw (0.4,0.4) to[out=-90,in=90] (0.6,0);
\draw (0.5,0.2) -- (0.7,0.4);
\draw (0.5,0.2) ..controls +(-90:0.2cm) and +(90:0.5cm) .. (0.3,-0.6);
\draw (0.7,0.4) to[out=-90,in=90] (0.9,0);
\draw (0,-0.6) to[out=-90,in=135] (0.15,-0.9);
\draw (0.6,0) to[out=-90,in=135] (0.75,-0.3);
\draw (0.3,-0.6) to[out=-90,in=45] (0.15,-0.9);
\draw (0.9,0) to[out=-90,in=45] (0.75,-0.3);
\draw (0.15,-0.9) -- (0.15,-1.2);
\draw (0.75,-0.3) -- (0.75,-0.6);
\draw (0.15,-1.2) -- (0.75,-0.6);
\begin{scope}[line width=0.1mm]
\draw (0.3,0) -- (0.3,0.3);
\draw (0.6,0) -- (0.6,0.3);
\draw (0.3,0) to[out=-45,in=90] (0.4,-0.2);
\draw (0.6,0) to[out=-45,in=90] (0.7,-0.2);
\draw (0.4,-0.2) to[out=-90,in=135] (0.55,-0.5);
\draw (0.7,-0.2) to[out=-90,in=45] (0.55,-0.5);
\draw (0.3,0) to[out=-135,in=90] (0.2,-0.4);
\draw (0.6,0) to[out=-135,in=90] (0.5,-0.4);
\draw (0.55,-0.5) -- (0.55,-0.8);
\draw (0.35,-0.7) -- (0.35,-1);
\draw (0.2,-0.4) to[out=-90,in=135] (0.35,-0.7);
\draw (0.5,-0.4) to[out=-90,in=45] (0.35,-0.7);
\draw (0.3,0) to[out=-135,in=90] (0.2,-0.4);
\draw (0.6,0) to[out=-135,in=90] (0.5,-0.4);
\draw (0.5,-0.4) to[out=-90,in=45] (0.35,-0.7);
\draw (0.7,-0.2) to[out=-90,in=45] (0.55,-0.5);
\draw (0.35,-0.7) -- (0.35,-1);
\draw (0.55,-0.5) -- (0.55,-0.8);
\end{scope}
\end{scope}
\draw (0.5,0.7) to[out=60,in=120] (2,0.7);
\draw (0.9,0.7) to[out=60,in=60] (2.4,0.7);
\draw (0.1,0.7) to[out=120,in=120] (1.6,0.7);
\draw (0.7,-0.6) to[out=-60,in=-90] (2.1,-0.6);
\draw (0.15,-0.8) to[out=-150,in=-150] (1.6,-0.8);
\draw (0.45,-0.8) to[out=-45,in=-45] (1.9,-0.8);
\draw (2.7,-0.4) -- (2.15,-0.1);
\node at (2.95,-0.4) {\tiny heart};
\end{tikzpicture}
\end{equation}

There is only one piece in $\Xi_2$
that is missing from $\Xi_1$,
which is the piece in the center;
we refer to this piece as the \emph{heart}.

The heart is schematically represented as follows
(the second diagram shows the heart when $Q$
is the torus with one boundary component):
\begin{equation}
\begin{tikzpicture}
\draw (-0.7,-0.3) -- (0.3,-0.3) -- (0.7,0.3) -- (-0.3,0.3)
	-- (-0.7,-0.3);
\draw (-0.7,-0.3) .. controls +(90:0.5cm) and +(-135:0.1cm) ..
	(-0.5,0.7) .. controls +(45:0.1cm) and +(90:0.5cm) ..
	(-0.3,0.3);
\draw (0.3,-0.3) .. controls +(90:0.5cm) and +(-135:0.1cm) ..
	(0.5,0.7) .. controls +(45:0.1cm) and +(90:0.5cm) ..
	(0.7,0.3);
\draw (-0.45,0.74) -- (0.55,0.74);
\draw (-0.7,-0.3) arc (180:360:0.5);
\draw (-0.3,0.3) arc (180:360:0.5);
\draw (0.2,-0.6) -- (0.6,0);
\draw[dotted] (-0.5,0.7) -- (0.5,0.7);
\draw[dotted] (-0.5,0.7) -- (-0.5,0);
\draw[dotted] (-0.5,0) arc (180:360:0.5);
\draw[dotted] (0.5,0) -- (0.5,0.7);
\node at (-0.8,-1.1) {\tiny $N_{20} \cup_{P_2} \overline{N_{20}}$};
\draw[line width=0.1mm] (-0.8,-1) -- (-0.5,-0.3);
\node at (-1.6,0.4) {\tiny $N_{10} \cup_{P_1} \overline{N_{10}}$};
\draw[line width=0.1mm] (-1.2,0.2) -- (-0.58,-0.1);
\end{tikzpicture}
\;\;\;
;
\;\;\;
\begin{tikzpicture}
\draw (0,0) circle (1cm);
\draw (0,0) ellipse (1cm and 0.4cm);
\draw (-0.2,-0.4) .. controls +(90:0.5cm) and +(180:0.2cm) ..
	(0,1) .. controls +(0:0.15cm) and +(90:0.4cm) ..
	(0.2,0.4);
\draw (-0.5,-0.35) .. controls +(90:0.45cm) and +(-170:0.2cm) ..
	(-0.3,0.95) .. controls +(10:0.15cm) and +(90:0.4cm) ..
	(-0.1,0.4);
\draw (-0.75,-0.3) .. controls +(90:0.4cm) and +(-165:0.2cm) ..
	(-0.55,0.83) .. controls +(15:0.15cm) and +(90:0.2cm) ..
	(-0.42,0.35);
\draw (-0.92,-0.15) .. controls +(90:0.2cm) and +(-135:0.1cm) ..
	(-0.8,0.6) .. controls +(45:0.1cm) and +(90:0.2cm) ..
	(-0.7,0.3);
\begin{scope}[line width=0.1mm]
\draw (-1,0.05) -- (0,0) -- (-0.2,-0.4);
\draw (-0.5,-0.35) -- +(63.4:0.405cm);
\draw (-0.75,-0.3) -- +(63.4:0.365cm);
\draw (-0.92,-0.15) -- +(63.4:0.21cm);
\end{scope}
\draw (-0.5,-0.35) arc (180:350:0.31);
\draw (-0.75,-0.3) arc (185:351:0.65);
\draw (-0.92,-0.15) .. controls +(-90:0.2cm) and +(135:0.1cm) ..
	(-0.8,-0.6);
\draw (0.85,-0.2) .. controls +(-90:0.4cm) and +(20:0.3cm) ..
	(0.3,-0.95);
\node at (-1.2,-0.8) {\tiny $N_{10}$};
\draw[line width=0.1mm] (-1.1,-0.7) -- (-0.6,-0.31);
\draw[line width=0.1mm] (-1.2,-0.65) -- (-0.95,-0.15);
\node at (-0.8,-1.2) {\tiny $N_{20}$};
\draw[line width=0.1mm] (-0.7,-1.1) -- (-0.3,-0.4);
\draw[line width=0.1mm] (-0.8,-1.05) -- (-0.82,-0.23);
\node at (-1.2,1.0) {\tiny $P \times I$};
\draw[line width=0.1mm] (-1.2,0.85) -- (-0.87,0.3);
\draw[line width=0.1mm] (-0.9,0.9) -- (-0.67,0.65);
\draw[line width=0.1mm] (-0.9,1.0) -- (-0.4,0.85);
\node at (-1.35,0) {\tiny $P_1$};
\draw[line width=0.1mm] (-1.2,0.05) -- (-1,0.05);
\node at (0,-1.4) {\tiny $P_2$};
\draw[line width=0.1mm] (0,-1.25) -- (-0.2,-0.4);
\end{tikzpicture}
\end{equation}

The dotted lines represent some fixed points
under the symmetry associated with $\coYYbar_Q^{(2)}$.
Concretely, the horizontal dotted line at the top
is a copy of $N_{20} \cup_{P_2} \ov{N_{20}}$,
the vertical dotted line on the left and right
are copies of $N_{10}$,
and the dotted arc at the bottom is a copy of
$P_1 \times I$.

From the discussion on $M_0$ as a trivial cobordism
after \eqnref{e:diagram-bialg-Q},
it follows easily that
the heart is a trivial cornered cobordism,
where the corner is the dotted line.

Thus, we may apply the following operation
which folds and flattens the heart,
and $\Xi_2$ will remain unchanged as a cornered cobordism:

\begin{equation}
\label{e:diagram-collapse-heart}
\begin{tikzpicture}
\draw (0,0.6) -- (0.6,0.6);
\draw (0.1,-0.6) -- (0.5,-0.4);
\draw (0,0.6) to[out=-135,in=90]
	(-0.2,0) to[out=-90,in=135]
	(0.1,-0.6);
\draw (0.6,0.6) to[out=-135,in=90]
	(0.4,0) to[out=-90,in=45]
	(0.1,-0.6);
\draw[line width=0.1mm] (0,0.6) to[out=-45,in=90] (0.2,0.2);
\draw[line width=0.1mm] (0.2,0.2) to[out=-90,in=135] (0.5,-0.4);
\draw (0.6,0.6) to[out=-45,in=90] (0.8,0.2);
\draw (0.8,0.2) to[out=-90,in=45] (0.5,-0.4);
\draw[dotted] (0,0.6) -- (0,0)
	to[out=-90,in=135] (0.3,-0.5);
\draw[dotted] (0.6,0.6) -- (0.6,0)
	to[out=-90,in=45] (0.3,-0.5);
\end{tikzpicture}
\;\;
\begin{tikzpicture}
\draw (0,0.6) -- (0.6,0.6);
\draw (0.2,-0.6) -- (0.3,-0.8);
\draw[line width=0.1mm] (0.3,-0.8) -- (0.4,-0.4);
\draw (0,0.6) to[out=-120,in=90]
	(-0.1,-0.1) to[out=-90,in=135]
	(0.2,-0.6);
\draw (0.6,0.6) to[out=-120,in=90]
	(0.5,0) to[out=-90,in=45]
	(0.2,-0.6);
\draw[line width=0.1mm] (0,0.6) to[out=-45,in=90]
	(0.1,0.1) to[out=-90,in=135]
	(0.4,-0.4);
\draw[line width=0.1mm] (0.6,0.6) to[out=-45,in=90]
	(0.7,0.1) to[out=-90,in=45]
	(0.4,-0.4);
\draw[dotted] (0,0.6) to[out=-135,in=90]
	(-0.2,0) to[out=-90,in=135]
	(0.3,-0.8);
\draw[dotted] (0.6,0.6) to[out=-45,in=90]
	(0.8,0) to[out=-90,in=45]
	(0.3,-0.8);
\end{tikzpicture}
\;\;
\begin{tikzpicture}
\draw (0,0.6) -- (0.6,0.6);
\draw (0,0.6) -- (0,0)
	to[out=-90,in=135] (0.3,-0.5);
\draw (0.6,0.6) -- (0.6,0)
	to[out=-90,in=45] (0.3,-0.5);
\draw (0.3,-0.5) -- (0.3,-0.8);
\draw[dotted] (0,0.6) to[out=-135,in=90]
	(-0.2,0) to[out=-90,in=135]
	(0.3,-0.8);
\draw[dotted] (0.6,0.6) to[out=-45,in=90]
	(0.8,0) to[out=-90,in=45]
	(0.3,-0.8);
\node at (1,-0.8) {\tiny fins};
\draw[line width=0.1mm] (0.8,-0.75) -- (0.5,-0.5);
\draw[line width=0.1mm] (0.8,-0.8) -- (0.2,-0.6);
\end{tikzpicture}
\end{equation}

(We may also think of the operation above
as occurring after we remove the heart from $\Xi_2$,
and the diagrams depict how to identify the new boundary faces.)

Observe that shrinking the ``fins''
(see \eqnref{e:diagram-collapse-heart}) after the collapse
does not affect $\Xi_2$,
for the same reason that $\wdtld{\ii}$ is a unit for
$\wdtld{\YY}$ (see \eqnref{e:Y-thick} and discussion on
(co)unit after that).
After a slight modification, it's clear that we have $\Xi_1$.

\begin{equation}
\label{e:diagram-after-folding}
\input{diagram-after-folding.tikz}
\end{equation}

\end{proof}


\bibliographystyle{plain}
\bibliography{references}

\begin{bibdiv}
\begin{biblist}

\bib{alexander1928topological}{article}{
      author={Alexander, James~W},
       title={Topological invariants of knots and links},
        date={1928},
     journal={Transactions of the American Mathematical Society},
      volume={30},
      number={2},
       pages={275\ndash 306},
}

\bib{barrett-etal-2007}{article}{
      author={Barrett, John~W.},
      author={Faria~Martins, João},
      author={García-Islas, J.~Manuel},
       title={Observables in the turaev-viro and crane-yetter models},
        date={2007},
     journal={Journal of Mathematical Physics},
      volume={48},
      number={9},
       pages={093508},
      eprint={https://doi.org/10.1063/1.2759440},
         url={https://doi.org/10.1063/1.2759440},
}

\bib{BakK}{book}{
      author={Bakalov, Bojko},
      author={Kirillov, Alexander~A},
       title={Lectures on tensor categories and modular functors},
   publisher={American Mathematical Soc.},
        date={2001},
      volume={21},
}

\bib{borodzik2016morse}{article}{
      author={Borodzik, Maciej},
      author={N{\'e}methi, Andr{\'a}s},
      author={Ranicki, Andrew},
       title={Morse theory for manifolds with boundary},
        date={2016},
     journal={Algebraic \& Geometric Topology},
      volume={16},
      number={2},
       pages={971\ndash 1023},
}

\bib{conway1970enumeration}{inproceedings}{
      author={Conway, John~H},
       title={An enumeration of knots and links, and some of their algebraic
  properties},
organization={Elsevier},
        date={1970},
   booktitle={Computational problems in abstract algebra},
       pages={329\ndash 358},
}

\bib{ENO2005-fusion}{article}{
      author={Etingof, Pavel},
      author={Nikshych, Dmitri},
      author={Ostrik, Viktor},
       title={On fusion categories},
        date={2005},
     journal={Annals of Mathematics},
       pages={581\ndash 642},
}

\bib{freyd1985new}{article}{
      author={Freyd, Peter},
      author={Yetter, David},
      author={Hoste, Jim},
      author={Lickorish, WB~Raymond},
      author={Millett, Kenneth},
      author={Ocneanu, Adrian},
       title={A new polynomial invariant of knots and links},
        date={1985},
     journal={Bulletin (new series) of the American mathematical society},
      volume={12},
      number={2},
       pages={239\ndash 246},
}

\bib{gunningham2019finiteness}{article}{
      author={Gunningham, Sam},
      author={Jordan, David},
      author={Safronov, Pavel},
       title={The finiteness conjecture for skein modules},
        date={2019},
     journal={arXiv preprint arXiv:1908.05233},
}

\bib{jones1985polynomial}{article}{
      author={Jones, Vaughan~FR},
       title={A polynomial invariant for knots via von neumann algebras},
        date={1985},
     journal={Bull. Amer. Math. Soc.(NS)},
      volume={12},
      number={1},
       pages={103\ndash 111},
}

\bib{kauffman1987state}{article}{
      author={Kauffman, Louis~H},
       title={State models and the jones polynomial},
        date={1987},
     journal={Topology},
      volume={26},
      number={3},
       pages={395\ndash 407},
}

\bib{kirillov-stringnet}{article}{
      author={Kirillov~Jr, Alexander},
       title={String-net model of turaev-viro invariants},
        date={2011},
     journal={arXiv preprint arXiv:1106.6033},
}

\bib{laudenbach2011morse}{article}{
      author={Laudenbach, Fran{\c{c}}ois},
       title={A morse complex on manifolds with boundary},
        date={2011},
     journal={Geometriae Dedicata},
      volume={153},
      number={1},
       pages={47\ndash 57},
}

\bib{mathoverflowVerlinde}{misc}{
         how={MathOverflow},
         url={https://mathoverflow.net/q/151221},
        note={URL:https://mathoverflow.net/q/151221 (version: 2013-12-08)},
}

\bib{milnor2016morse}{incollection}{
      author={Milnor, John},
       title={Morse theory.(am-51), volume 51},
        date={2016},
   booktitle={Morse theory.(am-51), volume 51},
   publisher={Princeton university press},
}

\bib{Moore1989}{article}{
      author={Moore, Gregory},
      author={Seiberg, Nathan},
       title={Classical and quantum conformal field theory},
        date={1989-06},
     journal={Communications in Mathematical Physics},
      volume={123},
      number={2},
       pages={177\ndash 254},
}

\bib{muger2003structure}{article}{
      author={M{\"u}ger, Michael},
       title={On the structure of modular categories},
        date={2003},
     journal={Proceedings of the London Mathematical Society},
      volume={87},
      number={2},
       pages={291\ndash 308},
}

\bib{przytycki1988invariants}{article}{
      author={PRZYTYCKI, JH},
      author={Traczyk, Pawel},
       title={Invariants of links of conway type},
        date={1988},
     journal={Kobe J. Math.},
      volume={4},
       pages={115\ndash 139},
}

\bib{tham-reduced}{article}{
      author={Tham, Ying~Hong},
       title={Reduced tensor product on the drinfeld center},
        date={2020},
     journal={arXiv preprint arXiv:2004.09611},
}

\bib{tham-thesis}{thesis}{
      author={Tham, Ying~Hong},
       title={On the category of boundary values in the extended crane-yetter
  tqft},
        type={Ph.D. Thesis},
        date={2021},
}

\bib{VERLINDE1988360}{article}{
      author={Verlinde, Erik},
       title={Fusion rules and modular transformations in 2d conformal field
  theory},
        date={1988},
        ISSN={0550-3213},
     journal={Nuclear Physics B},
      volume={300},
       pages={360\ndash 376},
  url={https://www.sciencedirect.com/science/article/pii/0550321388906037},
}

\bib{yetter1997portrait}{article}{
      author={Yetter, David},
       title={Portrait of the handle as a hopf algebra},
        date={199701},
     journal={Physics, Geometry},
      volume={184},
}

\end{biblist}
\end{bibdiv}

\end{document}